\documentclass[11pt]{article}

\usepackage[T1]{fontenc}
\usepackage[utf8]{inputenc}

\usepackage{fourier} %pour le symbole \danger
\usepackage{enumerate} %pour les \begin{enumerate}[(i)] par exemple

\usepackage{amsfonts,amssymb,amsmath,amsthm}
\usepackage{bbm}
\usepackage{pgf}

\usepackage{hyperref}

\numberwithin{equation}{section}
\def\PP{\mathbb{P}}
\def\QQ{\mathbb{Q}}
\def\RR{\mathbb{R}}
\def\NN{\mathbb{N}}
\def\ZZ{\mathbb{Z}}

\def\EE{\mathbb{E}}
\def\11{\mathbbm{1}}

\def\E{\mathbb{E}}
\def\P{\mathbb{P}}
\def\R{\mathbb{R}}

\def\Q{\mathbb{Q}}
\def\N{\mathbb{N}}

\def\d{\partial}
\def\Z{\mathbb{Z}}

\def\cL{{\cal L}}

\newtheorem{thm}{Theorem}[section]
\newtheorem{lem}[thm]{Lemma}
\newtheorem{cor}[thm]{Corollary}

\newtheorem{prop}[thm]{Proposition}

\theoremstyle{remark}
\newtheorem{rem}{Remark}
\newtheorem{exa}{Example}

\newcommand{\vertiii}[1]{{\left\vert\kern-0.25ex\left\vert\kern-0.25ex\left\vert #1 
    \right\vert\kern-0.25ex\right\vert\kern-0.25ex\right\vert}}
    \def\restriction#1#2{\mathchoice
                  {\setbox1\hbox{${\displaystyle #1}_{\scriptstyle #2}$}
                  \restrictionaux{#1}{#2}}
                  {\setbox1\hbox{${\textstyle #1}_{\scriptstyle #2}$}
                  \restrictionaux{#1}{#2}}
                  {\setbox1\hbox{${\scriptstyle #1}_{\scriptscriptstyle #2}$}
                  \restrictionaux{#1}{#2}}
                  {\setbox1\hbox{${\scriptscriptstyle #1}_{\scriptscriptstyle #2}$}
                  \restrictionaux{#1}{#2}}}
    \def\restrictionaux#1#2{{#1\,\smash{\vrule height .8\ht1 depth .85\dp1}}_{\,#2}}

\begin{document}

\title{General criteria for the study of quasi-stationarity}

\author{Nicolas Champagnat$^{1}$, Denis Villemonais$^{1}$}

\footnotetext[1]{Universit\'e de Lorraine, CNRS, Inria, IECL, F-54000 Nancy, France \\
  E-mail: Nicolas.Champagnat@inria.fr, Denis.Villemonais@univ-lorraine.fr}
% \footnotetext[2]{CNRS, IECL, UMR 7502, Vand{\oe}uvre-l\`es-Nancy, F-54506, France}  
% \footnotetext[3]{Inria, TOSCA team, Villers-l\`es-Nancy, F-54600, France.}

\maketitle

\begin{abstract}
  For Markov processes with absorption, we provide general criteria ensuring the existence and the exponential non-uniform
  convergence in wei\-ghted total variation norm to a quasi-stationary distribution. We also characterize a subset
  of its domain of attraction by an integrability condition, prove the existence of a right eigenvector for the semigroup of the
  process and the existence and exponential ergodicity of the $Q$-process. These results are applied to one-dimensional and
  multi-dimensional diffusion pro\-ces\-ses, to pure jump continuous time processes, to reducible processes with several
  communication classes, to perturbed dynamical systems and discrete time processes evolving in discrete state spaces.
\end{abstract}


\medskip\noindent\textit{2010 Mathematics Subject Classification.} Primary: 37A25, 60B10, 60F99, 60J05, 60J10, 60J25, 60J27; 
% \begin{itemize}
% \item 60J05 Discrete-time Markov processes on general state spaces
% \item ? 60J10 Markov chains (discrete-time Markov processes on discrete state spaces)
% \item 60J25 Continuous-time Markov processes on general state spaces
% \item ? 60J27 Continuous-time Markov processes on discrete state spaces
% \item 37A25 Ergodicity, mixing, rates of mixing
% \item 60F99 None of the above, but in this Section <- Limit Theorems <- Probability theory and stochastic processes
% \item 60B10 Convergence of probability measures
% \end{itemize}
Secondary: 60J60, 60J75, 60J80, 93E03.
% \begin{itemize}
% \item 60J80 Branching processes (Galton-Watson, birth-and-death, etc.)
% \item 60J60 Diffusion processes
% \item 60J75 Jump processes
% % \item ? 60G10 Stationary processes
% % \item 92D25 Population dynamics (general)
% \item 93E03 Stochastic systems, general
% \end{itemize}

\tableofcontents

\section{Introduction}
\label{sec:intro}

Let $(X_t,t\in I)$ be a Markov process in $E\cup\{\d\}$ where $E$ is a measurable space and $\d\not\in E$, with set of time indices
$I$ which might be $\RR_+$ or $\frac{1}{k}\ZZ_+$ for some $k\in\NN:=\{1,2,\ldots\}$, where $\ZZ_+:=\{0,1,\ldots\}$. For all $x\in
E\cup\{\d\}$, we denote as usual by $\PP_x$ the law of $X$ given $X_0=x$ and for any probability measure $\mu$ on $E\cup\{\d\}$, we
define $\PP_\mu=\int_{E\cup\{\d\}}\PP_x\,\mu(dx)$. We also denote by $\E_x$ and $\E_\mu$ the associated expectations. We assume that
$\d$ is absorbing, which means that $X_t=\d$ for all $t\geq \tau_\d$, $\PP_x$-almost surely, where
\[
\tau_\d=\inf\{t\in I,\,X_t=\d\}.
\]
Our goal is to study the existence of \emph{quasi-limiting distributions} on $E$ for the process $X$, i.e.\  probability
measures $\nu$ such that
\[
\lim_{t\in I,\ t\rightarrow+\infty}\PP_\mu(X_t\in A\mid t<\tau_\partial)=\nu(A)
\]
for some probability measure $\mu$ on $E$ and for all $A\subset E$ measurable. Such a measure $\nu$ is a \emph{quasi-stationary
  distribution} for $X$, i.e.\ a probability measure such that $\PP_\nu(X_t\in\cdot\mid t<\tau_\partial)=\nu(\cdot)$ for all
$t\in I$. We refer the reader to~\cite{ColletMartinezEtAl2013,MeleardVillemonais2012,DoornPollett2013} for general introductions to
quasi-stationary distributions. In particular, it is well-known that there exists a constant $\lambda_0\geq 0$, called the decay
parameter of the quasi-stationary distribution $\nu$, such that $\PP_{\nu}(t<\tau_\d)=e^{-\lambda_0 t}$ for all $t\in I$
(for discrete time processes, i.e. $I=\mathbb Z_+$, the term refers to $\theta_0=e^{-\lambda_0}$).

More precisely, our first goal is to give general criteria involving Lyapunov-type functions $\varphi_1\geq 1$ and $\varphi_2\leq 1$
ensuring the existence of a quasi-stationary distribution $\nu_{QSD}$ such that
\begin{equation}
  \label{eq:conv-intro}
  \left\|\P_\mu(X_t\in\cdot\mid t<\tau_\d)-\nu_{QSD}\right\|_{TV(\varphi_1)}\leq C\alpha^t\frac{\mu(\varphi_1)}{\mu(\varphi_2)},\quad\forall t\in  
  I,
\end{equation}
for some constants $C\in (0,+\infty)$ and $\alpha\in (0,1)$ and for all probability measure $\mu$ on $E$ such that
$\mu(\varphi_1)<+\infty$ and $\mu(\varphi_2)>0$, where $\mu(\varphi):=\int_E \varphi(x)\,\mu(dx)$
and, for all probability measures $\mu_1$ and $\mu_2$,
\[
\|\mu_1-\mu_2\|_{TV(\varphi_1)}=\sup_{f:E\rightarrow\R\text{ measurable s.t.\ }|f|\leq\varphi_1}|\mu_1(f)-\mu_2(f)|.
\]
When $\varphi_1$ is bounded, we recover convergence for the usual total variation distance $\|\cdot\|_{TV(1)}$ since the norms
$\|\cdot\|_{TV(1)}$ and $\|\cdot\|_{TV(\varphi_1)}$ are equivalent.
% Here, the total variation distance is
% defined as
% \[
% \|\mu_1-\mu_2\|_{TV}=\sup_{f:E\rightarrow[-1,1]\text{ measurable}}|\mu_1(f)-\mu_2(f)|.
% \]
The measure $\nu_{QSD}$ in~\eqref{eq:conv-intro} is the only quasi-stationary distribution $\nu$ such that $\nu(\varphi_1)<+\infty$
and $\nu(\varphi_2)>0$. 

Our second goal is to show how our criteria can be applied to a wide range of Markov processes, including several classes of
processes for which even the existence of a quasi-stationary distribution was not known, such as diffusions in irregular domains or
perturbed dynamical systems in unbounded domains.

General criteria ensuring that the convergence in~\eqref{eq:conv-intro} holds uniformly with respect to the initial distribution
$\mu$ have been studied in~\cite{birkhoff-57,ChampagnatVillemonais2016b}. In this case, $\nu_{QSD}$ is the quasi-limiting
distribution of any initial distributions. However, these results do not apply to processes admitting several quasi-stationary
distributions, which is known to happen in a variety of specific cases, even for processes irreducible in $E$ (including branching
processes~\cite{SenetaVere-Jones1966,athreya-ney-72,Lambert2007,maillard-18}, one-dimensional birth and death
processes~\cite{DoornErik1991,FerrariMartinezEtAl1991,FerrariKestenEtAl1995,Villemonais2015} and one-dimensional diffusion
processes~\cite{lladser-sanmartin-00,martinez-sanmartin-04}). In addition, as for non-absorbed processes, uniform convergence with
respect to the initial distribution only happens for processes that come back quickly in compact
sets~\cite{MeynTweedie2009,ChampagnatVillemonais2016b} or are killed fast~\cite{Velleret2018}. The present paper provides general
criteria generalizing those of~\cite{ChampagnatVillemonais2016b} to cases of non-uniform convergence. % and, contrary
  % to the above cited references, does not assume that $\P_x(t<\tau_\d)>0$ for all $x\in E$ and all $t\in I$. [ce dernier point est
  % secondaire,  je l'ai d\'eplac\'e en section 2]}

Given a quasi-stationary distribution $\nu$, its domain of attraction is defined as the set of probability measures $\mu$ on $E$ such
that $\PP_\mu(X_t\in \cdot\mid t<\tau_\partial)$ converges in total variation norm to $\nu$. In the case where the domain of
attraction of $\nu$ contains all Dirac masses, $\nu$ is called the \emph{Yaglom limit}, or the \emph{minimal quasi-stationary
  distribution}. In all the models admitting several quasi-stationary distributions cited above, it has been proved that the minimal
quasi-stationary distribution exists. The convergence~\eqref{eq:conv-intro} implies in addition that the domain of attraction of the
Yaglom limit $\nu_{QSD}$ actually contains all measures $\mu$ such that $\mu(\varphi_1)<\infty$ and $\mu(\varphi_2)>0$.

\bigskip

We provide in Section~\ref{sec:main-discrete-time} criteria ensuring~\eqref{eq:conv-intro} for all $t\in\ZZ_+$. We also obtain
several consequences, including a large subset of the domain of attraction of $\nu_{QSD}$ and the geometric uniform convergence of
$x\mapsto e^{\lambda_0 n}\PP_x(n<\tau_\d)/\varphi_1(x)$ as $n\rightarrow+\infty$ to $\eta/\varphi_1$, where $\eta$ is a function
which satisfies $\EE_x(\eta(X_n)\11_{n<\tau_\d})=e^{-\lambda_0 n}\eta(x)$ for all $n\in\ZZ_+$ and $x\in E$. We also obtain the
existence of the process $(X_n,n\in\ZZ_+)$ conditioned to never be absorbed (the so-called $Q$-process) and its geometric ergodicity.
Links between ergodicity of the $Q$-processes and quasi-limiting properties were already studied in various context (see for
instance~\cite{ArjasNummelinEtAl1980,GongQianEtAl1988,Miura2014,Takeda2019,FerreRoussetStoltz2018,Ocafrain2021}). All these results
are proved in Sections~\ref{sec:proof} and~\ref{sec:pf-other-results}.

The criterion developed in Section~\ref{sec:main-discrete-time} assumes that $(X_n,n\in\ZZ_+)$ is aperiodic but of course applies to
$1$-periodic processes $(X_t,t\in I)$. Under additional aperiodicity assumptions, we show in
Section~\ref{sec:equivalent-formulations} how the previous results extend to general time indices $t\in I$ and provide practical
versions of our criteria for continuous-time processes. We also provide alternative conditions allowing to check our criteria, that
are easier to check in some cases. We also show that the known criteria for uniform convergence in~\eqref{eq:conv-intro}
obtained in~\cite{ChampagnatVillemonais2016b} can be recovered using this new approach. These results are proved in
Section~\ref{sec:proof-equiv-form}.

These results allow us to put in a unified framework a large body of works on quasi-stationary distributions as illustrated by the
rest of the paper, which is devoted to the application of our abstract criteria. We start in Section~\ref{sec:diffusion} with
diffusion processes in $\RR^d$, $d\geq 1$, absorbed at the boundary of a domain $D$. Our analysis provides for example the following
general result.

\begin{thm}
  \label{thm:intro-diffusion}
  Assume that $E=D$ is a bounded connected open subset of $\RR^d$ and that $(X_t,t\in\RR_+)$ is solution to
  \[
  \mathrm dX_t=b(X_t)\mathrm dt+\sigma(X_t)\mathrm dB_t
  \]
  until its first exit time $\tau_\d$ from $D$, where $B$ is a $r$-dimensional Brownian motion and $b:\R^d\rightarrow\R^d$ and
  $\sigma:\R^d\rightarrow\R^{d\times r}$ are H\"older functions, such that $\sigma$ is uniformly elliptic. Then, the process $X$ has
  a unique quasi-stationary distribution $\nu_{QSD}$ which satisfies
  \begin{align*}
    \left\|\P_\mu(X_{t}\in\cdot\mid t<\tau_\d) -\nu_{QSD}\right\|_{TV}
    &\leq \frac{C}{\mu(\eta)}\,\alpha^t,\ \forall t\in [0,+\infty),
  \end{align*}
  for some constants $C<+\infty$ and $\alpha\in(0,1)$, where the function $\eta$ is
  $\mathcal{C}^2(D)$ and satisfies
  \[
  \sum_{i=1}^d b_i(x)\frac{\partial \eta}{\partial x_i}(x)+\frac{1}{2}\sum_{i,j=1}^d\sum_{k=1}^r
  \sigma_{ik}(x)\sigma_{jk}(x)\frac{\partial^2 \eta}{\partial x_i\partial x_j}(x)=-\lambda_0 \eta(x),\quad\forall x\in D
  \]
  and
  \[
  \eta(x) =\lim_{t\rightarrow+\infty}e^{\lambda_0 t}\PP_x(t<\tau_\d),\quad\forall x\in D,
  \]
  where the convergence is uniform in $D$.
\end{thm}

We emphasize that one of the main contributions of this result with respect to the existing literature (see for
example~\cite{Pinsky1985,GongQianEtAl1988,CattiauxMeleard2010,KnoblochPartzsch2010,DelVillemonais2018,ChampagnatCoulibaly-PasquierEtAl2016,ChampagnatVillemonais2017})
is that it applies to any bounded domain $D$ without any regularity assumption, with possible applications to recent Monte-Carlo
methods (see~\cite{PollockFearnheadEtAl2016,WangKolbEtAl2019}). Theorem~\ref{thm:intro-diffusion} is in fact obtained in
Section~\ref{sec:diffusion} as a particular case of a criterion for unbounded domains and coefficients $b$ and $\sigma$ only locally
H\"older and locally uniformly elliptic in $D$. We also consider the case of diffusions with killing in
Section~\ref{sec:diff-with-killing}. All these results are proved in Section~\ref{sec:pf-general-diffusion}.

Absorbed one-dimensional diffusions with or without killing have received a lot of attention (see for
instance~\cite{mandl-61,ColletMartinezEtAl1995,lladser-sanmartin-00,martinez-sanmartin-04,SteinsaltzEvans2004,CattiauxColletEtAl2009,LittinC.2012,KolbSteinsaltz2012,HeningKolb2014,miura-14,ChampagnatVillemonais2015,ChampagnatVillemonais2017ALEA}).
We consider these models in Section~\ref{sec:diffusion-1d}. Our main contributions with respect to the literature are the
characterization of a larger subset of the domain of attraction of the minimal quasi-stationary distribution, weaker regularity of
the drift and diffusion coefficients and explicit general bounds on $\varphi_1$ and $\lambda_0$ allowing practical verification of
our assumptions. Our criteria also provide alternative approaches to other classes of processes in continuous time and space, as
those studied for example in~\cite{ColletMartinezEtAl2011,benaim2021degenerate} using a spectral approach based on Tychonov's fixed
point theorem,
in~\cite{HinrichKolbEtAl2018,FerreRoussetStoltz2018,GuillinNectouxEtAl2020,CastroLambEtAl2021,BenaiemChampagnatEtAl2022} based on
compactness or quasi-compactness properties, and in~\cite{Marguet2017} for branching Markov processes using Lyapunov conditions on
the conditioned semigroup.

The case of continuous-time Markov processes in discrete state spaces is considered in Section~\ref{sec:appl-discr-state} with
application to multitype birth and death processes absorbed at the exit of any connected $E\subset\ZZ_+^d$ (in the sense of the
nearest neighbors structure of $\ZZ_+^d$). Note that the quasi-stationary behavior of finite state space
processes~\cite{DarrochSeneta1967} and of one-dimensional birth and death
processes~\cite{KarlinMcGregor1957,Good1968,Cavender1978,KijimaSeneta1991,DoornErik1991,DoornErik2012} has been extensively studied
using spectral methods that do not generalize easily to the multi-dimensional countable state-space setting. The quasi-stationary
behavior of multi-dimensional birth and death processes was studied in the case of uniform convergence in~\eqref{eq:conv-intro}
in~\cite{ChampagnatVillemonais2016a,ChampagnatVillemonais2017,ChazottesColletEtAl2016,ChazottesColletEtAl2017}.

All the previous examples assumed irreducibility of $X$ in $E$. In Section~\ref{sec:reducible}, we show that our criteria also apply
to reducible cases, as those considered in~\cite{Ogura1975} (for Galton-Watson processes),~\cite{gosselin-01} (for discrete
processes),~\cite{ChampagnatRoelly2008} (for Feller diffusions) and~\cite{ChampagnatDiaconisEtAl2012,DoornPollett2013} (in the finite
case). We first give a general criterion in Subsection~\ref{sec:three-com-class} and we study in details an example with
a countable infinity of communication classes in Subsection~\ref{sec:count-com-class}.

In Section~\ref{sec:continuous-state-space}, we consider general models in discrete time and continuous space, first extending the
criteria of~\cite{birkhoff-57,ChampagnatCoulibaly-PasquierEtAl2016} in order to cover the case of Euler schemes for stochastic
differential equations absorbed at the boundary of a domain (as defined in~\cite{manella-99,gobet-00}) and penalized semigroups (as
in~\cite{DelMoral2004,DelMoral2013}; note that all our results naturally extend to penalized homogeneous semigroups, provided the
penalization rate is bounded from above, see~\cite{ChampagnatVillemonais2018,ChampagnatVillemonais2019}). We then study in details
the case of perturbed dynamical systems, as those considered for example
in~\cite{BerglundLandon2012,baudel-berglund-17,HinrichKolbEtAl2018}, where the quasi-stationary behavior was studied using the
criterion of~\cite{birkhoff-57}. As an illustration of our method, let us mention the following original result.

\begin{thm}
  \label{thm:perturbed-SD-intro}
  Let $E=D$ be a measurable set of $\RR^d$ with positive Lebesgue measure and let $\d\not\in D$. Assume that
  \begin{align*}
    X_{n+1}=
    \begin{cases}
      f(X_n)+\xi_n & \text{if }X_n\neq\d\text{ and }f(X_n)+\xi_n\in D, \\
      \d & \text{otherwise,}
    \end{cases}
  \end{align*}
  where $f:\R^d\rightarrow \R^d$ is a locally bounded measurable function such that 
  \[
  |x|-|f(x)|\xrightarrow[|x|\rightarrow+\infty]{} +\infty
  \]
  and $(\xi_n)_{n\in\N}$ is an i.i.d.\ non-degenerate Gaussian sequence in $\R^d$. Then~\eqref{eq:conv-intro} is satisfied for
  $\varphi_1(x)=e^{|x|}$ and a positive measurable function $\varphi_2$ on $D$.
\end{thm}

Finally, we study in Section~\ref{sec:discrete-state-space} the case of processes in discrete time and discrete space. This is the
most studied situation in the literature since it covers both the Galton-Watson
processes~\cite{Yaglom1947,harris-63,JoffeSpitzer1967,athreya-ney-72} and the general discrete
case~\cite{darroch-seneta-65,SenetaVere-Jones1966,FerrariMartinezEtAl1991,FerrariMartinezEtAl1992,FerrariKestenEtAl1995,FerrariKestenMartinez1996,gosselin-01,MartinezSanMartinEtAl2014}.
We first show in Subsection~\ref{sec:R-positive} that our results allow to recover the general criterion
of~\cite{FerrariKestenMartinez1996}, based on the theory of $R$-positive matrices. We then consider general population processes
dominated by population-dependent multi-type Galton-Watson processes in Subsection~\ref{sec:ex-GW}. The case of population-dependent
Galton-Watson processes with a single type was studied in~\cite{gosselin-01} using quasi-compactness methods. We also obtain as
a corollary several results on subcritical multi-type Galton-Watson processes. We do not recover the optimal $L\log L$ assumption on the
offspring distribution~\cite{JoffeSpitzer1967,heathcote-seneta-verejones-67} for the existence of a minimal quasi-stationary
distribution $\nu_{QSD}$ having finite first moment, but we obtain a stronger form of convergence in~\eqref{eq:conv-intro}, a larger
subset of its domain of attraction and stronger moments properties on $\nu_{QSD}$.

\section{Main Results}
\label{sec:main-discrete-time}

Let $(X_t,t\in I)$ be a Markov process in $E\cup\{\d\}$ where $E$ is a measurable space and $\d\not\in E$, with set of time indices
$I$ which might be $\ZZ_+=\{0,1,\ldots\}$, $\RR_+$ or $\frac{1}{k}\ZZ_+$ for some $k\in\NN=\{1,2,\ldots\}$. We define the absorption
time $\tau_\d$ as
\[
\tau_\d=\inf\{t\in I,\,X_t=\d\}.
\]
In this section, we study the sub-Markovian transition semigroup of $X$ considered at integer times, $(P_n)_{n\in\Z_+}$, defined as
\begin{align*}
P_n f(x)=\E_x\left(f(X_n)\11_{n<\tau_\d}\right),\ \forall n\in\Z_+,
\end{align*}
for all bounded or nonnegative measurable function $f$ on $E$ and all $x\in E$. We also define as usual the left-action of $P_n$ on
measures as
\[
\mu P_n f=\EE_\mu\left(f(X_n)\11_{n<\tau_\d}\right)=\int_E P_nf(x)\,\mu(\mathrm dx),
\]
for all probability measure $\mu$ on $E$. We make the following assumption.

\medskip\noindent\textbf{Assumption (E).} There exist a positive integer $n_1$, positive real constants
$\theta_1,\theta_2,c_1,c_2,c_3$, two functions $\varphi_1,\varphi_2:E\rightarrow \R_+$ and a probability measure $\nu$ on a
measurable subset $K\subset E$ such that
\begin{itemize}
\item[(E1)] \textit{(Local Dobrushin coefficient).} $\forall x\in K$, 
  \begin{align*}
    \P_x(X_{n_1}\in\cdot)\geq c_1 \nu(\cdot\cap K).
  \end{align*}
\item[(E2)] \textit{(Global Lyapunov criterion).} We have $\theta_1<\theta_2$ and
  \begin{align*}
    &\inf_{x\in E} \varphi_1(x)\geq 1,\ \sup_{x\in K}\varphi_1(x)<\infty \\
    &\inf_{x\in K} \varphi_2(x)>0,\ \sup_{x\in E}\varphi_2(x)\leq 1,\\
    &P_{1}\varphi_1(x)\leq \theta_1\varphi_1(x)+c_2\11_K(x),\ \forall x\in E\\
    &P_{1}\varphi_2(x)\geq \theta_2\varphi_2(x),\ \forall x\in E.
  \end{align*}
\item[(E3)] \textit{(Local Harnack inequality).} We have
  \begin{align*}
    \sup_{n\in \Z_+}\frac{\sup_{y\in K} \P_y(n<\tau_\d)}{\inf_{y\in K} \P_y(n<\tau_\d)}\leq c_3
  \end{align*}
\item[(E4)] \textit{(Aperiodicity).} For all $x\in K$, there exists $n_4(x)$ such that, for all $n\geq n_4(x)$,
  \begin{align*}
    \P_x(X_n\in K)>0.
  \end{align*}
\end{itemize}

Note that it follows from (E2) that $\theta_2\leq 1$ and thus $\theta_1<1$. We also emphasize that our assumptions neither require
that $\tau_\d<+\infty$ $\P_x$-a.s., nor that $\P_x(n<\tau_\d)>0$ for all $t\geq 0$ and $x\in E$. Several examples of Markov processes
satisfying Assumption~(E) are provided in Sections~\ref{sec:diffusion} to~\ref{sec:discrete-state-space}.

Assumption~(E) is an extension of the ergodicity criteria developed in~\cite{MeynTweedie1993}. Indeed, if we assume that
$\tau_\d=\infty$ $\PP_x$-almost surely for all $x\in E$, then Condition (E3) becomes void and one can take $\varphi_2\equiv 1$
in~(E2), so that $\theta_2=\theta_0=1$. We recognize in~(E1) the standard ``small set'' assumption of~\cite{MeynTweedie1993}, in~(E2)
for $\varphi_1$ a standard Foster-Lyapunov criterion and in~(E4) an aperiodicity condition. As such, it is well-known that
alternative formulations of these conditions can be given. In the general case, we provide in Section~\ref{sec:general-comments} conditions ensuring the
existence of Lyapunov functions satisfying (E2) in terms of exponential moment of hitting times for $\varphi_1$ and exponential decay
of the probability to be in $K$ for $\varphi_2$, and conditions ensuring (E1) and (E3) based on comparisons between transition
probabilities. Similarly as for the ergodicity criteria developed in~\cite{MeynTweedie1993}, we extend our criterion to the
continuous-time setting in Section~\ref{sec:continuous-time}.

In the rest of this section, we state the main general consequences of Assumption~(E). We start with the exponential contraction in total variation of the
conditional marginal distributions of the process given non-absorption. Its proof is given in Section~\ref{sec:proof}. 

\begin{thm}
  \label{thm:main}
  Assume that Condition~(E) holds true. Then there exist a constant $C>0$, a constant
  $\alpha\in(0,1)$, a probability measure $\nu_{QSD}$ on $E$ such that $\nu_{QSD}(K)>0$ and such that
  \begin{align}
    \label{eq:thm-main-zxzx}
    \left\|\frac{\mu P_{n}}{\mu P_{n}\11_E} -\nu_{QSD}\right\|_{TV(\varphi_1)}
    &\leq C\,\alpha^n\, %\|h\|_{L^\infty(\varphi_1)}
      \frac{\mu(\varphi_1)}{\mu(\varphi_2)},\quad\forall n\geq 0,
  \end{align}
  for all probability measure $\mu$ on $E$ such that $\mu(\varphi_1)<\infty$ and $\mu(\varphi_2)>0$. %, and all functions $h\in
  % L^\infty(\varphi_1)$.
  In addition, $\nu_{QSD}$ is the unique quasi-stationary distribution satisfying $\nu_{QSD}(\varphi_2)>0$ and
  $\nu_{QSD}(\varphi_1)<\infty$. 
%   Moreover,  for all $\mu$ such that $0<\mu(\eta)<+\infty$, we have
%  \begin{align}
%  \label{eq:dom-act}
%  \left\|\frac{\mu P_{n}}{\mu P_{n}\11_E} -\nu_{QSD}\right\|_{TV}\xrightarrow[n\to+\infty]{} 0.
%  \end{align}
\end{thm}

%\begin{rem}\label{rem:uniquenessQSD}
%    Convergence~\eqref{eq:thm-main-zxzx} entails that all probability measure $\mu$ on $E$ such that $\mu(\varphi_1)<+\infty$ and $\mu(\varphi_2)>0$ is in the domain of attraction of $\nu_{QSD}$, which means that $\P_\mu(X_n\in\cdot\mid n<\tau_\d)$ converges to $\nu_{QSD}$. In particular,    
%     $\nu_{QSD}$ is the
%    unique quasi-stationary distribution of $X$ that satisfies $\nu_{QSD}(\varphi_1)<\infty$ and $\nu_{QSD}(\varphi_2)>0$. In particular, if  $\varphi_1$ is bounded, then $\nu_{QSD}$ is the unique quasi-stationary distribution of $X$ such that $\nu_{QSD}(\varphi_2)>0$.
%\end{rem}

\begin{rem}
  \label{rem:dom-attraction}
  For all $p\geq 1$, H\"older's inequality entails
  \begin{align*}
    P_1(\varphi_1^{1/p})\leq (\theta_1\varphi_1+c_2\11_K)^{1/p}\leq \theta_1^{1/p}\varphi_1^{1/p}+c_2^{1/p}\11_K,
  \end{align*}
  so that $(\varphi^{1/p}_1,\varphi_2)$ satisfies Assumption~(E) for all $p<\log\theta_1/\log \theta_2$. Therefore, the exponential
  convergence~\eqref{eq:thm-main-zxzx} actually holds true for the norm $\|\cdot\|_{TV(\varphi_1^{1/p})}$ and
  measures $\mu$ such that $\mu(\varphi_1^{1/p})<+\infty$ for some $p<\log\theta_1/\log \theta_2$.
\end{rem}

 In the following result,
 we show the existence of an eigenfunction $\eta$ of $P_1$ for the
eigenvalue $\theta_0$, where $\theta_0\in(0,1]$ is such that
$$
\PP_{\nu_{QSD}}(n<\tau_\d)=\theta_0^n,\quad\forall n\in\NN.
$$
We recall that the existence of the decay parameter $\theta_0$ is a classical general result for quasi-stationary
distributions~\cite{MeleardVillemonais2012,ColletMartinezEtAl2013}.  
%{\color{gray} Note that, if $\tau_\d<\infty$ with positive $\PP_x$-probability for all
%$x\in K$, $\theta_0<1$ and in this case, absorption occurs in finite time $\PP_{\nu_{QSD}}$-almost surely. The case
%$\theta_0=1$ corresponds to the case where $\tau_\d=\infty$ $\PP_{\nu_{QSD}}$-almost surely. Because of the next
%Theorem~\ref{thm:eta}, under Condition~(E), this will occurs if and only if there exists $x\in E$ such that
%$\tau_\d=+\infty$ $\PP_x$-almost surely.} 
The proof of the following result is initiated in Section~\ref{sec:pf-thm-eta} and concluded in Section~\ref{sec:end-eta-main}. To
state this result, we define for all positive function $\psi$ on $E$ the space $L^\infty(\psi)$ as the set of measurable real
functions $f$ on $E$ such that $\|f\|_{L^\infty(\psi)}:=\sup_{x\in E}f(x)/\psi(x)<\infty$. Note that
$(L^\infty(\psi),\|\cdot\|_{L^\infty(\psi)})$ is a Banach space.

\begin{thm}
  \label{thm:eta}
  Assume that Condition~(E) holds true. Then there exists a function $\eta:E\rightarrow\RR_+$  such that
  \begin{equation}
    \label{eq:conv-eta-general}
    \eta(x)=\lim_{n\rightarrow+\infty}\frac{\PP_x(n<\tau_\d)}{\PP_{\nu_{QSD}}(n<\tau_\d)}
    =\lim_{n\rightarrow+\infty}\theta_0^{-n}\PP_x(n<\tau_\d),\quad\forall x\in E,
  \end{equation}
  where the convergence is geometric in $L^\infty(\varphi_1)$. In addition,
 we have $\inf_{y\in K}\eta(y)>0$, $\nu_{QSD}(\eta)=1$, $\eta\in
 L^\infty\left(\varphi_1^{\log{(1/\theta_0)}/\log{(1/\theta_1)}}\right)$, %$E'=\{x\in E:\eta(x)>0\}$,
  \begin{align*}
    P_1\eta=\theta_0\eta\quad\text{and}\quad\theta_0\geq\theta_2>\theta_1.
  \end{align*}
\end{thm}

\begin{rem}
  \label{rem:link-eta-phi_2}
  In general, there is no simple relation between $\varphi_2$ and $\eta$, in particular $\varphi_2$ is not necessarily an element of $L^\infty(\eta)$.
  However, it is true that, for all $x\in E$, $P_k\varphi_2(x)>0$ for some $k\geq 0$ if and only if $\eta(x)>0$ (see
    Corollary~\ref{cor:attraction-domain} below).
\end{rem}

\begin{rem}
  Note that, when $\eta$ is bounded, the last result implies that one can actually take $\varphi_2=\eta/\|\eta\|_\infty$ in
  Condition~(E2). Results with unbounded $\varphi_2$ or $1/\varphi_1$ can also be obtained by taking the $\varphi_1$-transform of
  $(P_n)_{n\in\Z_+}$ (see~\cite{BansayeCloezEtAl2019,ChampagnatVillemonais2019}).
\end{rem}

We consider now the $Q$-process and its ergodicity properties under Condition~(E). In the next
result, proved in Section~\ref{sec:pf-thm-Q-proc}, $\Omega=E^{\Z_+}$ is the canonical state space of Markov chains on
$E$ and $({\cal F}_n)_{n\in\Z_+}$ is the associated canonical filtration. We emphasize that the constant $\alpha$ may differ from the
one in Theorem~\ref{thm:main}. In the following result, we define
\[
  E':=\left\{x\in E,\,\eta(x)>0\right\}.
\]

\begin{thm}
  \label{thm:Q-proc}
  Condition~(E) implies the following properties. 
  \begin{description}
  \item[\textmd{(i) Existence of the $Q$-process.}] There exists a family $(\QQ_x)_{x\in E'}$ of
    probability measures on $\Omega$ defined by
    $$
    \lim_{n\rightarrow+\infty}\PP_x(A\mid n<\tau_\partial)=\QQ_x(A)
    $$
    for all $x\in E'$, for all ${\cal F}_m$-measurable set $A$ and for all $m\geq 0$. The process $(\Omega,({\cal F}_n)_{n\in\Z_+},(X_n)_{n\in\Z_+},(\QQ_x)_{x\in E'})$ is an $E'$-valued homogeneous Markov
    chain.
  \item[\textmd{(ii) Semigroup.}] The semigroup of the Markov process $X$ under $(\QQ_x)_{x\in E'}$ is given
    for all bounded measurable function $\varphi$ on $E'$ and $n\geq 0$ by
    \begin{align}
      \label{eq:semi-group-Q}
      \widetilde{P}_n\varphi(x)=\frac{\theta_0^{-n}}{\eta(x)}P_n(\eta\varphi)(x).
    \end{align}
  \item[\textmd{(iii) Exponential ergodicity.}] The probability measure $\beta$ on $E'$ defined by
    \begin{align*}
      \beta(\mathrm dx)=\eta(x)\nu_{QSD}(\mathrm dx).
    \end{align*}
    is the unique invariant distribution of the Markov process $X$ under $(\QQ_x)_{x\in E'}$. Moreover,  there exist constants $C>0$ and $\alpha\in(0,1)$ such that, for all initial
    distributions $\mu$ on $E'$ such that $\mu(\varphi_1/\eta)<\infty$ and % for all measurable real function $h$ on $E'$ such that $|h|\leq \varphi_1/\eta$,
    \begin{align}
      \label{eq:Q-proc-11}
      \left\|\mu \widetilde{P}_n-\beta(h)\right\|_{TV(\varphi_1/\eta)}\leq C\alpha^{n}\,\mu\left(\varphi_1/\eta\right),\quad\forall n\geq 0,
    \end{align}
    where $\Q_\mu=\int_{E'} \Q_x\,\mu(dx)$. In addition, for all initial distributions $\mu$ on $E'$,
    \begin{align}
      \label{eq:Q-proc-12}
      \left\|\mu\widetilde{P}_n-\beta\right\|_{TV}\xrightarrow[n\rightarrow\infty]{} 0.
    \end{align}
  \end{description}
\end{thm}

%{\color{gray}
%\begin{rem}
%	Whether Assumption~(E) is necessary for~\eqref{eq:thm-main-zxzx} is still an open problem up to our knowledge. However, if one assumes that there exists a non-negative eigenfunction $\eta$ such that the convergence~\eqref{eq:Q-proc-11} holds for $p=1$ and some $\varphi_1>0$, a partial counterpart can be established, as detailed in~\cite{BansayeCloezEtAl2019,ChampagnatVillemonais2019}.
%\end{rem}}

We conclude this section with corollaries of the last theorem. The following result is proved in Section~\ref{sec:end-eta-main}.

\begin{cor}
    \label{cor:quasi-comp}
    Assume that Condition~(E) holds true. Then there exist constants $C>0$ and $\alpha\in(0,1)$ such that, for all probability measure $\mu$ on $E$ such that $\mu(\varphi_1)<+\infty$,% and all function $g\in L^\infty(\varphi_1)$,
    \begin{align}
    \label{eq:quasi-comp}
    \left\|\theta_0^{-n}\mu P_n -\mu(\eta)\nu_{QSD}\right\|_{TV(\varphi_1)}\leq C\,\alpha^n\mu(\varphi_1). %\|g/\varphi_1\|_\infty.
    \end{align}
\end{cor}

\begin{rem}
    The proof of Theorem~\ref{thm:Q-proc} makes use of~\cite{Hairer2010,HairerMattingly2011}, which allows to derive explicit expressions for the constants $C$ and $\alpha$ (we refer the interested reader to Remark~\ref{rem:explici-constants}). In particular, using these estimates in the proof of Corollary~\ref{cor:quasi-comp} would also provide explicit constants in~\eqref{eq:quasi-comp}.
  \end{rem}

    \begin{rem}
      The formulation~\eqref{eq:quasi-comp} for the convergence of the semigroup is natural in this setting, since a property of
      equivalence between~\eqref{eq:quasi-comp} and Condition~(E) is proved in~\cite{BansayeCloezEtAl2019,ChampagnatVillemonais2019}.
    \end{rem}

The last corollary has consequences on the attraction domain of $\nu_{QSD}$.
\begin{cor}
  \label{cor:attraction-domain}
  Assume that Condition~(E) holds true. Then
  \[
    E'=\left\{x\in E:\exists k\geq 0,\ P_k\varphi_2(x)>0\right\}
  \]
  and the domain of attraction of $\nu_{QSD}$ for the total variation norm contains all
  probability measures on $E$ such that $\mu(E')>0$ and $\mu(\varphi_1^{1/p})<+\infty$ for some $p<\log\theta_1/\log\theta_2$. If in
  addition $\varphi_1$ is bounded, then the domain of attraction of $\nu_{QSD}$ is the set of probability measures on $E$ such that
  $\mu(E')>0$ and $\nu_{QSD}$ is the unique quasi-stationnary distribution giving positive mass to $E'$.
\end{cor}

Convergence estimates can also be obtained for initial distributions on $E'$ satisfying $\mu(\eta)<+\infty$ but not necessarily
$\mu(\varphi_1)<+\infty$. The following result is proved in~\ref{sec:proof-cor-2}.
\begin{cor}
    \label{cor:ratios}
    Assume that Condition~(E) holds true. Then, for all probability measures $\mu$ on $E'$ such that $\mu(\eta)<+\infty$, 
    \begin{align}
    \label{eq:statementcor2}
    %\sup_{g:E'\to \R,\ \|g\|_{L^\infty(\eta)}\leq 1} 
    \left\|\theta_0^{-n}\mu P^n -\mu(\eta)\nu_{QSD}\right\|_{TV(\eta)}\xrightarrow[n\to+\infty]{} 0.
    \end{align}
    In particular, if $\eta$ is positive on $E$, then $\nu_{QSD}$ is the unique quasi-stationary distribution of $X$ such that $\nu_{QSD}(\eta)<+\infty$. If in addition $\eta$ is lower bounded away from $0$ on $E$, then for all probability measures $\mu$ on $E$ such that $\mu(\eta)<+\infty$, we have
    \begin{align}
    \label{eq:statementcor2bis}
    \left\|\P_\mu(X_n\in\cdot\mid n<\tau_\partial)-\nu_{QSD}\right\|_{TV(\eta)}\xrightarrow[n\to+\infty]{} 0.
    \end{align}
    In particular, the domain of attraction of $\nu_{QSD}$ contains all probability measures $\mu$ on $E$ such that $\mu(\eta)<+\infty$.
\end{cor}

\section{Other formulations and particular cases of Assumption~(E)}
\label{sec:equivalent-formulations}

In this section, we provide general comments on Assumption~(E). Alternative formulations of our assumptions and
  simple criteria are gathered in Subsection~\ref{sec:general-comments}. Subsection~\ref{sec:continuous-time} focuses on criteria
adapted to continuous time processes and we consider the case of uniform convergence in Theorem~\ref{thm:main} in
Subsection~\ref{sec:A1-A2}.

\subsection{General comments on the assumptions}
\label{sec:general-comments}

We propose here alternative formulations of Condition~(E2) and criteria ensuring~(E1) and~(E3) when~(E2) and~(E4) are satisfied, that
may be easier to check in some practical situations. In particular, we make strong use of these results in
Sections~\ref{sec:continuous-state-space} and~\ref{sec:discrete-state-space}.

% Assumption~(E) is an extension of the ergodicity criteria developped in~\cite{MeynTweedie1993}. Indeed, if we assume that
% $\tau_\d=\infty$ $\PP_x$-almost surely for all $x\in E$, then Condition (E3) becomes void and one can take $\varphi_2\equiv 1$
% in~(E2), so that $\theta_2=\theta_0=1$. We recognize in~(E1) the standard ``small set'' assumption of~\cite{MeynTweedie1993}, in~(E2)
% for $\varphi_1$ a standard Foster-Lyapunov criterion and in~(E4) an aperiodicity condition.

\subsubsection{Construction of Lyapunov functions satisfying (E2)}
\label{sec:constr-lyap-funct}

In order to prove the existence of functions $\varphi_1$ and $\varphi_2$ in Condition~(E2), one may use probabilistic properties of
the Markov process $X$, as stated in the following lemmas, proved in Sections~\ref{sec:proof-varphi2} and~\ref{sec:proof-varphi1}.
The first lemma shows a way to construct $\varphi_2$.

\begin{lem}
  \label{lem:varphi2}
  Let $K$ be a measurable subset of $E$. If there exists $\theta_2>0$ such that
  \[
  \inf_{x\in K} \theta_2^{-n}\P_x(X_n\in K)\xrightarrow[n\rightarrow+\infty]{} +\infty,
  \]
  then the function $\varphi_2:E\rightarrow [0,1]$ defined by
  $\varphi_2(x)=\frac{\theta^{-1}_2-1}{\theta_2^{-\ell}-1}\sum_{k=0}^{\ell-1} \theta_2^{-k}\P_x(X_k\in K)$, for any $\ell$ is such that
  $\theta_2^{-\ell}\inf_{x\in K}\P_x(X_\ell\in K)\geq 1$, verifies $\inf_{K}\varphi_2>0$ and
  $P_1\varphi_2(x)\geq \theta_2\,\varphi_2(x)$. Moreover, (E4) is satisfied.
\end{lem}

The second lemma shows how $\varphi_1$ can be constructed. This is a well-known result in the case without
absorption~\cite{MeynTweedie1993}, which can provide easier ways to check~(E2) in some situations. We define
\begin{equation}
  \label{eq:def-T_K-section-3}
  T_K=\inf\{n\in\Z_+,\ X_n\in K\}.
\end{equation}

\begin{lem}
  \label{lem:varphi1}
  Let $K$ be a measurable subset of $E$. 
If there exists a constant $\theta_1>0$ such that  
    \[
  \E_x\left(\theta_1^{- T_K\wedge\tau_\d}\right)<+\infty\,\ \forall x\in E\text{ and } C:=\sup_{y\in K}\,
  \E_y\left(\E_{X_1}\left(\theta_1^{- T_K\wedge\tau_\d}\right)\11_{1<\tau_\d}\right)<+\infty,
  \]
  then the function $\varphi_1:E\rightarrow[1,+\infty)$ defined by $\varphi_1(x)=\E_x\left(\theta_1^{- T_K\wedge\lceil\tau_\d\rceil}\right)$ satisfies 
  \[
  \sup_{K}\varphi_1<+\infty\quad\text{ and }\quad P_1\varphi_1\leq \theta_1\varphi_1+\frac{C}{\theta_1}\11_K.
  \]
  Conversely, if there exist two constants $C>0$, $\theta_1>0$ and a function $\varphi_1:E\to [1,+\infty)$ such that
  $\sup_{K}\varphi_1<+\infty$ and $ P_1\varphi_1\leq \theta_1\varphi_1+C\11_K$, then, for all $\theta>\theta_1$, there exists a
  constant $C_\theta$ such that
  \[
  \E_x\left(\theta^{- T_K\wedge\tau_\d}\right)\leq C_\theta\varphi_1(x)\,\ \forall x\in E\text{ and } \sup_{y\in K}\,
  \E_y\left(\E_{X_1}\left(\theta^{- T_K\wedge\tau_\d}\right)\11_{1<\tau_\d}\right)<+\infty.
  \]
\end{lem}

Note that the hitting time $T_K$ is defined from the process $(X_n)_{n\in\Z_+}$. When $I\neq\Z_+$, it might be easier to use criteria based
on the hitting time $\tau_K$ defined from the full process $(X_n)_{n\in I}$. We refer the reader to Lemma~\ref{prop:E-G} below for that.

\subsubsection{Checking (E1) and (E3) from comparisons between transition probabilities}
\label{sec:how-check-e1-e4}

Condition~(E3) is a form of Harnack inequality, and one can indeed use general versions of these inequalities to check~(E3) and~(E1)
(for example, our results on diffusions given in Section~\ref{sec:diffusion} use this idea, cf.\ Section~\ref{sec:harnack-diff}). We
propose below another criterion, based on comparison techniques on transition probabilities, to check that Conditions~(E1) and (E3)
hold true when Conditions~(E2) and~(E4) are satisfied. This result is proved in Subsection~\ref{sec:proof-E2thenE}.

\begin{prop}
  \label{lem:E2thenE}
  Assume that Conditions~(E2) and~(E4) are satisfied and that there exist two constants $C>0$ and $n_0\leq m_0\in \N$ such that
  \begin{equation}
    \label{eq:condition-E2thenE}
    \P_x(X_{n_0}\in \cdot\cap K)\leq C\,\P_y(X_{m_0}\in \cdot),\ \forall x\in E\text{ and }y\in K.
  \end{equation}
  Then Condition~(E) is satisfied.
\end{prop}

\subsubsection{Optimal value of $\theta_2$ in (E2)}
\label{sec:optimal-theta_2-in-(E2)}

As many results of Section~\ref{sec:main-discrete-time} make use of the function $\varphi_1^{1/p}$ with a parameter
$p\in[1,\log\theta_1/\log\theta_2)$, it is important to characterize the largest possible value of $\theta_2$. %{\color{gray} The following lemma shows
%that the domain of attraction provided by Corollary~\ref{cor:dom-attract-final} can be taken as the set of probability measures $\mu$
%on $E$ such that $\mu(E')>0$ and $\mu(\varphi^{1/p})<\infty$ for some $p<\log\theta_1/\log\theta_0$.} 
This result is proved in
Section~\ref{sec:proof-theta-2-et-0}.

\begin{lem}
  \label{lem:theta-2-et-0}
  If Condition~(E) is satisfied for some functions $\varphi_1 $ and $\varphi_2$ with constants $\theta_1$ and $\theta_2$, then, for
  all $\theta'_2\in(\theta_1,\theta_0)$ it is also satisfied for $\varphi_1$ and some function $\varphi'_2$ with constants $\theta_1$ and
  $\theta'_2$.
\end{lem}

\subsection{On continuous time}
\label{sec:continuous-time}

In Section~\ref{sec:main-discrete-time}, we only considered the conditional behavior of the process $X$ at integer times. In general,
the results of Section~\ref{sec:main-discrete-time} do not give information about the process at intermediate times. In this section,
we derive a sufficient condition which is well suited for practical verification in the case of continuous time Markov processes or
for aperiodic Markov processes, in particular because (F2) below is usually easier to check than (E2). We consider an
absorbed Markov process $(X_t)_{t\in I}$ with time parameter in $I=\Z_+$ or $[0,+\infty)$.

\medskip\noindent\textbf{Assumption (F).} There exist positive real constants $\gamma_1,\gamma_2,c_1,c_2$ and $c_3$, $t_1,t_2\in I$,
a measurable function $\psi_1:E\rightarrow [1,+\infty)$, and a probability measure $\nu$ on a measurable subset $L\subset E$ such
that
\begin{itemize}
\item[(F0)] \textit{(A strong Markov property).} 
  Defining
  \begin{align}
    \label{eq:def-T_L}
    \tau_L:=\inf\{t\in I:X_t\in L\},
  \end{align}
  assume that for all $x\in E$, $X_{\tau_L}\in L$, $\PP_x$-almost surely on the event $\{\tau_L<\infty\}$ and for all $t\in I$ and
  all measurable $f:E\cup\{\d\}\rightarrow\RR_+$,
  \begin{align*}
    \EE_x\left[f(X_t)\11_{\tau_L\leq t<\tau_\d}\right]=\EE_x\left[\11_{\tau_L\leq
        t\wedge\tau_\d}\restriction{\EE_{X_{\tau_L}}\left[f(X_{t-u})\11_{t-u<\tau_\d}\right]}{u=\tau_L}\right].   
  \end{align*}
\item[(F1)] \textit{(Local Dobrushin coefficient).} $\forall x\in L$, 
  \begin{align*}
    \P_x(X_{t_1}\in\cdot)\geq c_1 \nu(\cdot\cap L).
  \end{align*}
\item[(F2)] \textit{(Global Lyapunov criterion).} We have $\gamma_1<\gamma_2$ and
  \begin{align*}
    &\E_x(\psi_1(X_{t_2})\11_{t_2<\tau_L\wedge\tau_\d})\leq \gamma_1^{t_2}\psi_1(x),\ \forall x\in E\\
    &\E_x(\psi_1(X_t)\11_{t<\tau_\d})\leq c_2,\ \forall x\in L,\ \forall t\in[0,t_2]\cap I,\\
    &\gamma_2^{-t}\P_x(X_t\in L)\xrightarrow[t\rightarrow+\infty]{} +\infty,\ \forall x\in L.
  \end{align*}
\item[(F3)] \textit{(Local Harnack inequality).} We have
  \begin{align*}
    \sup_{t\geq 0}\frac{\sup_{y\in L} \P_y(t<\tau_\d)}{\inf_{y\in L} \P_y(t<\tau_\d)}\leq c_3
  \end{align*}
\end{itemize}

Be careful that the definition of $\tau_L$ in~\eqref{eq:def-T_L} is different from that of $T_L$ in~\eqref{eq:def-T_K-section-3}.
Note also that, in~(F2), the Lyapunov function $\varphi_2$ has been replaced by an alternative condition similar to
Lemma~\ref{lem:varphi2}. Both are actually equivalent thanks to~(F0) (see the beginning of Section~\ref{sec:proof-(E)}).

The following result is proved in Section~\ref{sec:proof-E-F}.

\begin{thm}
  \label{thm:E-F}
  Under Assumption (F), $(X_t)_{t\in I}$ admits a quasi-stationary distribution $\nu_{QSD}$, which is the unique one satisfying
  $\nu_{QSD}(\psi_1)<\infty$ and $\nu_{QSD}(L)>0$ for some $t\in I$. Moreover, there exist constants $\alpha\in(0,1)$
  and $C>0$ such that, for all probability measures $\mu$ on $E$ satisfying $\mu(\psi_1)<\infty$ and $\mu(\psi_2)>0$,% and for all $h\in L^\infty(\psi_1)$,
  \begin{align}
    \left\|\P_\mu(X_{t}\in\cdot\mid t<\tau_\d) -\nu_{QSD}\right\|_{TV(\psi_1)}
    &\leq C\,\alpha^t\,\frac{\mu (\psi_1)}{\mu(\psi_2)} %\,\|h/\psi_1\|_\infty
    ,\ \forall t\in I, \label{eq:expo-cv-prop-E-F}
  \end{align}
  where $\psi_2(x)=\sum_{k=0}^{n_0}\gamma_2^{-kt_2}\P_x(X_{kt_2}\in L)$ for some $n_0\geq 1$ large enough. In addition, there exists a
  constant $\lambda_0\geq 0$ such that $\lambda_0\leq\log(1/\gamma_2)<\log(1/\gamma_1)$ and $\PP_{\nu_{QSD}}(t<\tau_\d)=e^{-\lambda_0 t}$
  for all $t\geq 0$, and there exists a function $\eta$ such that
  \begin{equation}
    \label{eq:eta-prop-E-F}
    \eta(x)=\lim_{t\rightarrow +\infty}e^{\lambda_0 t}\PP_x(t<\tau_\d),\quad\forall x\in E,    
  \end{equation}
  where the convergence is exponential in $L^\infty(\psi_1^{1/p})$ for all $p\in[1,\log(1/\gamma_1)/\lambda_0)$, and
  $P_t\eta(x)=e^{-\lambda_0 t}\eta(x)$ for all $x\in E$ and $t\in I$. 
\end{thm}

A key point that guided our formulation of Condition~(F) is that,
% \begin{rem}
%   \label{rem:Lyapunov-generator}
  for continuous-time Markov processes, usual practical conditions for the existence of $\psi_1$ are provided by Foster-Lyapunov
  inequalities (cf.~\cite{MeynTweedie1993}). They involve the extended
  infinitesimal generator $\bar\cL$ of the process $X$ (see e.g.~\cite{MeynTweedie1993,ChampagnatVillemonais2017}) and take the form
  \begin{align}
    \bar\cL\psi_1(x) & \leq -\lambda_1\psi_1(x)+C\11_K(x), \quad\forall x\in E.\label{eq:Foster-Lyap-not-working-1} 
  \end{align}
  This inequality does not imply, in general, that (E2) holds true for $\varphi_1=\psi_1$. However,
  Equation~\eqref{eq:Foster-Lyap-not-working-1} implies (formally, assuming one can apply Dynkin's formula) that $\E_x[\11_{1\leq
    \tau_L\wedge \tau_\d}\psi_1(X_1)]\leq e^{-\lambda_1}\psi_1(x)$ and $\E_x[\psi_1(X_t)\11_{t<\tau_\d}]\leq e^{Ct}\psi_1(x)$. Hence
  the first two lines of~(F2) can be deduced from classical Foster Lyapunov criteria. This will be used for diffusion processes in
  Section~\ref{sec:diffusion} or in discrete state space in Section~\ref{sec:appl-discr-state}.%  Note that a function $\psi_1$
  % satisfying~\eqref{eq:Foster-Lyap-not-working-1} usually does not belong to the domain of the infinitesimal generator $\cL$, so one
  % needs to extend the notion of infinitesimal generator as in~\cite{MeynTweedie1993,ChampagnatVillemonais2017}.
% \end{rem}

  Alternatively, one can use controls on the exponential moments for the return times in $L$. The following result, similar to
  Lemma~\ref{lem:varphi1}, is proved in Section~\ref{sec:proof-E-G}.

\begin{lem}
  \label{prop:E-G}
  Assume that there exist positive constants $\gamma_1>0$ and $t_2\in I$ such that
  \[
    \E_x\left(\gamma_1^{-\tau_L\wedge\tau_\d}\right)<\infty,\ \forall x\in E\quad\text{ and }\quad\sup_{x\in L} \E_x\left(\E_{X_{t_2}}\left(\gamma_1^{-\tau_L\wedge\tau_\d}\right)\11_{t_2<\tau_\d}\right)<+\infty,
  \]
  then $\psi_1(x)=\E_x\left(\gamma_1^{-\tau_L\wedge\tau_\d}\right)$ satisfies
  \begin{align*}
    &\E_x(\psi_1(X_{t_2})\11_{t_2<\tau_L\wedge\tau_\d})\leq \gamma_1^{t_2}\psi_1(x),\ \forall x\in E\\
    &\E_x(\psi_1(X_t)\11_{t<\tau_\d})\leq c_2,\ \forall x\in L,\ \forall t\in[0,t_2]\cap I,
  \end{align*}
  for some constant $c_2>0$.
\end{lem}

\begin{rem}
  In the proof of Theorem~\ref{thm:E-F}, we will show that Assumption~(F) implies that Assumption~(E) is satisfied for the sub-Markovian semigroup $(P_n)_{n\geq
    0}$ of the absorbed Markov process $(X_{nt_2})_{n\in\ZZ_+}$, with the functions $\varphi_1=\psi_1$ and
  $\varphi_2=\frac{\gamma_2^{-t_2}-1}{\gamma_2^{-(n_0+1)t_2}-1}\psi_2$, any $\theta_1\in (\gamma_1^{t_2},\gamma_2^{t_2})$,
  $\theta_2=\gamma_2^{t_2}$ and the set
  \[
    K=\left\{y\in E,\ \P_y(\tau_L \leq t_2)/\psi_1(y)\geq (\theta_1-\gamma_1^{t_2})/c_2\right\}\supset L.
  \]
  In particular, all the consequences of~(E) stated in Section~\ref{sec:main-discrete-time} hold true. Moreover, it is also possible
  to obtain a continuous-time version of Theorem~\ref{thm:Q-proc} about the $Q$-process by adapting the proof given in
  Section~\ref{sec:pf-thm-Q-proc}. 
\end{rem}

\begin{rem}
  If $I=\R_+$, it follows from the fact that $P_t\eta=e^{-\lambda_0 t}\eta$ that, setting $\eta(\d)=0$, the function $\eta$ defined
  on $E\cup\{\d\}$ belongs to the domain of the infinitesimal generator $\cL$ of the semigroup of the Markov process $X$ on
  $E\cup\{\d\}$, seen as acting on $L^\infty(\psi_1^{1/p})$ for $p\in[1,\log(1/\gamma_1)/\lambda_0)$, and $\cL\eta=-\lambda_0
  \eta$.
\end{rem}

\subsection{The case of uniform exponential convergence}
\label{sec:A1-A2}

We now want to characterize the case of exponential convergence in total variation of the conditional distributions of $(X_n)$ to
$\nu_{QSD}$, uniformly with respect to the initial distribution $\mu$. This question was already studied in~\cite{ChampagnatVillemonais2016b}. The next result, proved in Section~\ref{sec:proof-A1-A2}, gives a necessary and sufficient condition based on
Condition~(E).

\begin{prop}
  \label{prop:A1-A2}
  There exists constants $C$ and $\alpha<1$ such that, for all probability measure $\mu$ on $E$ and all integer $n$,
  \begin{align}
    \label{eq:unif-1}
    \left\|\P_\mu(X_n\in\cdot\mid n<\tau_\d)-\nu_{QSD}\right\|_{TV}\leq C\alpha^n,
  \end{align}
  if and only if Condition~(E) is satisfied with a bounded function $\varphi_1$ and there exists an integer $n'_4>0$ such that
  \begin{align}
    \label{eq:unif-2}
    \underline{c}:=\inf_{x\in E} \P_x(X_{n'_4}\in K\mid n'_4<\tau_\d)>0.
  \end{align}
\end{prop}

\section{Application to diffusion processes}
\label{sec:diffusion}

In this section, we apply the criteria~(E) and~(F) to diffusion processes absorbed at the boundary of a domain. We give a general
criterion in Subsection~\ref{sec:general-diffusion} and apply it to uniformly elliptic diffusions in
Subsection~\ref{sec:appl-diff-unif-ellipt} and to an example with vanishing diffusion coefficient at the boundary of the domain in
Subsection~\ref{sec:non-unif-ellipt}. Our criteria are extended to diffusions with killing in Subsection~\ref{sec:diff-with-killing}
and the particular case of one-dimensional diffusions is studied in Subsection~\ref{sec:diffusion-1d}.

\subsection{A general criterion in any dimension}
\label{sec:general-diffusion}

We consider a diffusion process $X$ on a connected, open domain $D\subset\R^d$ for some $d\geq 1$, solution to the SDE
\begin{align}
  \label{eq:SDE}
  \mathrm d X_t=b(X_t) \mathrm dt+\sigma(X_t) \mathrm dB_t,
\end{align}
where $B$ is a standard, $r$-dimensional Brownian motion and $b:D\rightarrow\R^d$ and $\sigma:D\rightarrow\R^{d\times r}$ are locally
H\"older functions, such that $\sigma$ is locally uniformly elliptic in $D$, i.e.\
$$
\forall K\subset D\text{ compact,}\quad\inf_{x\in K}\inf_{s\in\RR^d\setminus\{0\}}\frac{s^*\sigma(x)\sigma^*(x)s}{|s|^2}>0,
$$
where $|\cdot|$ is the standard Euclidean norm on $\RR^d$. We assume that the process is immediately absorbed at some cemetery point
$\d\not\in D$ at its first exit time of $D$, denoted $\tau_\d$. The existence and basic properties of this process need some care
since the coefficients $b$ and $\sigma$ are only defined in the open set $D$ without any assumption on the boundary
  of $D$, and so may not be possible to extend as continuous functions out of this set. Details are given in
Subsection~\ref{sec:construction-diff-multidim}. For the moment, let us only observe that, for all $k\geq 1$, defining the compact
set
\begin{align*}
  K_k=\left\{x\in D:|x|\leq k\text{ and }d(x,D^c)\geq 1/k\right\},
\end{align*}
a weak solution to~\eqref{eq:SDE} can be constructed up to the first exit time $\tau_{K_k^c}$ of $K_k$ as defined in~\eqref{eq:def-T_L}.
The proper definition of the absorption time $\tau_\d$ is
\begin{align}
  \label{eq:def-tau-d-diffusion-general}
  \tau_\d=\sup_{k\geq 1}\tau_{K_k^c}.
\end{align}

We introduce the differential operator associated to the SDE~\eqref{eq:SDE}, related to the infinitesimal generator of the process
$X$: for all $f\in\mathcal{C}^2(D)$, we define for all $x\in D$
\begin{align}
  \label{eq:def-gene-diff}
  \cL f(x):=\sum_{i=1}^d b_i(x)\frac{\partial f}{\partial x_i}(x)+\frac{1}{2}\sum_{i,j=1}^d\sum_{k=1}^r
  \sigma_{ik}(x)\sigma_{jk}(x)\frac{\partial^2 f}{\partial x_i\partial x_j}(x).
\end{align}
We also define the constant
\begin{align}
  \label{eq:def-lambda_0}
  \lambda_0:=\inf\left\{\lambda>0,\text{ s.t. }\liminf_{t\rightarrow+\infty}e^{\lambda t}\,\P_{x}\left(X_t\in B\right)>0\right\}
\end{align}
for some $x\in D$ and some open ball $B$ such that $\overline{B}\subset D$. It is standard to prove using Harnack inequalities (proved in
our case in Section~\ref{sec:harnack-diff}) that, under the previous assumptions, $\lambda_0<+\infty$ and its value is independent of
the choice of $x\in D$ and of the non-empty, open ball $B$ such that $\overline{B}\subset D$.

The following result is proved in Sections~\ref{sec:construction-diff-multidim} to~\ref{sec:proof-diff-multidim}. 

 \begin{thm}
    \label{thm:main-diffusion}
    Assume that there exist some constants $C>0$, $\lambda_1>\lambda_0$, a $\mathcal{C}^2(D)$ function $\varphi: D\rightarrow
    [1,+\infty)$ and a subset $D_0\subset D$ closed in $D$ such that $\sup_{x\in D_0}\varphi(x)<+\infty$ and
   \begin{align}
     \label{eq:lyapunov-diffusions}
     \cL \varphi(x) \leq -\lambda_1 \varphi(x)+C\11_{x\in D_0},\ \forall x\in D.
   \end{align}
   Assume also that there exists a time $s_1>0$ such that
  \begin{align}
  \label{eq:conv-to-0}
 \sup_{x\in D_0} \P_x(s_1<\tau_{K_k}\wedge \tau_\d)\xrightarrow[k\rightarrow\infty]{} 0.
  \end{align}
  Then $X$ admits a quasi-stationary distribution $\nu_{QSD}$ which satisfies $\nu_{QSD}(\varphi^{1/p})<+\infty$ for all $p>1$.
  Moreover, for all $p\in (1,\lambda_1/\lambda_0)$, there exist a constant $\alpha_p\in(0,1)$, a constant $C_p$ and a % positive
  function $\varphi_{2,p}:D\rightarrow (0,+\infty)$ uniformly bounded away from $0$ on compact subsets of $D$ such that, for all
  probability measures $\mu$ on $E$ satisfying $\mu(\varphi^{1/p})<\infty$,
  \begin{align*}
    \left\|\P_\mu(X_{t}\in\cdot\mid t<\tau_\d) -\nu_{QSD}\right\|_{TV(\varphi^{1/p})}
    &\leq C_p\alpha_p^t \frac{\mu(\varphi^{1/p})}{\mu(\varphi_{2,p})},\ \forall t\in [0,+\infty).
  \end{align*}
  In particular, $\nu_{QSD}$ is the only quasi-stationary distribution of $X$ which satisfies $\nu_{QSD}(\varphi^{1/p})<+\infty$ for at least one
  value of $p\in (1,\lambda_1/\lambda_0)$.
\end{thm}

\begin{rem}
  Note that $\tau_{K_k}=0$ $\P_x$-a.s.\ for all $x\in K_k$, thus
  \[
    \sup_{x\in D_0} \P_x(s_1<\tau_{K_k}\wedge \tau_\d)=\sup_{x\in
      D_0\setminus K_k} \P_x(s_1<\tau_{K_k}\wedge \tau_\d).
  \]
  Hence Condition~\eqref{eq:conv-to-0} requires the process to be absorbed
  or return in $K_k$ fast starting in $D_0\setminus K_k$.
\end{rem}

\begin{rem}
  \label{rem:(F)-diffusion}
  We shall actually prove that, under the conditions of the previous theorem, Assumption~(F) is satisfied with $L=K_k$ for some
  $k\geq 1$, and $\psi_1=\varphi^{1/p}$, for any $p\in (1,\lambda_1/\lambda_0)$.
\end{rem}

\begin{rem}
  \label{rem:non-explosion}
  In general, the assumptions of Theorem~4.1 do not ensure the non-explosion of the Markov process $X$. In the case of an explosive
  Markov process, the definition of $\tau_\d$ in~\eqref{eq:def-tau-d-diffusion-general} implies that, in the event of an explosion, the
  absorption time $\tau_\d$ is defined as equal to the explosion time.
\end{rem}

The last result has other consequences of interest, gathered in the next corollary, proved in Section~\ref{sec:pf-cor-diff}.
\begin{cor}
  \label{cor:main-diffusion}
  Under the assumptions of Theorem~\ref{thm:main-diffusion}, the infimum defining the constant $\lambda_0$ in~\eqref{eq:def-lambda_0}
  is actually a minimum and it satisfies $\P_{\nu_{QSD}}(t<\tau_\d)=e^{-\lambda_0 t}$ for all $t\geq 0$. In addition, the function
  $\eta$ of Theorem~\ref{thm:E-F} satisfies $P_t\eta=e^{-\lambda_0 t}\eta$ for all $t\geq 0$. In particular, $\eta$ belongs to the
  domain of the infinitesimal generator of the semigroup of the process $X$ defined as acting on the Banach space
  $L^\infty(\varphi_1)$, and it is an eigenfunction for the eigenvalue $-\lambda_0$. In addition, $\eta\in\mathcal{C}^2(D)$ and
  $\mathcal{L}\eta(x)=-\lambda_0\eta(x)$ for all $x\in D$.
\end{cor}

\subsection{Application to uniformly elliptic diffusion processes}
\label{sec:appl-diff-unif-ellipt}

We consider the case where $\sigma$ can be extended to $\RR^d$ as a locally uniformly elliptic matrix-valued function. In the
following corollary, we give a general situation where~\eqref{eq:conv-to-0} holds true. We emphasize that, contrary to previous
results on existence of quasi-stationary distributions for diffusions in a domain (see
e.g.~\cite{Pinsky1985,GongQianEtAl1988,KnoblochPartzsch2010,DelVillemonais2018,ChampagnatCoulibaly-PasquierEtAl2016}), no regularity
on the boundary of $D$ is required.

\begin{cor}
  \label{cor:diff-unbounded-domain}
  Let $D$ be an open connected subset of $\R^d$, $d\geq 1$. Let $X$ be solution to the SDE
  \begin{align}
    \label{eq:SDE-unif-ellipt}
    \mathrm dX_t=b(X_t)\mathrm dt+\sigma(X_t)\mathrm dB_t,\ t<\tau_\d,
  \end{align}
  where $b:\R^d\rightarrow \R^d$ and $\sigma:\R^d\rightarrow\R^{d\times r}$ are locally H\"older continuous in $\R^d$ and $\sigma$ is
  locally uniformly elliptic on $\R^d$. Recall the definition~\eqref{eq:def-lambda_0} of $\lambda_0$ and assume that there exist
  constants $C>0$, $\lambda_1>\lambda_0$, a $\mathcal{C}^2(D)$ function $\varphi: D\rightarrow [1,+\infty)$ and a bounded
  subset $D_0\subset D$ closed in $D$ such that
   \begin{align}
     \label{eq:lyapunov-diff-unbounded}
     \cL \varphi(x) \leq -\lambda_1 \varphi(x)+C\11_{x\in D_0},\ \forall x\in D.
   \end{align}
   Then the process $X$ absorbed at the boundary of $D$ (in the sense of~\eqref{eq:def-tau-d-diffusion-general}) satisfies the assumptions of Theorem~\ref{thm:main-diffusion}.
\end{cor}

Note that we do not assume that $\varphi(x)\to+\infty$ when $|x|\to+\infty$, hence the process $X$ may be explosive (see
Remark~\ref{rem:non-explosion}).

\begin{proof}
  Let us consider the diffusion process $Y$ solution to~\eqref{eq:SDE-unif-ellipt} on $\RR^d$. Due to our regularity assumptions on $b$ and
  $\sigma$, this process is well-defined up to a possibly finite explosion time $\tau_{\text{expl}}$. The Harnack
  inequality~\eqref{eq:Harnack-u} applied to $Y$ on the compact set $\overline{D}_0$ ensures the existence of constants $\delta>0$
  and $N$ such that, for all $f:\RR^d\rightarrow [0,1]$, for all $x\in\overline{D}_0$ and all $y\in B(x,\delta)$,
  \begin{align*}
    \EE_x[\11_{\delta+\delta^2<\tau_{\text{expl}}}f(Y_{\delta+\delta^2})]\leq N \EE_y[\11_{\delta+2\delta^2<\tau_{\text{expl}}}f(Y_{\delta+2\delta^2})].
  \end{align*}
  By compactness of $\overline{D}_0$, there exist a positive integer $n$ and $y_1,\ldots,y_n\in D_0$ such that
  $\overline{D}_0\subset\bigcup_{i=1}^n B(y_i,\delta)$. Setting $s_1=\delta+\delta^2$, we deduce that, for all $k\geq 1$ and all
  $x\in D_0$,
  \begin{align*}
    \PP_x(Y_{s_1}\in D\setminus K_k)\leq N\max_{1\leq i\leq n}\PP_{y_i}(Y_{s_1+\delta^2}\in D\setminus K_k)\xrightarrow[k\rightarrow+\infty]{}0.
  \end{align*}
  Hence~\eqref{eq:conv-to-0} is satisfied. This and Theorem~\ref{thm:main-diffusion} end the proof of
  Corollary~\ref{cor:diff-unbounded-domain}.
\end{proof}

We give three examples of application.

\begin{exa}
  \label{exa:D-bounded}
  Assume that $D$ is bounded. Then, one can choose $D_0=D$ and $\varphi=1$ in Corollary~\ref{cor:diff-unbounded-domain}. This
  implies that Assumption~(F) is satisfied for $\psi_1=\varphi^p$ bounded (see Remark~\ref{rem:(F)-diffusion}), so that it follows from
  Theorem~\ref{thm:E-F} that the convergence of $e^{\lambda_0 t}\mathbb{P}_X(t<\tau_\d)$ to $\eta$ is uniform and that $\eta$ is
  bounded. Theorem~\ref{thm:E-F} also implies that Assumption~(E) is satisfied for some bounded $\varphi_1$ and $\varphi_2$. Since
  $P_1\eta=e^{-\lambda_0}\eta$ and $e^{-\lambda_0}\geq\theta_2>\theta_1$, we deduce that (E) is still satisfied if $\varphi_2$ is
  replaced by $\eta/\|\eta\|_\infty$. Therefore,
  \[
  \left\|\P_\mu(X_{n}\in\cdot\mid n<\tau_\d) -\nu_{QSD}\right\|_{TV}
  \leq \frac{C}{\mu(\eta)}\,\alpha^n,\ \forall n\in\mathbb{N}.
  \]
  The extension to any $t\in[0,+\infty)$ can be obtained using the same argument as in Section~\ref{sec:full-QSD} replacing
  $\varphi_2$ and $\varphi'_2$ with $\eta$. This implies Theorem~\ref{thm:intro-diffusion} of the introduction.
\end{exa}

\begin{exa}
  \label{exa:any-sigma}
  Assume that $D\subset\R_+^d$ is open connected and that
  \begin{align*}
    \mathrm d X_t=b(X_t) \mathrm dt+\sigma(X_t)\mathrm dB_t
  \end{align*}
  in $D$, where $b:\R^d\rightarrow \R^d$ and $\sigma:\R^d\rightarrow\R^{d\times r}$ are locally H\"older continuous in $\R^d$,
  $\sigma$ is locally uniformly elliptic on $\R^d$ and
  \begin{align*}
    \frac{\langle b(x),1\rangle}{\langle x,1\rangle}\xrightarrow[|x|\rightarrow+\infty]{}-\infty,
  \end{align*}
  where $\langle\cdot,\cdot\rangle$ is the standard Euclidean product in $\RR^d$ and $|\cdot|$ is the associated norm.
  Then~\eqref{eq:lyapunov-diff-unbounded} is satisfied for $\varphi(x)=1+x_1+\ldots+x_d$ and hence the process $X$ absorbed at the
  boundary of $D$ satisfies the assumptions of Theorem~\ref{thm:main-diffusion}.
\end{exa}

\begin{exa}
  \label{exa:drifted-Brownian}
  Assume that $D\subset\R^d$ is open connected and that
  \begin{align*}
    \mathrm d X_t=b(X_t) \mathrm dt+\mathrm dB_t
  \end{align*}
  in $D$, where $b:\R^d\rightarrow \R^d$ is locally H\"older continuous in $\R^d$ and
  \begin{align}
  \label{eq:lyap-cond-expli}
   \limsup_{|x|\rightarrow+\infty} \frac{\langle b(x),x\rangle}{|x|}< -\frac{3}{2}\sqrt{\lambda_0},
  \end{align}
  where $\langle\cdot,\cdot\rangle$ is the standard Euclidean product in $\RR^d$ and $\lambda_0$ is defined
  in~\eqref{eq:def-lambda_0}. Then the process $X$ absorbed at the boundary of $D$ satisfies the assumptions of
  Theorem~\ref{thm:main-diffusion}.

  Indeed, let us check that~\eqref{eq:lyapunov-diff-unbounded} is satisfied for $\varphi(x)=\exp(\sqrt{\lambda_0}|x|)$. One has, for
  all $x\neq 0$,
  \begin{align*}
    \cL\varphi(x)&=\sum_{i=1}^d\frac{e^{\sqrt{\lambda_0}|x|}}{2}\left(\frac{\sqrt{\lambda_0}}{|x|}-\frac{\sqrt{\lambda_0}x_i^2}{|x|^{3}}+\frac{\lambda_0
        x_i^2}{|x|^2}\right)+\sum_{i=1}^d e^{\sqrt{\lambda_0}|x|}\,\frac{\sqrt{\lambda_0}b_i(x)\,x_i}{|x|}\\
    &\leq \sqrt{\lambda_0}\varphi(x)\,\left(\frac{d-1}{2|x|}+\frac{\sqrt{\lambda_0}}{2}+\frac{\langle b(x),x\rangle}{|x|}\right)\\
    &\leq -(\lambda_0+\varepsilon) \varphi(x)
  \end{align*}
  for some $\varepsilon>0$ and for all $x$ such that $|x|$ is large enough. This implies~\eqref{eq:lyapunov-diff-unbounded}.

  To apply this criterion, it is necessary to obtain a priori bounds on $\lambda_0$. We will give some ideas about how to do so for
  one-dimensional diffusions in Section~\ref{sec:diffusion-1d}. In general, one can also use of course that~\eqref{eq:lyap-cond-expli} is
  implied by
  \begin{align*}
   \lim_{|x|\rightarrow+\infty} \frac{\langle b(x),x\rangle}{|x|}=-\infty.
  \end{align*}  
\end{exa}

\subsection{Non-uniformly elliptic diffusions: the Feller diffusion with competition}
\label{sec:non-unif-ellipt}

We provide an example where the diffusion matrix $\sigma$ cannot be extended out of $D$ as a locally uniformly elliptic matrix. This
example deals with Feller diffusions with competition and is motivated by models of population dynamics with $d$ species in
interaction, where absorption corresponds to the extinction of one of the
populations~\cite{CattiauxMeleard2010,ChampagnatVillemonais2017}.

Assume that $D=(0,\infty)^d$ and
\begin{align*}
  \mathrm d X^i_t & =\sqrt{\gamma_i X^i_t}\,\mathrm dB^i_t+X^i_t b_i(X_t)\,\mathrm dt,
\end{align*}
where $\gamma_i>0$ for all $1\leq i\leq d$, $B^1,\ldots,B^d$ are independent standard Brownian motions and $b_i$ are locally H\"older
in $(0,\infty)^d$ and locally bounded in $\RR_+^d$.

\begin{prop}
  \label{prop:Feller}
  Assume that there exist constants $c_0,c_1>0$ such that
  \begin{align*}
    \sum_{i=1}^d\frac{x_i b_i(x)}{\gamma_i}\leq c_0-c_1|x|,\quad\forall x\in (0,\infty)^d.
  \end{align*}
  Then the process $X$ absorbed at the boundary of $D$ satisfies the assumptions of Theorem~\ref{thm:main-diffusion}.
\end{prop}

Compared to the existing literature on multi-dimensional Feller diffusions \cite{CattiauxMeleard2010,ChampagnatVillemonais2017}, the
main novelty of this result is that it covers cases where the process does not come down from infinity, e.g.
$b_i(x)=r_i-\sum_{j=1}^d c_{ij}\frac{x_j}{1+x_j}$, for some positive constants $r_i$ and $c_{ij}$ such that $r_i<c_{ii}$ for all
$1\leq i\leq d$, and where $b$ does not derive from a potential (see for instance~\cite{CattiauxMeleard2010}, based on a spectral
theoretic approach). While our results on existence and convergence to quasi-stationary distributions are more general than those
of~\cite{CattiauxMeleard2010}, we do not recover finer results on the spectrum of the process, such as its discreteness.

\begin{proof}
Our aim is to prove that the assumptions of Theorem~\ref{thm:main-diffusion} hold true with $\varphi(x)=\exp(c(x_1/\gamma_1+\ldots+x_n/\gamma_n))$, where $c= c_1\,\min_i\gamma_i/\sqrt{d}$.

We have, for all $x\in D$,
\begin{align*}
\cL\varphi(x)&=\sum_{i=1}^d \left(\frac{x_i c^2}{2\gamma_i}+\frac{c x_i b_i(x)}{\gamma_i}\right)\varphi(x)\leq \left(c_0 c-\frac{c_1 c|x|}{2}\right)\varphi(x).
\end{align*}
Choosing $\lambda_1=\lambda_0+1$ and $D_0=\{x\in D,\text{ s.t. }|x|\leq (2c_0+2\lambda_1/c)/c_1\}$, one deduces that~\eqref{eq:lyapunov-diffusions} holds true with $C=c_0 c\,\max_{D_0}\varphi$.

Let us now prove that
\begin{align}
\label{eq:feller-diff-proof}
\P_x(1<\tau_\d)\xrightarrow[x\rightarrow \d D, x\in D_0]{} 0,
\end{align}
which implies that~\eqref{eq:conv-to-0} holds true with $s_1=1$. Fix $\varepsilon>0$ and define the set $F=\left\{x\in \R_+^d,\text{
    s.t. } \varphi(x)\geq e^{C}\sup_{y\in D_0}\varphi(y)/\varepsilon\right\}$. Using It\^o's formula (see the proof
of~\eqref{eq:Ito-Lyap} in Section~\ref{sec:proof-diff-multidim} for details), we deduce from~\eqref{eq:lyapunov-diffusions} that, For
all $x\in D_0$,
\[
\P_x(\tau_F\leq 1)\,e^{C}\sup_{y\in D_0}\varphi(y)/\varepsilon \leq \E_x\left(\varphi(X_{\tau_F\wedge 1})\11_{\tau_F\wedge 1<\tau_\d}\right)\leq e^C \varphi(x),
\] 
so that $\P_x(\tau_F\leq 1)\leq \varepsilon$ for all $x\in D_0$. Since $F^c$ is bounded, we have 
\[
\beta:=\sup_{x\in F^c,i\in\{1,\ldots,d\}} |b_i(x)|<+\infty.
\]
Let $(Z_t)_{t\in[0,+\infty)}:=(Z^1_t,\ldots,Z^d_t)_{t\in[0,+\infty)}$ be the solution of the system of SDEs
\begin{align*}
\mathrm dZ^i_t=\sqrt{\gamma_i Z^i_t}\,\mathrm dB^i_t+Z^i_t \beta\,\mathrm dt,\ Z^i_0=X^i_0\in (0,+\infty),
\end{align*}
with absorption at the boundary of $D$. Note that the components of $Z$ are independent one dimensional diffusion processes such that $0$ is reachable and hence that
\[
\P_x\left(\forall t\in [0,1],\,\forall i\in\{1,\ldots,d\},\ Z^i_t>0\right)\xrightarrow[x\rightarrow \d D]{} 0.
\]
Standard comparison arguments show that $X^i_t\leq Z^i_t$ for all $t<\tau_\d\wedge \tau_F\wedge 1$ and all $i\in \{1,\ldots,d\}$, so that
\[
\P_x\left(\forall t\in [0,1],\,\forall i\in\{1,\ldots,d\},\ X^i_t> 0\text{ and }1<\tau_F\right)\xrightarrow[x\rightarrow \d D]{} 0.
\]
But $\P_x(1<\tau_F)\geq 1-\varepsilon$, so that
\[
\limsup_{x\rightarrow \d D}\, \P_x\left(\forall t\in [0,1],\,\forall i\in\{1,\ldots,d\},\ X^i_t> 0\right)\leq \varepsilon.
\]
Since this is true for all $\varepsilon>0$ and since $\{\forall t\in [0,1],\,\forall i\in\{1,\ldots,d\},\ X^i_t> 0\}=\{1<\tau_\d\}$, we deduce that~\eqref{eq:feller-diff-proof} holds true, which concludes the proof or Proposition~\ref{prop:Feller}.
\end{proof}

\subsection{Diffusion processes with killing}

\label{sec:diff-with-killing}

This section is devoted to the study of diffusion processes with killing. More precisely, we consider as above a diffusion process $X$ on a connected, open domain $D\subset\R^d$ for some $d\geq 1$, solution to the SDE
\begin{align}
  \label{eq:SDE-new}
  \mathrm dX_t=b(X_t)\mathrm dt+\sigma(X_t)\mathrm dB_t
\end{align}
absorbed in $\d$ at its first exit time $\tau_{\text{exit}}$ of $D$, as defined in~\eqref{eq:def-tau-d-diffusion-general}, with the
same assumptions as in Section~\ref{sec:general-diffusion}. We also assume that the process is subject to an additional measurable killing
rate $\kappa:D\rightarrow\R_+$ which is locally bounded: there exists an independent exponential random variable $\xi$ with parameter
$1$ such that the process is instantaneously sent to the cemetery point $\d\notin D$ at time
\begin{align*}
\tau_\d=\tau_{\text{exit}} \wedge \inf\left\{t\geq 0,\, \int_0^t \kappa(X_s)\,\mathrm ds\,>\xi \right\}.
\end{align*}

Since $\kappa$ is assumed to be locally bounded, one easily checks that $\lambda_0$ in~\eqref{eq:def-lambda_0} is finite, and that it
does not depend on $x\in D$ or on the open ball $B$ such that $\overline{B}\subset D$.

The following result is an extension to the multi-dimensional setting of~\cite[Theorem~4.3]{KolbSteinsaltz2012}.

\begin{thm}
Assume that there exist a subset $D_0\subsetneq D$ closed in $D$ such that
\begin{align}
\label{eq:kappa-bigger-than-lambda0}
\inf_{x\in D\setminus D_0} \kappa(x) > \lambda_0,
\end{align}
and a time $s_1>0$ such that
\begin{align}
\label{eq:wk-conv-to-0}
\sup_{x\in D_0} \P_x(s_1<\tau_\d\wedge \tau_{K_k})\xrightarrow[k\rightarrow+\infty]{} 0.
\end{align}
Then the process $X$ absorbed at time $\tau_\d$ admits a unique quasi-stationary distribution $\nu_{QSD}$ and there exist a positive
function $\varphi_2$ on $D$ (uniformly bounded away from $0$ on compact subsets of $D$) and a positive constant $C$ such that
\begin{align*}
    \left\|\P_\mu(X_{t}\in\cdot\mid t<\tau_\d) -\nu_{QSD}\right\|_{TV}
    &\leq \frac{C}{\mu(\varphi_{2})}\,\alpha^t,\ \forall t\in [0,+\infty)
  \end{align*}
  for all probability measures $\mu$ on $E$.
\end{thm}

\begin{rem} Let us make some comments on the assumptions of the above result.
\begin{enumerate}
\item If the process without killing rate satisfies~\eqref{eq:wk-conv-to-0}, then the process with killing rate also satisfies this
  property. Hence the analysis provided in Section~\ref{sec:appl-diff-unif-ellipt} can also be used to check the assumptions of the
  above theorem.
\item If $\inf_{x\in D\setminus K_k}\kappa(x)\rightarrow +\infty$ when $k\rightarrow+\infty$, then the assumptions of Theorem~4.7 are
  trivially satisfied.
\item In order to reach the conclusion of Theorem~\ref{thm:main-diffusion} in the setting of killed diffusion, it is also possible to
  use a Lyapunov type criterion: the assumption~\eqref{eq:lyapunov-diffusions} can be simply replaced by the assumption that there
  exist $\lambda>\lambda_0$ and $C>0$ such that
\begin{align*}
\mathcal{L}\varphi(x)-\kappa(x)\varphi(x)\leq -\lambda \varphi(x)+C\11_{x\in D_0}.
\end{align*}
Note that~\eqref{eq:kappa-bigger-than-lambda0} of course implies the last inequality for $\varphi\equiv 1$. This extension follows
from a simple adaptation of the arguments of Theorem~\ref{thm:main-diffusion} observing that
\begin{align*}
  \EE_x\left[f(X_t)\11_{t<\tau_\d}\right]=\EE_x\left[f(X^D_t)\11_{t<\tau_{\text{exit}}}\exp\left(-\int_0^t \kappa(X^D_s) \mathrm ds\right) \right],
\end{align*}
where the process $X^D$ is the process solution to~\eqref{eq:SDE-new} without killing, absorbed at its first exit time of $D$, at
time $\tau_{\text{exit}}$.
\item If in addition the killing rate $\kappa$ is locally H\"older in $D$, we can apply~\cite[Cor.\,3.1]{Friedman1964} as in
  Section~\ref{sec:pf-cor-diff} to prove that $\eta$ is $\mathcal{C}^2(D)$ and
  $\mathcal{L}\eta(x)-\kappa(x)\eta(x)=-\lambda_0\eta(x)$ for all $x\in D$.
\end{enumerate}
\end{rem}

\begin{proof}
  The proof follows the same lines as the proof of Theorem~\ref{thm:main-diffusion} in Section~\ref{sec:pf-general-diffusion}. We
  emphasize that the construction of the process in Section~\ref{sec:construction-diff-multidim} is still valid. The same is true for
  the Harnack inequalities of Section~\ref{sec:harnack-diff} since they are based on Krylov's and Safonov's general
  result~\cite{KrylovSafonov1980} which is obtained for diffusion processes with a bounded and measurable killing rate. The rest of
  the proof is exactly the same, replacing $\varphi_1=\varphi$ by $\varphi_1=1$.
\end{proof}

\subsection{The case of one-dimensional diffusions}
\label{sec:diffusion-1d}

In this section, we consider the case of one-dimensional diffusion processes. Here, the H\"older regularity of the
coefficients is not needed. Let $X$ be the solution in $D=(\alpha,\beta)$, where $-\infty\leq\alpha<\beta\leq+\infty$, to the SDE
\begin{align*}
\mathrm dX_t=\sigma(X_t)\,\mathrm dB_t+b(X_t)\,\mathrm dt,\quad X_0\in D,
\end{align*}
where $\sigma:D\rightarrow(0,+\infty)$ and $b:D\rightarrow \R$ are measurable functions such that $(1+|b|)/\sigma^2$ is locally
integrable on $D$. We assume that the process is sent to a cemetery point $\d$ when it reaches the boundary of $D$ and that it is
subject to an additional killing rate $\kappa:D\rightarrow\R_+$ which is measurable and locally integrable w.r.t.\ Lebesgue's
measure. This assumption implies that the killed process is regular in the sense that, for all $x,y\in D$,
$\P_x(\tau_{\{y\}}<\infty)>0$.

We define $\lambda_0$ as in~\eqref{eq:def-lambda_0}. The fact that $\lambda_0$ does not depend on $x$ nor $B$ is a consequence of the
regularity of the process.

Let $\delta:D\rightarrow \R_+$ and $s:D\rightarrow\R$ be defined by
\begin{align*}
\delta(x)=\exp\left(-2\int_{\alpha_0}^x \frac{b(u)}{\sigma(u)^2}\,\mathrm du\right)\quad\text{and}\quad s(x)=\int_{\alpha_0}^x \delta(u)\,\mathrm du,
\end{align*}
for some arbitrary $\alpha_0\in D$.  We recall that $s$ is the scale function of $X$ (unique up to an affine
transformation), meaning that $s(X_t)$ is a local martingale. We also recall that the boundary $\alpha$ (and similarly
for $\beta$) is said to be reachable (for the process without killing) if $s(\alpha_+)>-\infty$ and
\begin{align*}
\int_\alpha^+ \frac{s(x)-s(\alpha_+)}{\sigma(x)^2\delta(x)}\,\mathrm dx\,<+\infty.
\end{align*}

\begin{thm}
\label{thm:diffusion-1d}
Assume that one among the following conditions \textnormal{(i)}, \textnormal{(ii)} or \textnormal{(iii)} holds true:
\begin{description}
\item[\textmd{(i)}] $\alpha$ and $\beta$ are reachable boundaries;
\item[\textmd{(ii)}] $\alpha$ is reachable and there exist $\lambda_1>\lambda_0$, a $\mathcal{C}^2(D)$ function $\varphi: D\rightarrow
    [1,+\infty)$ and $x_1\in D$ such that, for all $x\geq x_1$,
   \begin{align}
     \label{eq:1d-lyapunov-diffusions}
     \frac{\sigma(x)^2}{2} \varphi''(x) +b(x)\varphi'(x)-\kappa(x)\varphi(x) \leq -\lambda_1 \varphi(x);
   \end{align}
 \item[\textmd{(iii)}] there exist $\lambda_1>\lambda_0$, a $\mathcal{C}^2(D)$ function $\varphi: D\rightarrow [1,+\infty)$
   and $x_0<x_1\in D$ such that~\eqref{eq:1d-lyapunov-diffusions} holds true for all $x\in (\alpha,x_0)\cup (x_1,\beta)$.
\end{description}
   Then the conclusions of Theorem~\ref{thm:main-diffusion} hold true.
\end{thm}

\begin{rem}
  We shall not detail the proof of this result since it is very close to the proof of Theorem~\ref{thm:main-diffusion} given in
  Section~\ref{sec:pf-general-diffusion}. We only explain the places that need to be modified. First, weak existence, weak uniqueness
  and the strong Markov property are well-known under the assumptions that $\sigma>0$ and $(1+|b|)/\sigma^2\in
  L^1_{\textnormal{loc}}(D)$ (weak existence and uniqueness in law are proved up to an explosion time in~\cite[Thm.
  5.5.15]{KaratzasShreve1991}, so we can construct a unique weak solution and prove the strong Markov property as in
  Section~\ref{sec:construction-diff-multidim}). Second, in order to construct an appropriate function $\varphi$ on $D$, we choose
  $D_0=(\alpha,x_1]$ in case (ii) and $D_0=[x_0,x_1]$ in case (iii) and we can extend $\varphi$ on $D_0$ as a bounded
  $\mathcal{C}^2(D)$ function. In case (i), we can take $\varphi\equiv 1$ and $D_0=D$. Third,~\eqref{eq:conv-to-0} follows from the
  fact that the boundaries $\alpha$ and $\beta$ are reachable in case (i) and $\alpha$ is reachable in case (ii), since
  \[
  \sup_{x\in (\alpha,\alpha+1/k]}\P_x(s_1<\tau_\d)\leq \P_{\alpha+1/k}(s_1<\tau_{\{\alpha\}})\xrightarrow[k\rightarrow+\infty]{} 0.
  \]
  In case (iii), the limit is trivial since $D_0\subset K_k$ for $k$ large enough. Finally, all the arguments using Harnack's
  inequality can be replaced by arguments using the regularity of the process and standard coupling arguments for one-dimensional
  diffusions (see~\cite{ChampagnatVillemonais2015,ChampagnatVillemonais2017ALEA}).
\end{rem}

In order to apply this result in practice, one needs to find computable estimates for $\lambda_0$ and candidates for $\varphi$. One
may for instance use the bounds for the first eigenvalue of the (Dirichlet) infinitesimal generator of $(X_t,t\geq 0)$ obtained in a
$L^2$ (symmetric) setting using Rayleigh-Ritz formula in~\cite{Pinsky2009,Wang2009,Wang2012}, as observed
in~\cite{KolbSteinsaltz2012}. We propose here two different upper bounds for $\lambda_0$ which follow from the
characterization~\eqref{eq:def-lambda_0} of the eigenvalue $\lambda_0$ and Dynkin's formula. %......

\begin{prop}
\label{prop:lambda0-bound}
For all $\alpha< \mathfrak{a}<\mathfrak{b} <\beta$, we have
\begin{align*}
\lambda_0 \leq \sup_{x\in[\mathfrak{a},\mathfrak{b}]}\left\{\frac{1}{2}\left(\frac{\pi\sigma(x)}{\int_\mathfrak{a}^\mathfrak{b}
      \exp\left(-2\int_x^y\,\frac{b(z)}{\sigma^2(z)}\,\mathrm dz\right)\,\mathrm dy}\right)^2+\kappa(x)\right\}. 
\end{align*}
If $x\mapsto b(x)/\sigma(x)^2$ is $\mathcal{C}^1([\mathfrak{a},\mathfrak{b}])$, then
\begin{align*}
\lambda_0\leq \sup_{x\in[\mathfrak{a},\mathfrak{b}]} \frac{\pi^2\sigma(x)^2}{2(\mathfrak{b}-\mathfrak{a})^2}+\sigma(x)^2\left(\frac{b}{2\sigma^2}\right)'(x)+\frac{b(x)^2}{2\sigma(x)^2}+\kappa(x).
\end{align*}
\end{prop}

\begin{proof}
For the proof of the first inequality, set
\begin{align*}
f (x)=\sin\left(\pi\frac{s(x)-s(\mathfrak{a})}{s(\mathfrak{b})-s(\mathfrak{a})}\right).
\end{align*}
Then, for all $x\in(\mathfrak{a},\mathfrak{b})$,
\begin{align*}
\frac{\sigma(x)^2}{2} f ''(x) +b(x)f '(x) & -\kappa(x)f (x)
= -\left(\frac{\pi^2\sigma(x)^2\delta(x)^2}{2(s(\mathfrak{b})-s(\mathfrak{a}))^2}+\kappa(x)\right)f (x)\\
&= -\left(\frac{\pi^2\sigma(x)^2}{2\left(\int_\mathfrak{a}^\mathfrak{b} \exp\left(-2\int_x^y\,\frac{b(z)}{\sigma^2(z)}\,\mathrm dz\right)\,\mathrm dy\right)^2}+\kappa(x)\right)f (x)\\
&\geq -Cf (x),
\end{align*}
where 
\[
C:=\sup_{x\in[\mathfrak{a},\mathfrak{b}]}\left\{\frac{1}{2}\left(\frac{\pi\sigma(x)}{\int_\mathfrak{a}^\mathfrak{b}
    \exp\left(-2\int_x^y\,\frac{b(z)}{\sigma^2(z)}\,\mathrm dz\right)\,\mathrm dy}\right)^2+\kappa(x)\right\}.
\]
Since $f $ is $C^2$ and bounded, we deduce from It\^o's formula that, for all $x\in(\mathfrak{a},\mathfrak{b})$,
\begin{align*}
\E_x(f (X_t)\11_{t<\tau_{\{\mathfrak{a},\mathfrak{b}\}}})\geq e^{-Ct}f (x).
\end{align*}
Now, using the fact that $0<f (x)\leq 1$ for all $x\in(\mathfrak{a},\mathfrak{b})$,
we deduce that
\begin{align*}
\P_x(X_t\in (\mathfrak{a},\mathfrak{b}))\geq e^{-Ct}f (x),\ \forall x\in D.
\end{align*}
As a consequence, the definition of $\lambda_0$ entails $\lambda_0\leq C$.

The proof of the second inequality is the same, using instead the function
\begin{align*}
f (x):=\exp\left(-\int_{\mathfrak{c}}^x\frac{b(u)}{\sigma(u)^2}\,\mathrm du\right)\,\sin\left(\pi\frac{x-\mathfrak{a}}{\mathfrak{b}-\mathfrak{a}}\right)
\end{align*}
for some $\mathfrak{c}\in(\mathfrak{a},\mathfrak{b})$.
\end{proof}

The next result provides two candidates for $\varphi$. Its proof is a straightforward computation.

\begin{prop}
  \label{prop:varphi-candidate}
  Let $\varphi:(0,+\infty)$ be any $\mathcal{C}^2(D)$ function such that, for some constants $\alpha_-<\alpha_0<\alpha_+\in D$,
  \begin{equation}
    \label{eq:varphi-candidate-1}
    \varphi(x)=
    \begin{cases}
      \sqrt{s(x)}&\text{ if }x\geq \alpha_+,\\
      \sqrt{-s(x)}&\text{ if }x\leq \alpha_-.
    \end{cases}
  \end{equation}
  Then, for all $x\in(\alpha,\alpha_-]\cup[\alpha_+,\beta)$
  \begin{align*}
    \frac{\sigma(x)^2}{2} \varphi''(x) +b(x)\varphi'(x)-\kappa(x)\varphi(x) \leq -\left(\frac{\sigma(x)^2\delta(x)^2}{8s(x)^2}+\kappa(x)\right)\,\varphi(x).
  \end{align*}
  If $x\mapsto b(x)/\sigma(x)^2$ is $C^1(D)$, then
  \begin{equation}
    \label{eq:varphi-candidate-2}
    \varphi(x)= \exp\left(-\int_{\alpha_0}^x\frac{b(u)}{\sigma^2(u)}\,\mathrm du\right)
  \end{equation}
  satisfies
  \begin{align*}
    \frac{\sigma(x)^2}{2} \varphi''(x) +b(x)\varphi'(x)-\kappa(x)\varphi(x)= -\left(\frac{b^2(x)}{2\sigma^2(x)}+\frac{\sigma^2(x)}{2}\left(\frac{b}{\sigma^2}\right)'(x)+\kappa(x)\right)\varphi(x).
  \end{align*}
\end{prop}

\begin{rem}
The first function $\varphi$ is always uniformly lower bounded on $(\alpha,\alpha_-]\cup[\alpha_+,\beta)$ by
$\min\{\sqrt{s(\alpha_+)},\sqrt{-s(\alpha_-)}\}$. To ensure that the second one is also uniformly lower bounded, one needs further
assumptions on the behavior of $b/\sigma^2$ close to $\alpha$ and $\beta$.
\end{rem}

The above results can be used as follows. In the case where $\alpha$ is reachable and $b\equiv 0$, Condition (ii) of Theorem~\ref{thm:diffusion-1d} holds true if
\[
\liminf_{x\rightarrow \beta-}\frac{\sigma^2(x)}{8(x-\alpha)^2}+\kappa(x)>\lambda_0,
\]
choosing $\alpha_0=\alpha$ and using the function $\varphi$ of~\eqref{eq:varphi-candidate-1}. Similarly, in the case where $\alpha$
is reachable, $\sigma\equiv 1$ and $b$ is $C^1$, condition (ii) of Theorem~\ref{thm:diffusion-1d} holds true if
\[
\liminf_{x\rightarrow \beta-}\frac{b^2(x)}{2}+\frac{b'(x)}{2}+\kappa(x)>\lambda_0,
\] 
using the function $\varphi$ of~\eqref{eq:varphi-candidate-2}.

We give below more precise examples.

\begin{exa}
Assume that $D=(0,+\infty)$, $\kappa$ is locally bounded and that $X$ is solution to the SDE in $D$
\begin{align*}
\mathrm dX_t=\sqrt{X_t}\mathrm dB_t-X_t \mathrm dt.
\end{align*}
Then $0$ is reachable for $X$ and, since 
\[
\frac{\sigma(x)^2\delta(x)^2}{8s(x)^2}\xrightarrow[x\rightarrow+\infty]{} +\infty,
\]
we deduce from Proposition~\ref{prop:varphi-candidate} and Theorem~\ref{thm:diffusion-1d} that $X$ admits a quasi-statio\-na\-ry
distribution $\nu_{QSD}$ and, for all $p\geq 1$, there exist positive constants $C_p,\gamma_p$ and a positive function $\varphi_{2,p}$ on
$(0,+\infty)$ such that
\begin{align*}
\left\|\P_\mu(X_t\in\cdot\mid t<\tau_\d)-\nu_{QSD}\right\|_{TV(\exp(\cdot/p))}\leq C_p\frac{\int_{(0,+\infty)}
  \exp(x/p)\,\mu(\mathrm dx)}{\mu(\varphi_{2,p})}\,e^{-\gamma_p t}, 
\end{align*}
for all probability measure $\mu$ on $D$. In particular, one deduces that the domain of attraction $\nu_{QSD}$ contains any initial
distribution $\mu$ admitting a finite exponential moment. Note that, in the case where $\kappa\equiv 0$, the process $X$ is a
continuous state branching process (Feller diffusion), for which quasi-stationarity was already studied (see~\cite{Lambert2007} and
the references therein).
\end{exa}

\begin{exa}
Assume that $(\alpha,\beta)=\R$, that $b\equiv 0$ and $\sigma$ is bounded measurable on $\RR$. Assume also that the absorption of $X$
is due to the killing rate $\kappa(x)=\kappa_0\left(1-\frac{1}{1+|x|}\right)$ for some constant $\kappa_0>0$. We deduce from the
first inequality of Proposition~\ref{prop:lambda0-bound} (taking $\mathfrak{b}>0$ and $\mathfrak{a}=-\mathfrak{b}$) that
 \[
 \lambda_0\leq \frac{\pi^2\|\sigma\|_\infty^2}{8\mathfrak{b}^2}+\kappa_0\left(1-\frac{1}{1+\mathfrak{b}}\right)\leq \kappa_0\left(1-\frac{1}{1+2\mathfrak{b}}\right)
 \]
 for $\mathfrak{b}$ large enough.
 Moreover, choosing $\varphi=1$ and $x_0=-3\mathfrak{b}$, $x_1=3\mathfrak{b}$, one deduces that, for all $x\not\in[-x_1,x_1]$,
\[
\frac{\sigma(x)^2}{2} \varphi''(x) -\kappa(x)\varphi(x) \leq -\kappa_0\left(1-\frac{1}{1+3\mathfrak{b}}\right) \varphi(x).
\]
Hence Theorem~\ref{thm:diffusion-1d} implies that there exists a unique quasi-stationary distribution $\nu_{QSD}$ for $X$ and that it
attracts all probability measures $\mu$ on $D$. 
\end{exa}

\begin{exa}
  We consider the case $(\alpha,\beta)=(0,+\infty)$, $\sigma(x)=1$, $b(x)=x\sin x$, and
  $\kappa(x)=\kappa_0\left(1-\frac{1}{1+x}\right)$ for some constant $\kappa_0>\pi^2+3$. This corresponds to a SDE $\mathrm dX_t=\mathrm dB_t+\nabla
  U(X_t) \mathrm dt$ where the potential $U(x)=\sin x-x\cos x$ has infinitely many wells with arbitrarily large depths, meaning that the process
  $X$ without killing has a tendency to be ``trapped'' away from zero for large initial conditions. Nevertheless, thanks to the
  killing, we are able to prove convergence to a unique quasi-stationary distribution. Indeed, using the second inequality of
  Proposition~\ref{prop:lambda0-bound}, we have
  \begin{align*}
    \lambda_0 &\leq \sup_{x\in(0,1)} \frac{\pi^2}{2}+\frac{\sin x+x\cos x+x^2\sin^2 x}{2}+\kappa_0\left(1-\frac{1}{1+x}\right)
    \leq \frac{\pi^2}{2}+\frac{3}{2}+\kappa_0/2.
  \end{align*}
  Moreover, $0$ is a reachable boundary for $X$ and, taking $\varphi=1$, one has, for all $x_1>0$ and all $x>x_1$,
  \[
  \frac{\sigma(x)^2}{2} \varphi''(x)+b(x)\varphi'(x) -\kappa(x)\varphi(x) \leq -\kappa_0\left(1-\frac{1}{1+x_1}\right) \varphi(x)
  \]
  Hence, since we assumed that $\kappa_0>\pi^2+3$, one deduces that there exists a unique quasi-stationary distribution
  $\nu_{QSD}$ for $X$ and that it attracts all probability measures $\mu$ on $D$.
\end{exa}

\begin{rem}
  The case of general one-dimensional diffusion processes~\cite{Kallenberg2002} can be handled using our framework, although using
  the infinitesimal generator is more tricky~\cite{ItoMcKean1974}. However, in the case of a regular diffusion process on
  $(0,+\infty)$ such that $0$ is a reachable boundary and such that $+\infty$ is entrance, one easily shows (see for
  instance~\cite{ChampagnatVillemonais2015}) that, for all $\lambda>0$, there exists $y>0$ such that
  \begin{align*}
    \sup_{x\in(0,+\infty)}\E_x\left(e^{\lambda \tau_{[0,y]}}\right)<+\infty.
  \end{align*}
  Hence, using the same proof as in Theorem~\ref{thm:main-diffusion} and using Lemma~\ref{prop:E-G}, we deduce that there
  exists a unique quasi-stationary distribution $\nu_{QSD}$ for $X$ and that it satisfies
  \[
  \left\|\P_\mu(X_{t}\in\cdot\mid t<\tau_\d) -\nu_{QSD}\right\|_{TV}
  \leq \frac{1}{\mu(\varphi_{2})}\,\alpha^t,\ \forall t\in [0,+\infty)
   \]
   for some positive function $\varphi_2$ and some $\alpha<1$. Whether the convergence to $\nu_{QSD}$ holds uniformly with respect to
   the initial distribution (as in Proposition~\ref{prop:A1-A2}) without further assumptions remains an open problem. It has been
   shown to be true for a wide range of cases in~\cite{ChampagnatVillemonais2015,ChampagnatVillemonais2017ALEA}.
\end{rem}

\section{Application to processes in discrete state space and continuous time}
\label{sec:appl-discr-state}

Let $X$ be a non-explosive\footnote{One could actually consider the case of explosive Markov processes as in
  Section~\ref{sec:diffusion} (see Remark~\ref{rem:non-explosion}), with $\tau_\d$ defined as the infimum between the first hitting
  time of $\d$ and the explosion time.} Markov process in a countable state space $E\cup\{\d\}$ absorbed in $\d$, with jump rate 
$q_{x,y}$ from $x$ to $y\neq x$ such that $\sum_{y\in E\cup\{\d\}\setminus\{x\}}q_{x,y}<\infty$ for all $x\in E$. The
extended generator $\cL$ acts on nonnegative real functions $f$ on $E\cup\{\d\}$ such that $\sum_{y\in
  E\cup\{\d\}}q_{x,y}f(y)<\infty$ for all $x\in E$ as
\begin{equation}
  \label{eq:def-L-discret}
  \cL f(x)=\sum_{y\neq x\in E\cup\{\d\}}q_{x,y}(f(y)-f(x)),\quad\forall x\in E,\quad Lf(\d)=0.
\end{equation}

\begin{thm}
  \label{thm:main-E-denumerable}
  Assume that there exists a finite subset $D_0$ of $E$ such that $\P_x(X_1=y)>0$ for all $x,y\in D_0$, so that the constant
  \begin{align*}
    \lambda_0:=\inf\left\{\lambda>0,\text{ s.t. }\liminf_{t\rightarrow+\infty}e^{\lambda t}\,\P_{x}\left(X_t=x\right)>0\right\}
  \end{align*}
  is finite and independent of $x\in D_0$. If in addition there exist constants $C>0$,
  $\lambda_1>\lambda_0$, a function $\varphi: E\cup\{\d\}\rightarrow \R_+$ such that
  $\restriction{\varphi}{E}\geq 1$, $\varphi(\d)=0$, $\sum_{y\in E\setminus\{x\}}q_{x,y}\varphi(y)<\infty$ for all
  $x\in E$ and such that
  \begin{align}
    \label{eq:discr-cont-lyap}
    \cL\varphi(x) \leq -\lambda_1 \varphi(x)+C\11_{x\in D_0},\ \forall x\in E,
  \end{align}
  then Assumption~(F) is satisfied with $L=D_0$, $\gamma_1=e^{-\lambda_1}$, any $\gamma_2\in(e^{-\lambda_1},e^{-\lambda_0})$ and
  $\psi_1=\restriction{\varphi}{E}$. {In addition, $\P_{\nu_{QSD}}(t<\tau_\d)=e^{-\lambda_0 t}$ for all $t\geq 0$, the
    function $\eta$ of Theorem~\ref{thm:eta} satisfies $P_t\eta=e^{-\lambda_0 t}\eta$ for all $t\geq 0$ and $\sum_{y\in
      E\setminus\{x\}}q_{x,y}\eta(y)<\infty$ and $\mathcal{L}\eta(x)=-\lambda_0\eta(x)$ for all $x\in E$.}
\end{thm}

\begin{rem}
  If in addition to the assumptions of Theorem~\ref{thm:main-E-denumerable} we assume that $\lambda_1>\sup_{x\in E} q(x,\d)$, it is
  possible to adapt the proof of Theorem~\ref{thm:E-F} given in Section~\ref{sec:proof-E-F} to prove that the conclusion of
  Theorem~\ref{thm:E-F} holds true with $\psi_2\equiv 1$. Therefore, we obtain the improved convergence, for all $h\in L^\infty(\varphi)$,
  \[
  \left|\E_\mu(h(X_{t})\mid t<\tau_\d) -\nu_{QSD}(h)\right|
  \leq C\,\mu (\varphi)\,\alpha^t\,\|h/\varphi\|_\infty,\ \forall t\geq 0,
  \]
  instead of~\eqref{eq:expo-cv-prop-E-F}. If moreover $\varphi$ is bounded over $E$, the convergence is uniform and there exists a
  unique quasi-stationary distribution.
\end{rem}

Before turning to the proof of Theorem~\ref{thm:main-E-denumerable}, we give an example of application.

\begin{exa}
	Assume that $X$ is a birth and death process with killing on $E=\NN$ and $\d=0$. This means that there exist non-negative numbers $(b_x)_{x\in\N}$, $(d_x)_{x\in\N}$, $(\kappa_x)_{x\in\N}$ such that $b_x>0$ for all $x\geq 1$, $d_x>0$ for all $x\geq 2$, and $d_1=0$, and such that, for all $x\in E$,
	\[
		q_{x,y}=\begin{cases}
		b_x&\text{ if }y=x+1,\\
		d_x&\text{ if }y=x-1,\\
		\kappa_x&\text{ if }y=0,\\
		0&\text{ otherwise.}
		\end{cases}
	\]
	We set
	\begin{align}
	S:=\sum_{k\geq 1}\frac{1}{d_k\alpha_k}\sum_{l\geq k} \alpha_l,
	\end{align}
	with $\alpha_k=\left(\prod_{i=1}^{k-1} b_i\right)/\left(\prod_{i=1}^{k} d_i\right)$. Recent advances on existence of quasi-stationary distribution of birth and death processes with killing were obtained in~\cite{Coolen-SchrijnerDoorn2006,DoornErik2012,Doorn2013}, see also the nice survey~\cite{DoornPollett2013}.
	
In this setting, Theorems~\ref{thm:main-E-denumerable} and~\ref{thm:E-F}  translate as follows: if there exists a function
$\varphi:\Z_+\to[1,+\infty)$ such that $\varphi(0)=0$ and
\begin{align}
  \label{eq:critere-general-BDK}
  \lambda_0<\liminf_{x\rightarrow+\infty}\,-\frac{b_x (\varphi(x+1)-\varphi(x))+d_x (\varphi(x-1)-\varphi(x))}{\varphi(x)}+\kappa_x,
\end{align}
then there exists a unique quasi-stationary distribution $\nu_{QSD}$ such that $\nu_{QSD}(\varphi)<+\infty$ which attracts
exponentially fast all initial distributions integrating $\varphi$. To check~\eqref{eq:critere-general-BDK}, one may use in practice
the fact that $\lambda_0\leq \inf_{x\in\N} b_x+d_x+\kappa_x$, or adapt the ideas of Section \ref{sec:diffusion-1d} to birth and death
processes, or use the finer upper bounds for $\lambda_0$ proved in~\cite{DoornZeifman2005}. We consider now three situations where
the criterion~\eqref{eq:critere-general-BDK} improves known results in the literature.
	
	% {\color{red} Ici, on peut utiliser considérer l'ensemble des $\varphi_2$ à support compact et majorer $\lambda_0$ par $\inf_{\varphi_2} \sup_{
            % x\in E}-\cL \varphi_2/\varphi_2$. Mais je pense que c'est inutilement lourd...}

	% {\color{red} Est-ce qu'il n'y aurait pas un $S$ directement exprimé pour le processus avec killing? Les résultats ci-dessous resteraient vrais...}

First, if $S<+\infty$,~\cite[Theorem~6.6]{Coolen-SchrijnerDoorn2006} proves that there exists a unique quasi-stationary distribution
for $X$ assuming that $(\kappa_x)_{x\in\NN}$ has finite support. We extend this result to any killing rates $(\kappa_x)_{x\in\NN}$
and also prove that the unique quasi-stationary distribution attracts all initial distributions exponentially fast. We can indeed
check that~\eqref{eq:critere-general-BDK} is satisfied by a bounded function $\varphi$ defined as follows: fix $\lambda_1>\lambda_0$
and choose $x_0\in\NN$ large enough such that (see for instance~\cite[Equation~(4.7)]{ChampagnatVillemonais2016b})
\[
  \sup_{x\in \NN} \E_x(e^{\lambda_1 \tau_{D_0}\wedge\tau_\d})<+\infty.
\]
where $D_0=\{1,\ldots,x_0\}$. Then we define $\varphi(0)=0$ and
\[
  \varphi(x)=  \E_x\left(e^{\lambda_1 \tau_{D_0}\wedge\tau_\d}\right),\quad\forall x\in\NN.
\]
Using Markov's property at the first time of jump, one checks that (since $b_x+d_x\to+\infty$ when $x\to+\infty$, we assume w.l.o.g.
that $b_x+d_x+\kappa_x>\lambda_1$ for all $x\geq x_0+1$),
\begin{multline*}
  \varphi(x)=\frac{d_x}{b_x+d_x+\kappa_x-\lambda_1}\varphi(x-1)+\frac{b_x}{b_x+d_x+\kappa_x-\lambda_1}\varphi(x+1)\\+\frac{\kappa_x}{b_x+d_x+\kappa_x-\lambda_1},\ \forall x\geq x_0+1.
\end{multline*}
Hence $\lambda_1=-[b_x (\varphi(x+1)-\varphi(x))+d_x (\varphi(x-1)-\varphi(x))]/\varphi(x)+\kappa_x$ for $x\geq x_0+1$
and~\eqref{eq:critere-general-BDK} is satisfied.

% $\cL\varphi(x)=-\lambda_1 \varphi(x)$ for all $x\geq x_0+1$. Since $\varphi(x)=1$ for all $x\in x_0$, setting $C=\max_{x\in D_0}\cL\varphi(x)$, we deduce that~\eqref{eq:discr-cont-lyap} holds true. Since $\varphi$ is bounded, we deduce from Theorem~\ref{thm:main-E-denumerable} and \ref{thm:E-F} that there exists a unique quasi-stationary distribution which attracts all initial distributions exponentially fast.

Second, if $S=+\infty$ and if $\lambda_0<\liminf_{x\rightarrow+\infty}\kappa_x$, it was proved in~\cite[Theorem~4.3]{Doorn2013} that
there exists a quasi-stationary distribution. The criterion~\eqref{eq:critere-general-BDK} improves this result since it implies that
the quasi-stationary distribution is unique and that it attract all initial distributions exponentially fast.
Indeed,~\eqref{eq:critere-general-BDK} is clearly satisfied for $\varphi\equiv 1$.
 % then our result entail that there exists a unique
 %        quasi-stationary distribution (simply choose $\varphi=1$). This is an improvement over state of the art results obtained
 %        in~\cite{Coolen-SchrijnerDoorn2006}, where the authors prove under the same condition the existence of a quasi-stationary
 %        distribution.
        % Moreover, our result implies that the quasi-stationary distribution attracts all initial distributions
        % exponentially fast.
	 
Last, we can also extend~\cite[Theorem~4.3]{Doorn2013} to processes that do not necessarily admit a unique quasi-stationary
distribution, and in particular that do not come down from infinity. For example, assuming that, for some $\varepsilon>0$,
\[
  \liminf_{x\rightarrow+\infty}\ \kappa_x+\frac{\varepsilon}{1+\varepsilon}d_x-\varepsilon b_x>\lambda_0,
\]
Condition~\eqref{eq:critere-general-BDK} is satisfied for $\varphi(x)=(1+\varepsilon)^x$.

Note that, because of Corollary~\ref{cor:quasi-comp}, our criteria imply the $\lambda_0$-positive recurrence of the process $X$ (cf.\
e.g.~\cite[Eq.\,(26)]{DoornPollett2013}). Therefore, it can only apply to such situations. For results on birth and death processes
with killing which are not $\lambda_0$-positive recurrent, we refer the reader to~\cite[Theorem~4.2]{Doorn2013}.

 % and the results only concern processes whose quasi-stationary distributions attract all initial distributions. In~\cite[Theorem~4.2]{Doorn2013}, the authors prove that, {\color{red} a verifier} if $\lambda_0>\sup_x \kappa(x)$, then the birth and death process with killing admits a quasi-stationary distribution if and only if the corresponding birth and death process without killing is recurrent. Since this situation covers cases of processes that are not $\lambda_0$-positive recurrent, our result does not apply in this full generality. However, we can cover situations where there exists 

\end{exa}

\begin{exa}
\label{exa:multi-bd}
  We consider general multitype birth and death processes in continuous time, taking values in a connected (in the sense of the
  nearest neighbors structure of $\ZZ^d$) subset $E$ of $\ZZ_+^d$ for some $d\geq 1$, with transition rates
  \[
  q_{x,y}=
  \begin{cases}
    b_i(x) & \text{if }y=x+e_i,\\ 
    d_i(x) & \text{if }y=x-e_i,\\
    0 & \text{otherwise,}
  \end{cases}
  \]
  with $e_i=(0,\ldots,0,1,0,\ldots,0)$ where the nonzero coordinate is the $i$-th one and with the convention that the process is
  sent instantaneously to $\d$ when it jumps to a point $y\not\in E$ according to the previous rates. To ensure irreducibility, it is
  sufficient (although not optimal) to assume that $b_i(x)>0$ and $d_i(x)>0$ for all $1\leq i\leq d$ and $x\in E$.
  
  We show below that Theorem~\ref{thm:main-E-denumerable} applies either under the assumption that
  \begin{align}
  \label{eq:multi-bd-1}
  \frac{1}{1+|x|}\sum_{i=1}^d(d_i(x)-b_i(x))\xrightarrow[x\in E,\ |x|\rightarrow+\infty]{}+\infty.
  \end{align}
  or that there exists $\delta>1$ such that
  \begin{align}
  \label{eq:multi-bd-2}
  \sum_{i=1}^d(d_i(x)-\delta\,b_i(x))\xrightarrow[x\in E,\ |x|\rightarrow+\infty]{}+\infty.
  \end{align}
  This improves the general criteria obtained in~\cite{ChampagnatVillemonais2017} since this reference assumes (among other
  assumptions) that $E=\Z_+^d$ and that $\sum_{i=1}^d(d_i(x)-b_i(x))\geq |x|^{1+\eta}$ for some $\eta>0$ and $|x|$ large enough.
  % Note that this example applies to birth and death processes in any connected domain of $\ZZ_+^d$.

  Let us first show that~\eqref{eq:multi-bd-1} implies that the assumptions of Theorem~\ref{thm:main-E-denumerable} are satisfied. In
  order to do so, we define $\varphi(x)=|x|+1=x_1+\ldots+x_d+1$ and $\varphi(\d)=0$ and obtain
  \begin{align*}
    \cL\varphi(x) & =\sum_{i=1}^d (b_i(x)-d_i(x))-\sum_{i=1}^d\left(b_i(x)\11_{x+e_i\not\in E}\varphi(x+e_i)+d_i(x)\11_{x-e_i\not\in E}\varphi(x-e_i)
                    \right)\\ & \leq-\varphi(x)\frac{\sum_{i=1}^d (d_i(x)-b_i(x))}{|x|+1}
  \end{align*}
  The proof is concluded by setting $D_0=\left\{x\in E,\text{ s.t. }\frac{\sum_{i=1}^d (d_i(x)-b_i(x))}{|x|+1}\leq \lambda_0+1\right\}$. 

  Let us now show that~\eqref{eq:multi-bd-2} implies that the assumptions of Theorem~\ref{thm:main-E-denumerable} are satisfied.
  Setting $\varphi(x)=\exp\langle a,x\rangle$ for a given $a\in(0,\infty)^d$ and $\varphi(\d)=0$, we obtain
  \[
  \cL\varphi(x)\leq-\varphi(x)\left(\sum_{i=1}^d (1-e^{-a_i})d_i(x)+(1-e^{a_i})b_i(x)\right).
  \]
  Choosing $a=(\varepsilon,\ldots,\varepsilon)$ with $\varepsilon$ small enough, we have
  \[
 \liminf_{x\in E,\ |x|\rightarrow+\infty}\sum_{i=1}^d (1-e^{-a_i})d_i(x)+(1-e^{a_i})b_i(x)=+\infty.
  \]
  Taking $D_0=\left\{x\in E,\text{ s.t. }\sum_{i=1}^d (1-e^{-a_i})d_i(x)+(1-e^{a_i})b_i(x)\leq \lambda_0+1\right\}$ allows us to conclude the proof.
\end{exa}

\begin{proof}[Proof of Theorem~\ref{thm:main-E-denumerable}]
  The fact that $\lambda_0$ is independent of $x$ is classical for irreducible processes (cf.\ e.g.~\cite{kingman-63}). We set
  $L=D_0$. Since $X$ is a non-explosive pure jump continuous time process, it satisfies the strong Markov property and the entrance
  times $\tau_L$ and $\tau_\d$ are stopping times. This entails~(F0).

  For all $x,y\in L$, we have
  \[
    \P_x(X_2\in \cdot)\geq\inf_{u,v\in L}\P_u(X_1=v)\,\P_y(X_1\in\cdot),
  \]
  where $\inf_{u,v\in L}\P_u(X_1=v)>0$ by assumption, which implies~(F1) and (F3).

  We set $\psi_1=\varphi$. For all $0\leq s\leq 1$, using~\eqref{eq:discr-cont-lyap} and Dynkin's formula, one has that for all $x\in
  L$
  \begin{align*}
    \E_x\left(\psi_1(X_s)\11_{s<\tau_\d}\right)&\leq e^{Cs}\sup_{y\in L}\psi_1(y).
  \end{align*}
  Similarly, setting $\gamma_1=e^{-\lambda_1}$, for all $x\in E\setminus L$,
  \begin{align*}
    \E_x\left(\psi_1(X_1)\11_{1<\tau_L\wedge \tau_\d}\right)&\leq e^{-\lambda_1}\psi_1(x)=\gamma_1\psi_1(x).
  \end{align*}
  Choosing any $\gamma_2\in(\gamma_1,e^{-\lambda_0})$, one obtains that~(F2) is satisfied and the first part of
  Theorem~\ref{thm:main-E-denumerable} is proved.

  {The inequality $\sum_{y\in E\setminus\{x\}}q_{x,y}\eta(y)<\infty$ for all $x\in E$ follows from the fact that $\eta\in
    L^\infty(\psi_1)$ and the fact that $P_t\eta(x)=e^{-\lambda_0 t}\eta(x)$ was proved in Theorem~\ref{thm:E-F}. It then follows
    from Markov's property and the last equality that $(e^{\lambda_0 t}\eta(X_t),t\geq 0)$ is a martingale for the canonical
    filtration associated to $X$, with the convention that $\eta(\d)=0$. Now, it is standard to represent the Markov process $X$ as a
    solution to a stochastic differential equation driven by a Poisson point process: assume that the elements of the finite or
    countable set $E$ are labeled by distinct positive integers, that $\d=0$ and, for all $x,i\in\ZZ_+$, let
    $\kappa_i(x)=q_{x,0}+q_{x,1}+\ldots+q_{x,i}$ with the convention that $q_{x,x}=0$ and $q_{x,i}=0$ for all $x$ or $i\not\in
    E\cup\{\d\}$ and set $q(x)=\sum_{i\in\ZZ_+}q_{x,i}<\infty$. Given a Poisson point measure $N(\mathrm ds,\mathrm d\theta)$ on
    $\RR_+^2$ with intensity the Lebesgue measure on $\RR_+^2$, the process $X$ solution
    \[
    X_t=X_0+\int_0^t\int_0^{q(X_{s-})}\sum_{i=0}^\infty\11_{\theta\in[\kappa_{i+1}(X_{s-}),\kappa_i(X_{s-}))}(i-X_{s-})N(\mathrm ds,\mathrm d\theta)
    \]
    is well-defined for all time $t\geq 0$ almost surely and is a Markov process with matrix of jump rates $(q_{i,j})_{i,j\in\ZZ_+}$.
    Introducing the compensated Poisson measure $\widetilde{N}(\mathrm ds,\mathrm d\theta)=N(\mathrm ds,\mathrm d\theta)-\mathrm
    ds\,\mathrm d\theta$, it follows from basic stochastic calculus for jump processes (cf.\ e.g.~\cite{protter-04}) that
    \begin{align*}
      e^{\lambda_0 t}\eta(X_t) = X_0 & +\int_0^t \int_0^{q(X_{s-})}e^{\lambda_0
        s}\sum_{i=0}^\infty\11_{\theta\in[\kappa_{i+1}(X_{s-}),\kappa_i(X_{s-}))}(\eta(i)-\eta(X_{s-}))\widetilde{N}(\mathrm ds,\mathrm d\theta) \\ &
      +\int_0^t e^{\lambda_0
        s}\left(\sum_{i=0}^\infty q_{X_{s},i}(\eta(i)-\eta(X_{s}))+\lambda_0 \eta(X_s)\right) \mathrm ds.
    \end{align*}
    Since $e^{\lambda_0 t}\eta(X_t)$ is a $\PP_x$-martingale, the Doob-Meyer decomposition theorem entails that
    \[
    \int_0^t e^{\lambda_0
      s}\left(\sum_{i=0}^\infty q_{X_{s},i}(\eta(i)-\eta(X_{s}))+\lambda_0 \eta(X_s)\right) \mathrm ds=0
    \]
    $\PP_x$-almost surely for all $t\geq 0$ and all $x\in E$. Hence, if there exists $y\in E$ such that $\mathcal{L}
    \eta(y)\neq-\lambda_0\eta(y)$, by irreducibility, there exists an event with positive probability under $\PP_x$ such that the
    previous integral is non-constant. We obtain a contradiction and hence $\mathcal{L}\eta(x)=-\lambda_0\eta(x)$ for all $x\in E$.
  }
\end{proof}

\section{On reducible examples}
\label{sec:reducible}

The criteria and examples studied in the last two sections assume that the process $X$ is irreducible in $E$. However, the abstract
results of Section~\ref{sec:main-discrete-time} do not require the state space to be irreducible. Our goal in this section is to
explain that our criteria are also well-suited to cases of reducible absorbed Markov processes, in the sense that the state space $E$
can be partitioned in a finite or countable family of communication classes. The study of quasi-stationary behavior for such
processes has been up to now restricted to the case of finite state spaces or to particular classes of
models~\cite{Mandl1959,darroch-seneta-65,Ogura1975,gosselin-01,ChampagnatRoelly2008,DoornPollett2008,DoornPollett2009,ChampagnatDiaconisEtAl2012,DoornPollett2013,BenaimCloezEtAl2016}. Our criteria provide new
practical tools to tackle this problem, further exploited in~\cite{ChampagnatVillemonais2022}.

In Subsection~\ref{sec:three-com-class}, we consider a general setting with three successive sets. In
Subsection~\ref{sec:count-com-class}, we consider a birth and death process with a countable infinity of communication
classes.

\subsection{Three successive sets}
\label{sec:three-com-class}

In this section, we consider a discrete time Markov process $(X_n,n\in\ZZ_+)$ evolving in a measurable set $E\cup\{\d\}$ with
absorption at $\d\notin E$. We assume that the transition probabilities of $X$ satisfy the structure displayed in
Figure~\ref{fig:trangraph} : one can find a partition $\{D_1,D_2,D_3\}$ of $E$ such that the process starting from $D_1$ can access
$D_1\cup D_2\cup D_3\cup\{\d\}$, the process starting from $D_2$ can only access $D_2\cup D_3\cup\{\d\}$, and the process starting
from $D_3$ can only access $D_3\cup\{\d\}$. More formally, we assume that $\P_x( T_{D_3}\wedge\tau_\d< T_{D_1})=1$ for all $x\in D_2$
and that $\P_x(\tau_\d< T_{D_1\cup D_2})=1$ for all $x\in D_3$, where we recall that, for any measurable set $A\subset E$, $
T_A=\inf\{n\in\ZZ_+,\ X_n\in A\}$.

\begin{figure}
  \center
  \includegraphics[width=6cm]{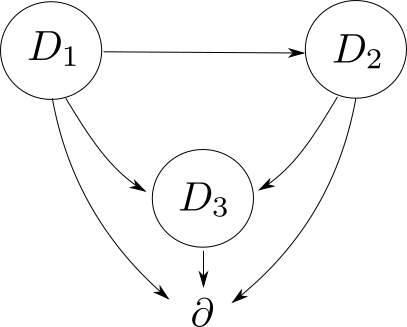}
   \caption{Transition graph displaying the relation between the sets $D_1$, $D_2$, $D_3$ and $\d$.}
   \label{fig:trangraph}
\end{figure}

Our aim is to provide sufficient conditions ensuring that $X$ satisfies Assumption~(E). In order to do so, we assume that
Assumption~(E) is satisfied by the process $X$ before exiting $D_2$. This corresponds to the following assumption.

\medskip

\textbf{Assumption (H1).} The absorbed Markov process $Y$ evolving in $D_2\cup\{\d\}$, defined by
\[
  Y_n=
  \begin{cases}
    X_n&\text{ if }n< T_{D_1\cup D_3\cup\{\d\}},\\
    \d&\text{ if }n\geq  T_{D_1\cup D_3\cup\{\d\}},
  \end{cases}
\]
satisfies Assumption~(E). In what follows, we denote the objects related to $Y$ with a superscript $Y$, for instance, the constants
of Assumption~(E) for $Y$ are denoted by $\theta^Y_1>0$, $\theta^Y_2>0$.

\medskip
We also assume that the exit times from $D_1$ and $D_3$ for the
process $X$ admit exponential moments of sufficiently high order, as
stated by the following assumption.

\medskip
\textbf{Assumption (H2).}
There exists a positive constant $\gamma<\theta^Y_0$ such that, for all $x\in D_1$, 
\[
  \E_x\left(\gamma^{- T_{D_2}}\varphi^Y_1\left(X_{ T_{D_2}}\right)\11_{ T_{D_2}< T_{D_3}\wedge\tau_\d}\right)<+\infty,\quad
  \E_x\left(\gamma^{- T_{D_3}\wedge \tau_\d}\11_{ T_{D_3}\wedge\tau_\d< T_{D_2}}\right)<+\infty,
\]
and such that
\[
 \sup_{x\in D_3}\E_x\left(\gamma^{-\tau_{\d}}\right)<+\infty.
\]

\medskip

We are now able to state the main result of this section.

\begin{thm}
  \label{thm:three-comm-class-1}
  Under Assumptions~(H1) and~(H2), the process $X$ satisfies Assumption~(E) with $K=K^Y$,
  \begin{equation}
    \label{eq:def-phi-reducible}
    \varphi_1(x)=\E_x\left(\gamma^{- T_K\wedge\tau_\d}\right)\quad\text{and}\quad\varphi_2(x)\geq c\11_{x\in K},\ \forall x\in E,
  \end{equation}
  for some constant $c>0$.
  In particular, it admits a unique quasi-stationary distribution $\nu_{QSD}$ such that $\nu_{QSD}(\varphi_1)<\infty$ and
  $\nu_{QSD}(\varphi_2)>0$. Moreover, there exist two constants $C>0$ and $\alpha\in(0,1)$ such that, for all probability measure
  $\mu$ on $E$ such that $\mu(\varphi_1)<\infty$ and $\mu(\varphi_2)>0$,
  \begin{align}
    \label{eq:cv-thm-reducible}
    \left\|\P_{\mu}(X_n\in\cdot\mid n<\tau_\d)-\nu_{QSD}\right\|_{TV(\varphi_1)}\leq C\alpha^n\,\frac{\mu(\varphi_1)}{\mu(\varphi_2)}.
  \end{align}
  Finally, $\theta_0=\theta_0^Y$, $\nu_{QSD}(D_1)=0$ and the function $\eta$ of Theorem~\ref{thm:eta} vanishes on $D_3$.
\end{thm}

% In particular, one deduces from the last property that $\{x\in E,\ \eta(x)>0\}\subset D_1\cup D_2$ (see Remark~\ref{rem:converse-attraction-domain}).
%,
%where we recall that $E'=\{x\in E:\exists n\in\NN,\ P_n\11_K(x)>0\}$.

Before turning to the proof of this result, let us make some remarks.

\begin{rem}
  \label{rem:comments-reducible}
  \begin{enumerate}
  \item The fact that there are three different sets $D_1$, $D_2$ and $D_3$ in the decomposition of $E$ is not restrictive on the
    number of communication classes. Indeed, the three sets can contain several communication classes. 
%    {\color{gray} An example will be given in
%    Section~\ref{sec:finite-classes}.}
  \item A similar result can be obtained for continuous time processes, based on Assumption~(F) instead of~(E), with the additional
    technical assumption that the strong Markov property can be applied at the exit times of $D_1$ and $D_2$.
  \item We emphasize that, besides the exponential moment assumption, there is no additional requirement on the behavior of the Markov
    process in $D_1$ and $D_3$. In these sets, the process might be for instance periodic or deterministic and could satisfy that
    $\P_x(n<\tau_\d)=0$ for some $n\in\N$. 
%    {\color{gray}In particular, one might
%    have $\P_x(n<\tau_\d)=0$ for some $x\in D_1\cup D_3$ and some $n\in\N$ (this situation is discussed in
%    Section~\ref{sec:general-comments} \textcolor{red}{[ceci a chang\'e]}).}
  \item One easily checks from the proof that the function $\varphi_1$ in~\eqref{eq:def-phi-reducible} is bounded (up to a
    multiplicative positive constant) from above by
    \[
    \E_x\left(\gamma^{- T_{D_2}}\varphi^Y_1\left(X_{ T_{D_2}}\right)\11_{ T_{D_2}< T_{D_3}\wedge\tau_\d}\right)+
    \E_x\left(\gamma^{- T_{D_3}\wedge \tau_\d}\11_{ T_{D_3}\wedge\tau_\d< T_{D_2}}\right)
    \]
    on $D_1$, by $\varphi^Y_1$ on $D_2$ and by a constant on $D_3$.
  \item In particular, if $\varphi^Y_1$ is uniformly bounded and if the first statement in Assumption~(H2) is replaced by
    \[
    \sup_{x\in D_1}\,\E_x\left(\gamma^{- T_{D_2\cup D_3}\wedge \tau_\d}\right)<+\infty,
    \]
    then one can also choose a bounded function $\varphi_1$ in Assumption~(E) for $X$.
%     This fact will be used in
%    Section~\ref{sec:finite-classes}.
  \end{enumerate}
\end{rem}

%{\color{gray}\begin{rem}
%  \label{rem:several-QSD-and-eigenfunctions}
%  Assume in addition that Condition~(E) is satisfied for the process $X^{D_1}$ started from $D_1$ with the same transitions as $X$
%  but absorbed at its first exit time of $D_1$ and for the process $X^{D_3}$ started from $D_3$ with the same transitions as $X$ and
%  absorbed in $\d$. Then the quasi-stationary distribution $\nu_{QSD}^{D_3}$ of Theorem~\ref{thm:main} for $X^{D_3}$ extended by 0 to
%  $E\setminus D_3$ is a quasi-stationary distribution for the absorbed process $X$. This shows that uniqueness of a quasi-stationary
%  distribution may not hold even if $\varphi_1$ is bounded (see Corollary~\ref{cor:dom-attract-final}). Moreover, the constant
%  $\theta_0^{D_1}$ and the function $\eta^{D_1}$ of Theorem~\ref{thm:eta} for $X^{D_1}$ extended by 0 to $D_2\cup D_3\cup\{\d\}$
%  satisfies, for all $x\in D_1$,
%  \[
%  \hat{P}_1\eta^{D_1}(x)=\E_x\left[\11_{1< T_{D_2\cup D_3}\wedge \tau_\d}\eta^{D_1}(X_1)\right]=\theta_0^{D_1}\eta^{D_1}(x).
%  \]
%  Hence $\eta^{D_1}$ is an eigenfunction for $\hat{P}_1$ corresponding to case 3.\ in Corollary~\ref{cor:eta}, so $\theta_0^{D_1}\leq
%  \theta_0\alpha_1$. \textcolor{red}{[\`a relire... je ne comprends plus trop]}
%\end{rem}}

 \begin{rem}
   In general, processes on reducible state spaces may not satisfy Assumption~(E). For example the convergence
     in~\eqref{eq:cv-thm-reducible} may not be exponential, or quasi-stationary distributions may not be unique, even if the process
     $X$ restricted to $D_1$, $D_2$ or $D_3$ satisfy condition~(E). We refer the reader to~\cite{ChampagnatVillemonais2022} for a
   more general discussion on quasi-stationary distributions and quasi-limiting behavior for general processes on reducible state spaces.
 \end{rem}   

\begin{proof}[Proof of Theorem~\ref{thm:three-comm-class-1}]
  Let us prove that Assumption~(E) is satisfied by the process $X$. Note that, because of Lemma~\ref{lem:theta-2-et-0}, one can
  assume without loss of generality that $\gamma<\theta_2^Y$.

\medskip
\noindent
\textit{Step 1. Assumption~(E1).}

We set $K=K^Y$, $n_1=n_1^Y$, $c_1=c_1^Y$ and $\nu=\nu^Y$.
% (remember that the objects with a superscript $Y$ are those of Assumption~(E) satisfied by the process $Y$).
Assumption~(E1) for $X$ is an immediate consequence of Assumption~(E1) for $Y$.

\medskip
\noindent
\textit{Step 2. Assumption~(E2).} 

We set $\theta_2=\theta_2^Y$ and
\[
\varphi_2(x)=\begin{cases}
\varphi_2^Y(x)&\text{ if }x\in D_2\\
0&\text{ if }x\in D_1\cup D_3.
\end{cases}
\]
Then the second and fourth lines of Assumption~(E) for $X$ are direct consequences of the same lines of Assumption~(E) for $Y$.

Without loss of generality, we assume (increasing $\gamma$ if necessary, which does not change the fact that Assumptions~(H2) is
true) that $\gamma\in (\theta^Y_1,\theta^Y_2)$. We define
\[
\varphi_1(x)=\E_x\left(\gamma^{- T_K\wedge \tau_\d}\right),\quad\forall x\in E\cup\{\d\}.
\]
Let us first check that $\varphi_1$ is finite on $E$. For all $x\in D_3$, using that $\P_x(\tau_\d< T_{D_1\cup D_2})=1$ and that $K\subset D_2$, one deduces that
\[
\varphi_1(x)=\E_x\left(\gamma^{-\tau_\d}\right)\leq A:=\sup_{x\in D_3}\E_x\left(\gamma^{-\tau_\d}\right)<+\infty.
\]
For all $x\in D_2$, using the strong Markov property and inequality~\eqref{eq:comp-surv-1} for the process $Y$, one deduces that 
\begin{align}
  \varphi_1(x)&=\E_x\left(\gamma^{- T_K\wedge  T_{D_2^c}}\11_{ T_K< T_{D_2^c}}\right)+\E_x\left(\gamma^{-\tau_\d}\11_{ T_{D_2^c}<
      T_K}\right) \notag \\
  &=\E_x\left(\gamma^{- T_K\wedge  T_{D_2^c}}\11_{ T_K< T_{D_2^c}}\right)+\E_x\left(\gamma^{- T_K\wedge T_{D_2^c}}\11_{ T_{D_2^c}< T_K}\E_{X_{
        T_{D_2^c}}}\left(\gamma^{-\tau_\d}\right)\right) \notag \\
  &\leq A \E_x\left(\gamma^{- T_K\wedge  T_{D_2^c}}\right)\leq \frac{A}{1-\theta_1^Y/\gamma}\varphi_1^Y(x). \label{eq:pf-cas-reductible}
\end{align}

For all $x\in D_1$, one has, using the Markov property and the above inequalities,
\begin{multline}
  \E_x\left(\gamma^{- T_K\wedge\tau_\d}\right)=\E_x\left(\gamma^{- T_{D_2\cup D_3}\wedge \tau_\d}\varphi_1(X_{ T_{D_2\cup D_3}\wedge
      \tau_\d})\right)\\ 
  \leq \frac{A}{1-\theta_1^Y/\gamma}\left[
    \E_x\left(\gamma^{- T_{D_2}}\varphi^Y_1\left(X_{ T_{D_2}}\right)\11_{ T_{D_2}< T_{D_3}\wedge\tau_\d}\right)+ 
    \E_x\left(\gamma^{- T_{D_3}\wedge \tau_\d}\11_{ T_{D_3}\wedge\tau_\d< T_{D_2}}\right)\right],
  \label{eq:pf-cas-reductible-bis}
\end{multline}
which is finite by Assumption~(H2).

The definition of $\varphi_1$ immediately implies that $\inf_E\varphi_1\geq 1$ and, since $\varphi_1^Y$ is uniformly bounded over
$K\subset D_2$,~\eqref{eq:pf-cas-reductible} implies that $\sup_K \varphi_1<+\infty$. Hence the first line of Assumption~(E2) is
satisfied. Moreover, for all $x\in K$,
\begin{align*}
  P_1\varphi_1(x)&=\E_x\left(\11_{X_1\in D_2}\E_{X_1}\left(\gamma^{- T_K\wedge\tau_\d}\right)\right)+\E_x\left(\11_{X_1\in
      D_3}\E_{X_1}\left(\gamma^{-\tau_\d}\right)\right)\\
  &\leq \E_x\left(\11_{X_1\in D_2}\frac{A}{1-\theta_1^Y/\gamma}\varphi^Y_1(X_1)\right)+A\\
  &= \frac{A}{1-\theta_1^Y/\gamma}\,P^Y_1\varphi_1^Y(x)+A\,\leq \frac{A}{1-\theta_1^Y/\gamma}\,\left(\theta^Y_1\sup_K\varphi^Y_1+c_2^Y\right)+A.
\end{align*}
Hence, the third line of (E2) for $X$ with $\theta_1=\gamma$ follows from Lemma~\ref{lem:varphi1}.

\medskip
\noindent
\textit{Step 3. Assumption~(E3).}

For all $x\in K$, we have, for all $n\geq 1$,
\begin{align}
  \label{eq:three_sets_sub_step}
  \P_x(n<\tau_\d)\leq \P_x(n<\tau_\d\wedge  T_{D_3})+\P_x( T_{D_3}\leq n<\tau_\d).
\end{align}
On the one hand, by Lemma~\ref{lem:comp-surv}, there exists a constant $C>0$ such that
\[
\P_x(n<\tau_\d\wedge T_{D_3})\leq \frac{C\varphi_1^Y(x)}{1-\theta^Y_1/\theta^Y_2}\inf_{y\in K} \P_y(n< T_{D_2^c})\leq
\frac{C\,\sup_K\varphi_1^Y}{1-\theta^Y_1/\theta^Y_2}\inf_{y\in K} \P_y(n<T_{D_2^c}).
\]
On the other hand, using Markov's property and Markov's inequality,
\begin{align*}
\P_x( T_{D_3}\leq n<\tau_\d)&=\E_x\left(\11_{ T_{D_3}\leq n}\restriction{\P_{X_{ T_{D_3}}}(n-u<\tau_\d)}{u= T_{D_3}}\right)\\
      &\leq \E_x\left(\11_{ T_{D_3}\leq n}\varphi_1(X_{ T_{D_3}})\gamma^{n- T_{D_3}}\right)\leq A \E_x\left(\11_{ T_{D_2^c}\leq n}\gamma^{n- T_{D_2^c}}\right),
\end{align*}
since $\{ T_{D_3}\leq n\}\subset\{ T_{D_2^c}= T_{D_3}\}$. Now, using Theorem~\ref{thm:eta} and the fact that $\eta^Y$ is uniformly
bounded from above and away from 0 on $K$, we deduce that there exist constants $C,C'>0$ such that
\begin{align*}
\E_x\left(\11_{ T_{D_2^c}\leq n}\gamma^{n- T_{D_2^c}}\right)
      &=\sum_{k=1}^n \P_x( T_{D_2^c}=k)\,\gamma^{n-k}\leq \sum_{k=1}^n \P_x( T_{D_2^c}>k-1)\,\gamma^{n-k}\\
      &\leq C\,\sum_{k=1}^n (\theta_0^Y)^{k-1}\,\gamma^{n-k}\leq C\,(\theta_0^Y)^{n-1} \frac{1}{1-\gamma/\theta_0^Y}\\
      &\leq C\, C'\,\frac{(\theta_0^Y)^{-1}}{1-\gamma/\theta_0^Y}\,\inf_{y\in K}\P_y(n< T_{D_2^c}).
\end{align*}

Finally, we obtain from~\eqref{eq:three_sets_sub_step} that there exists a constant $C''>0$ such that, for all $x\in K$,
\begin{equation}
  \label{eq:pf-cas-reductible-2}
  \P_x(n<\tau_\d)\leq C''\inf_{y\in K}\P_y(n< T_{D_2}^c)\leq C''\inf_{y\in K}\PP_y(n<\tau_\d).
\end{equation}
This concludes Step~3.

\medskip
\noindent
\textit{Step 4. Conclusion.}

Assumption~(E4) for the process $X$ is an immediate consequence of Assumption~(E4) for the process $Y$, and hence we have checked
that $X$ satisfies Assumption~(E). The convergence result of Theorem~\ref{thm:three-comm-class-1} is exactly the convergence result
obtained in~Theorem~\ref{thm:main}.

Note that~\eqref{eq:pf-cas-reductible-2} entails that, for any $x\in K$,
\begin{align*}
  \limsup_{n\rightarrow+\infty} \,(\theta_0^Y)^{-n} \P_x(n< T_{D_2^c})&\leq\limsup_{n\rightarrow+\infty} \,(\theta_0^Y)^{-n} \P_x(n<\tau_\d)\\
  &\leq C'' \limsup_{n\rightarrow+\infty} \,(\theta_0^Y)^{-n} \P_x(n< T_{D_2^c})
\end{align*}
and that Theorem~\ref{thm:eta} applied to $Y$ entails
\[
\limsup_{n\rightarrow+\infty} \,(\theta_0^Y)^{-n} \P_x(n< T_{D_2^c})=\eta^Y(x)<+\infty.
\]
Since it follows from Theorem~\ref{thm:eta} applied to $X$ that $\lim_{n\rightarrow+\infty} \theta_0^{-n} \P_x(n<\tau_\d)>0$, we
deduce that $\theta_0= \theta_0^Y$.

Finally, for all $x\in K$, the structure of the transition graph of $X$ implies that
\begin{align*}
0=\P_x(X_n\in D_1\mid n<\tau_\d)\xrightarrow[n\rightarrow+\infty]{}\nu_{QSD}(D_1),
\end{align*}
so that $\nu_{QSD}(D_1)=0$. Moreover, for all $x\in D_3$, Markov's inequality and Assumption~(H2) yield the inequality
$\P_x(n<\tau_\d)\leq A\,\gamma^n$, for all $x\in K$ and all $n\geq 1$. Since  $\theta_0=\theta_0^Y>\gamma$ by assumption, we deduce that, for all $x\in K$,
$\lim_{n\rightarrow+\infty} \theta_0^{-n}\P_x(n<\tau_\d)=0$, which means that $\eta(x)=0$.

This concludes the proof of Theorem~\ref{thm:three-comm-class-1}.
\end{proof}

\subsection{Countably many communication classes}
\label{sec:count-com-class}

In this section, we study a particular case of a continuous time c\`adl\`ag Markov process $(X_t)_{t\in[0,+\infty)}$
with a countable infinity of communication classes and we show that the process admits a quasi-stationary distribution.

More precisely, we assume that $X$ evolves in the state space $\N\times \Z_+$ and, denoting $N_t\in \N$ and
$Y_t\in \Z_+$ the two components of $X_t$ for all $t\in[0,+\infty)$, that there exist three positive functions
$b,d,f:\N\rightarrow (0,+\infty)$ such that
\begin{itemize}
\item $N$ is a Poisson process with intensity $1$,
\item $Y$ is a process such that, at time $t$,
\[
\text{ $Y$ jumps from $Y_t$ to $y\in\Z_+$ with rate }\begin{cases}
f(N_t)\,b(Y_t)&\text{ if }y=Y_t+1\text{ and }Y_t\geq 1,\\
f(N_t)\,d(Y_t)&\text{ if }y=Y_t-1\text{ and }Y_t\geq 1,\\
0&\text{ otherwise.}
\end{cases}
\]
\end{itemize}
The set $\N\times\{0\}$ is absorbing for $X$ and we are interested in the quasi-stationary behavior of $X$ conditioned to not hit
this set. Note that, in this case, each set $\{n\}\times \N$ is a communication class.

\begin{rem}
  This process can be used to model the survival of an individual (for example a bacterium) whose
  metabolic efficiency (for example its ability to consume resources) changes with time, due to
  aging~\cite{SteinsaltzEvans2004}. Here $Y$ is the vitality of the individual, who dies when its vitality hits $0$,
  $f(N)$ is the metabolic rate of the individual, which may for example decrease in the early life of the individual up to age $n_0$
  and then accelerates progressively.

  This can also model the accumulation of deleterious mutations in a population under the assumption that mutations do not overlap,
  i.e.\ that when a mutant succeeds to invade the population (either because they are advantaged or due to genetic drift for
  deleterious mutations), other types of mutants disappear rapidly. Here $Y$ represents the size of the population and $N$ the number
  of mutations. It is typical to assume that the first $n_0$ mutations that invade are advantageous (which corresponds to
  adaptation), and afterwards that deleterious mutations start to accumulate, hence accelerating the extinction of the species
  (extinction vortex~\cite{CoronMeleardEtAl2013,Coron2013}).

  In both cases, it is relevant to assume that $f$ is decreasing on $\{1,2,\ldots,n_0\}$ and increasing on $\{n_0,n_0+1,\ldots\}$.
\end{rem}

We assume that $(d(y)-b(y))/y\rightarrow +\infty$ when $y\rightarrow+\infty$ or that there exists $\delta>1$ such that
$d(y)-\delta\,b(y)\rightarrow +\infty$. Hence the birth and death process $Z$ evolving in $\N$, with birth rates $(b(z))_{z\in\N}$
and death rates $(d(z))_{z\in\N}$, satisfies Assumption~(F) by Theorem~\ref{thm:main-E-denumerable} (see Example~\ref{exa:multi-bd}).
In particular, there exist an eigenvalue $\lambda^Z_0>0$ and eigenfunction $\eta^Z:\N\rightarrow(0,+\infty)$ such that, for all
$z\in\N$, $\mathcal{L}^Z \eta^Z=-\lambda^Z_0\eta^Z$, where the operator $\mathcal{L}^Z$ is defined as the operator $\mathcal{L}$
in~\eqref{eq:def-L-discret}.

\begin{thm}
  \label{thm:inf_comm_class}
  Assume also that there exists a unique $n_0\in\NN$ such that $f(n_0)=\min_{n\in\NN}f(n)$ and that
  $\liminf_{n\rightarrow+\infty} f(n)>f(n_0)+\frac{1}{\lambda_0^Z}$. Then the process $X$ satisfies
  Assumption~(F) and admits a quasi-stationary distribution $\nu_{QSD}$ whose domain of attraction
  contains all Dirac measures $\delta_{n,y}$, with $n\leq n_0$ and $y\in
  \N$.
\end{thm}

Of course, all the consequences of Theorem~\ref{thm:E-F} also apply here, taking the functions $\psi_1$ and $\psi_2$ as described in
the proof.

% In practice, one may use the fact that $\lambda_0^Z$ is always smaller than $d(1)$. 
% The proof of the last result mainly makes use of this special structure of the process and might be generalized to processes $Z$ that
% are not birth-death processes.

\begin{proof}
  The proof maks use of the special structure of the process $Y$, which can be constructes as
  \[
  Y_t=Z_{\int_0^t f(N_s)\mathrm ds},\quad\forall t\geq 0.
  \]
  In general, we shall denote the objects related to $Z$ with a superscript $Z$, for example $\psi^Z_1$ is the functions involved
  in~(F2) and $L^Z$ is the set involved in~(F) for $Z$. We can assume without loss of generality 
  as in Theorem~\ref{thm:main-E-denumerable} that $L^Z=D_0^Z$, i.e.
  \begin{equation}
    \label{eq:preuve-preuve}
    \mathcal{L}^Z\psi_1^Z\leq-\lambda_1^Z\psi_1^Z+\bar{C}\11_{L^Z}  
  \end{equation}
  with $\psi_1^Z(0)=0$ and $\lambda^Z_1>\lambda^Z_0$.

  Our goal is to apply Theorem~\ref{thm:main-E-denumerable} to the process $X=(N,Y)$. We define the finite set $D_0=\{n_0\}\times
  L^Z$, so that $\P_{(n_0,x)}(X_1=(n_0,y))>0$ for all $(n_0,x)$ and $(n_0,y)$ in $D_0$, and check that $\lambda_0\leq
  f(n_0)\lambda_0^Z+1$. Indeed, for all $y\in L^Z$,
  \begin{align*}
    e^{t(f(n_0)\lambda_0^Z+1)}\P_{(n_0,y)}((N_t,Y_t)=(n_0,y)) & \geq
    e^{tf(n_0)\lambda_0^Z}\P^Z_{y}(Z_{f(n_0)t}=y) \\ & \xrightarrow[t\rightarrow+\infty]{}
    \eta^Z(y)\nu^Z_{QSD}(\{y\})>0.
  \end{align*}
  We fix $\lambda_1$ such that
  \[
  f(n_0)\lambda_0^Z+1<\lambda_1<\left(\lambda_0^Z\inf_{n\neq n_0}f(n)+1\right)\wedge \left(\lambda_0^Z\liminf_{n\rightarrow+\infty}
    f(n)\right)\wedge\left(\lambda^Z_1 f(n_0)+1\right)
  \]
  and we choose
  \begin{itemize}
  \item $n_1> n_0$ such that, for all $n\geq n_1$, $\lambda_1<\lambda_0^Z f(n)$;
  \item $c>0$ small enough so that $\psi_1^Z(x)\geq c\eta^Z(x)$ for all $x\geq 1$ (such a constant exists thanks to Theorem~\ref{thm:eta});
  \item $a>0$ large enough so that $\lambda_1<\lambda^Z_1 f(n_0)+1-e^{-a}$;
  \item $\varepsilon>0$ small enough so that $\lambda_1<(\lambda_0^Z-\varepsilon)\inf_{n\neq n_0}f(n) +1$;
  \item $b>a$ large enough so that $\lambda_1<(\lambda_0^Z-\varepsilon)\inf_{n\neq n_0}f(n) +1-e^{-b}$ and $\bar{C}
    e^{a-b}<\varepsilon\inf_{y\in L^Z}\eta^Z(y)$, where the constant $\bar{C}$ is the one of~\eqref{eq:preuve-preuve}.
  \end{itemize}
  We can now define
  \[
  \psi_1(n,y)=
  \begin{cases}
    \psi_1^Z(y) & \text{if }n=n_0, \\
    e^{a(n_0-n)}\psi_1^Z(y)+e^{b(n_0-n)}\eta^Z(y) & \text{if }n<n_0, \\
    ce^{-a(n-n_0)}\eta^Z(y) & \text{if }n_0<n< n_1, \\
    ce^{-a(n_1-n_0)}\eta^Z(y) & \text{if } n_1\leq n.
  \end{cases}
  \]
  In the case where $n<n_0$, it follows from~\eqref{eq:preuve-preuve} that
  \begin{align*}
    \mathcal{L}\psi_1(n,y) \leq & -\left(\lambda_1^Z f(n)+1-e^{-a}\right) e^{a(n_0-n)}\psi_1^Z(y) \\ & \qquad -\left(\lambda_0^Z f(n)+1-e^{-b}\11_{n<n_0-1}\right)
                                                                                                       e^{b(n_0-n)}\eta^Z(y)\\
    & \qquad\qquad+\frac{\bar{C}}{\inf_{z\in L^Z}\eta^Z(z)}f(n)e^{a(n_0-n)}\eta^Z(y) \\
    \leq & -\lambda_1 e^{a(n_0-n)}\psi_1^Z(y)-\left[(\lambda_0^Z-\varepsilon)f(n)+1-e^{-b}\11_{n<n_0-1}\right]e^{b(n_0-n)}\eta^Z(y) \\ 
    & \qquad+\varepsilon f(n) e^{a(n_0-n)}\left(e^{b-a}-e^{(b-a)(n_0-n)}\right)\eta^Z(y) \\
    \leq & -\lambda_1\psi_1(n,y).
  \end{align*}
  When $n=n_0$, we have
  \begin{align*}
    \mathcal{L}\psi_1(n_0,y) & \leq -\lambda_1^Zf(n_0)\psi_1^Z(y)+\bar{C}\11_{L^Z}(y)f(n_0) + ce^{-a}\eta^Z(y)-\psi^Z_1(y) \\
    & \leq -\lambda_1 \psi_1(n_0,y)+\bar{C}f(n_0)\11_{D_0}(n_0,y).
  \end{align*}
  When $n_0<n<n_1$, we have
  \begin{align*}
    \mathcal{L}\psi_1(n,y) & \leq -\lambda_0^Z f(n)\,c\,e^{-a(n-n_0)}\eta^Z(y) +
    c\,e^{-a(n-n_0)}\eta^Z(y)\left(e^{-a}-1\right) %-c\,e^{-a(n-n_0)}\eta^Z(y) 
    \\ & \leq -\lambda_1 \psi_1(n,y).
  \end{align*}
  When $n_1\leq n$, we have
  \begin{align*}
    \mathcal{L}\psi_1(n,y) & \leq -\lambda_0^Z f(n)\eta^Z(y) \leq -\lambda_1 \psi_1(n,y).  
  \end{align*}
  Finally we have proved that $\mathcal{L}\psi_1(n,y) \leq -\lambda_1 \psi_1(n,y)+\bar{C}f(n)\11_{D_0}(n,y)$, where
  $\lambda_1>\lambda_0$. Now, note that, since $Z$ is a birth-death process, basic comparison arguments imply that
  $\eta^Z(k)\geq\eta^Z(1)>0$ for all $k\geq 1$. Therefore, the function $\psi_1$ is uniformly lower bounded, so that it satisfies the
  assumptions of Theorem~\ref{thm:main-E-denumerable} up to a multiplicative constant.

  Hence, Theorem~\ref{thm:main-E-denumerable} allows us to conclude the proof. The fact that all Dirac masses $\delta_{(n,y)}$ with
  $n\leq n_0$ belong to the domain of attraction follows from Corollary~\ref{cor:attraction-domain}.
\end{proof}

\section{Application to processes in continuous state space and discrete time}
\label{sec:continuous-state-space}

Discrete time Markov models in continuous state space and with absorption naturally arise in many applications. Examples of such
processes are given by perturbed dynamical systems, cf.\ e.g.\
\cite{FaureSchreiber2014,BerglundLandon2012,baudel-berglund-17,HinrichKolbEtAl2018}, or piecewise deterministic Markov processes when
one looks at the process at jump times only (see e.g.~\cite{azais-et-al-14}). We provide in Section~\ref{sec:two-side} a general
criterion applying to such processes with arbitrarily close to 1, state-dependent killing probability, and we give applications to
Euler schemes for diffusions absorbed at the boundary of a domain. In Section~\ref{sec:perturbed-dyn-syst}, we consider perturbed
dynamical systems in finite dimension. We first consider the case of unbounded domains with unbounded perturbation.
Subsection~\ref{sec:unbounded-perturbation} assumes that the perturbation has bounded density with respect to Lebesgue's measure and
Subsection~\ref{sec:singul-density} provides examples with perturbations with unbounded density. Finally, the case of bounded
perturbations is studied in Subsection~\ref{sec:bounded-perturbation}. Theorem~\ref{thm:perturbed-SD-intro} of the Introduction is
obtained as an application of the results of Section~\ref{sec:unbounded-perturbation}.
% The particular case of dynamical systems perturbed by a Gaussian noise is considered in Example~\ref{exa:expo-perturbed-SD} of
% Subsection~\ref{sec:unbounded-perturbation}. In this setting, it is shown that the perturbed dynamical system $X_{n+1}=f(X_n)+\xi_n$
% with $(\xi_i)_{i\in\ZZ_+}$ i.i.d.\ Gaussian, absorbed when it leaves any given measurable set $D$ of $\RR^d$ with positive Lebesgue
% measure, admits a quasi-stationary distribution as soon as $|x|-|f(x)|\rightarrow+\infty$ when $|x|\rightarrow+\infty$.

\subsection{Two sided estimates for processes with killing}
\label{sec:two-side}

Let $(Y_n,n\in\ZZ_+)$ be a Markov process evolving on a measurable state space $E\cup\{\d\}$ with transition kernel
$(Q(y,\cdot)_{y\in E\cup\{\d\}})$ such that $\d\notin E$ is absorbing (i.e. $Q(\d,\{\d\})=1$) satisfying a two-sided estimate (see for
instance~\cite{birkhoff-57,davies-simon-84,ChampagnatCoulibaly-PasquierEtAl2016}), which means that there exist a probability measure $\zeta$ on
$E$, a positive function $g:E\rightarrow (0,+\infty)$ and a constant $C>1$ such that, for all $y\in E$ and all measurable sets
$A\subset E$,
\begin{align}
  \label{eq:two-sided-Y}
  g(y)\zeta(A)\leq Q(y,A) \leq C g(y)\zeta(A).
\end{align}
Condition~\eqref{eq:two-sided-Y} is known to be satisfied for various models (see e.g.~\cite{BerglundLandon2012} or the references
in~\cite{ChampagnatCoulibaly-PasquierEtAl2016}). It is also well known (see~\cite{birkhoff-57,ChampagnatCoulibaly-PasquierEtAl2016})
that this implies that $Y$ admits a unique quasi-stationary distribution $\nu_{QSD}^Y$ for which the convergence
in~\eqref{eq:thm-main-zxzx} holds true for the total variation distance with geometric speed uniform with respect to the initial
distribution $\mu$ on $E$. Our aim is to generalize this result to processes obtained from $Y$ with additional killing (or
penalization).

More precisely, let $p:E\times E\rightarrow(0,1]$ be measurable and consider the Markov process $X$ evolving in $E\cup\{\d\}$ with
transition kernel $P(x,\cdot)_{x\in E\cup\{\d\}}$ defined by
\begin{align*}
  P(x,\mathrm dy)=
  \begin{cases}
    p(x,y)Q(x,\mathrm dy)+(1-p(x,y))\delta_\d(\mathrm dy)&\text{if }x\in E\\
    \delta_\d(\mathrm dy)&\text{if }x=\d.
  \end{cases}
\end{align*}
Observe that Condition~\eqref{eq:two-sided-Y} may not be satisfied by the kernel $P$ in cases where $\inf_{x,y\in E}p(x,y)=0$.
We also emphasize that the kernel $P$ generates a penalized semigroup of $(Y_n)_{n\in\Z_+}$, in the sense that, for any function $f:E\rightarrow \R_+$, all $x\in E$ and all $n\geq 1$, one has
    \[
    \E_x\left(f(X_n)\11_{n<\tau_\d}\right)=\E_x\left(p(x,Y_1)\cdots
    p(Y_{n-1},Y_n)\,f(Y_n)\11_{n<\tau^Y_\d}\right),
    \]
    where $\tau_\d$, resp.\ $\tau_\d^Y$, is the absorption time for $X$, resp.\ $Y$, in $\d$.

\begin{thm}
  \label{thm:two-side}
  Assume that there exists an increasing sequence $(L_k)_{k\geq 1}$ of measurable subsets of $E$
  such that $E=\cup_{k=1}^{+\infty}L_k$ and $\inf_{x,y\in L_k}p(x,y)>0$ for all $k\geq 1$.
  Then $X$ satisfies Assumption~(E) with $\varphi_1=1$ and $\varphi_2$ positive on $E$. In
  particular, $X$ admits a unique quasi-stationary distribution whose domain of attraction contains
  all probability measures on $E$.
\end{thm}

\begin{exa}
Typical examples of discrete-time Markov processes in continuous state space are given by Euler schemes for stochastic differential
equations. We consider the SDE $\mathrm dY_t=b(Y_t)\mathrm dt+\sigma(Y_t)\mathrm dB_t$ in $\RR^d$, with $b$ and $\sigma$ bounded measurable
on $\RR^d$ and $\sigma$ uniformly elliptic on $\RR^d$. Its standard Euler scheme with time-step $\delta$ is
the Markov chain $(X_n,n\geq 0)$ defined as
\begin{equation}
  \label{eq:Euler-scheme}
  X_{n+1}=b(X_n)\delta+\sqrt{\delta}\sigma(X_n)G_n,
\end{equation}
where $(G_n,n\geq 0)$ is an i.i.d.\ sequence of $\mathcal{N}(0,\text{Id})$ Gaussian variables in $\RR^d$. In the case of a SDE
absorbed at its first exit time of a bounded open connected domain $D\subset \RR^d$, the ``naive'' Euler scheme, constructed as above
with the additional rule that $X_n$ is immediately sent to $\d$ when $X_n\not\in D$, is not good in terms of weak error. Indeed, when
$X_n$ is close to the boundary of $D$ and $X_{n+1}$ remains in $D$, the path of the SDE $Y$ in the time interval
$[n\delta,(n+1)\delta]$ might have exited $D$. In this case, it is more efficient to construct the Brownian path that links $0$ to
$G_n$ on the time interval $[n\delta,(n+1)\delta]$ as a Brownian bridge $(\tilde{G}_t,t\in [n\delta,(n+1)\delta])$ such that
$\tilde{G}_{n\delta}=0$ and $\tilde{G}_{(n+1)\delta}=G_n$, so that one can approximate the path of the diffusion on this time
interval as
\[
\tilde{X}_{t}=b(X_n)(t-n\delta)+\sqrt{\delta}\sigma(X_n)\tilde{G}_t,\quad\forall t\in [n\delta,(n+1)\delta],
\]
and approximate the absorption event as $\{\exists t\in [n\delta,(n+1)\delta]:\tilde{X}_t\not\in D\}$. The corresponding Euler scheme
is thus obtained as the Markov chain $X$ as defined in~\eqref{eq:Euler-scheme} with the penalization $p(X_n,X_{n+1})=\PP(\exists t\in
[n\delta,(n+1)\delta]:\tilde{X}_t\not\in D)$. For a detailed presentation and study of this kind of modified Euler schemes, we refer
the reader to~\cite{manella-99,gobet-00,Gobet2001,buchmann-05}.

Using Theorem~\ref{thm:two-side}, we obtain the existence and convergence to a unique quasi-stationary distribution for this modified
Euler scheme. Indeed,~\eqref{eq:two-sided-Y} is satisfied for the naive Euler scheme with $\zeta$ equal to the restriction of
Lebesgue's measure to $D$ and a constant function $g$, thanks to the boundedness of the domain $D$, the uniform ellipticity of
$\sigma$ and the boundedness of $b$ and $\sigma$. In addition, it follows from the connectedness of the domain $D$, the uniform
ellipticity of $\sigma$ and the boundedness of $b$ and $\sigma$ that $\inf_{x,y\in K} p(x,y)>0$ for any compact subset $K$ of $D$.
\end{exa}

\begin{proof}[Proof of Theorem~\ref{thm:two-side}] 
 { For all $k\geq 1$, we define the set $K_k=\{x\in L_k\text{ s.t. }g(x)\geq 1/k\}$. Let
  $k_0$ be large enough so that $\zeta(K_{k_0})>0$. Then one has, for all $k\geq k_0$, all $x\in K_{k}$ and all
  measurable set $A\subset E$,
  \begin{align}
    \label{eq:two-side-pf1}
    \P_x(X_1\in A\cap K_{k_0})\geq g(x) \int_{A\cap K_{k_0}} p(x,y)\,\zeta(\mathrm dy)\geq  \frac{\zeta(K_{k_0})\inf_{u,v\in L_{k}}p(u,v)}{k}\,\nu(A\cap K_{k_0}),
  \end{align}
  where $\nu$ is the probability measure on $K_{k_0}$ defined by $\nu(A)=\zeta(A)/\zeta(K_{k_0})$. We fix $k\geq k_0$ large enough so
  that $C/k< \frac{\zeta(K_{k_0})\inf_{u,v\in L_{k_0}}p(u,v)}{k_0}$, where the constant $C$ is the one of~\eqref{eq:two-sided-Y}, and
  set $K=K_k$.}

{ Let us now check that Condition~(E) is satisfied with the above choices of $K$ and $\nu$ (extended by $0$ to
  $K_k\setminus K_{k_0}$), and with $\theta_1=C/k$ and $\theta_1<\theta_2<\frac{\zeta(K_{k_0})\inf_{u,v\in L_{k_0}}p(u,v)}{k_0}$.}

 {
   Setting $\varphi_1=1$, one has
  \begin{align*}
    & P_1\varphi_1(x)\leq 1,\ \forall x\in K,\\
    & P_1\varphi_1(x)\leq C\,g(x)\leq \theta_1 =\theta_1\varphi_1(x),\ \forall x\in E\setminus K,
  \end{align*}
  so that the first and third lines of Condition~(E2) are satisfied. Using Markov's property, one deduces
  from~\eqref{eq:two-side-pf1} that $\theta_2^{-n}\inf_{x\in K} \P_x(X_n\in K_{k_0})\rightarrow+\infty$ when
  $n\rightarrow+\infty$. Hence Lemma~\ref{lem:varphi2} implies that the second and fourth lines of Condition~(E2) are
  satisfied. It also implies that Condition~(E4) is satisfied.  Note also that the function $\varphi_2$ provided by
  Lemma~\ref{lem:varphi2} is positive on $E$ since $g$ is positive in~\eqref{eq:two-sided-Y}.}

{
  Moreover, for all $x\in E$, all $y\in K$ and all measurable set $A\subset E$,
  \begin{align*}
    \P_x(X_1\in A\cap K)&\leq Cg(x)\zeta(A\cap K)\leq \frac{Cg(x) k g(y)}{\inf_{K\times K} p}\,\int_{A\cap K} p(y,z)\,\zeta(\mathrm dz)\\
                        &\leq \frac{C\|g\|_\infty k}{\inf_{K\times K} p}\,\P_y(X_1\in A\cap K).
  \end{align*}
  We deduce from Proposition~\ref{lem:E2thenE} with $n_0=m_0=1$ that Conditions~(E1) and~(E3) are satisfied, which concludes the
  proof of Theorem~\ref{thm:two-side}.}
\end{proof}

\subsection{Perturbed dynamical systems}
\label{sec:perturbed-dyn-syst}

We consider the following perturbed dynamical system
\begin{align*}
X_{n+1}=f(X_n)+\xi_n,
\end{align*}
where $f:\R^d\rightarrow \R^d$ is a measurable function and $(\xi_n)_{n\in\N}$ is an i.i.d. sequence in $\R^d$. We assume that the
process evolves in a measurable set $D$ of $\RR^d$ with positive Lebesgue measure, meaning that it is immediately sent to
$\d\not\in\RR^d$ as soon as $X_n\not\in D$. We shall consider two situations below, where the random variables $\xi_n$ are unbounded
or almost surely bounded. In the unbounded case, different methods must be used depending on whether $\xi_n$ has a bounded density
with respect to Lebesgue's measure or not.

The same arguments would also work if $X_{n+1}=f(X_n)+\xi_n(X_n)$, where the sequence of random maps $(x\mapsto \xi_n(x))_{n\geq 0}$
are i.i.d. We leave the appropriate extensions of our assumptions and arguments to the reader.

\subsubsection{The case of unbounded perturbation with bounded density}
\label{sec:unbounded-perturbation}

We consider here the case where the random variables $\xi_n$ have support $\RR^d$.

\begin{prop}
  \label{prop:unbounded-perturbation}
  Assume that $f$ is locally bounded, that the law of $\xi_n$ has a boun\-ded density $g(x)$ with respect to Lebesgue's measure such that
  \begin{equation*}
    \inf_{|x|\leq R} g(x)>0,\quad\forall R>0,
  \end{equation*}
  and that there exists a locally bounded function $\varphi:\R^d\to [1,+\infty)$ such that
  $x\mapsto \E(\varphi(x+\xi_1))$ is locally bounded on $\R^d$ and such that
  \begin{equation}
    \label{eq:Lyap-SD-perturbed}
    \limsup_{|x|\rightarrow+\infty,\,x\in D} \frac{\E(\varphi(f(x)+\xi_1))}{\varphi(x)}=0.    
  \end{equation}
  Then Condition~(E) is satisfied with $\varphi_1=\varphi$ and $\varphi_2$ positive on $D$.
\end{prop}

Note that,  if $D$ is bounded, the last result is already a
consequence of the classical criterion based on~\eqref{eq:two-sided-Y}.
Before proving this result, let us give three applications.

\begin{exa}
  \label{exa:expo-perturbed-SD}
  If there exists $\alpha>0$ such that $\EE e^{\alpha|\xi_1|}<+\infty$ and if $|x|-|f(x)|\rightarrow+\infty$ when
  $|x|\rightarrow+\infty$, then Proposition~\ref{prop:unbounded-perturbation} applies. Indeed, choosing $\varphi(x)=\exp(\alpha|x|)$,
  we have
  \begin{align*}
    \frac{\EE\varphi(|f(x)+\xi_1|)}{\varphi(x)} & \leq e^{\alpha(|f(x)|-|x|)}\EE e^{\alpha|\xi_1|}\xrightarrow[|x|\rightarrow+\infty]{} 0.
  \end{align*}
  For instance, this covers the case of Gaussian perturbations, as stated in Theorem~\ref{thm:perturbed-SD-intro} in
    the introduction.
\end{exa}

\begin{exa}
  If there exists $p>0$ such that $\EE(\xi_1^p)<+\infty$ and if $|f(x)|=o(|x|)$ when
  $|x|\rightarrow+\infty$, then Proposition~\ref{prop:unbounded-perturbation} applies. Indeed, choosing $\varphi(x)=(1+|x|)^p$, we have
  \begin{align*}
     \frac{\EE\varphi(|f(x)+\xi_1|)}{\varphi(x)} & \leq \frac{(1+|f(x)|)^p}{(1+|x|)^p}\EE[(1+|\xi_1|)^p]\xrightarrow[|x|\rightarrow+\infty]{} 0.
  \end{align*}
\end{exa}

\begin{exa}
  If $\EE\log(1+|\xi_1|)<\infty$ and $|f(x)|\leq C|x|^{\varepsilon(x)}$ for some $C>0$ and some $\varepsilon(x)\rightarrow 0$ when
  $|x|\rightarrow+\infty$, then Proposition~\ref{prop:unbounded-perturbation} applies. Indeed, choosing
  $\varphi(x)=\log(e+|x|)$, we have
  \begin{align*}
    \frac{\EE\varphi(|f(x)+\xi_1|)}{\varphi(x)} & \leq \frac{\log (e+C)+\varepsilon(x)\log(e+|x|)}{\log(e+|x|)}+\frac{\EE\log(1+|\xi_1|)}{\log(e+|x|)}.
  \end{align*}
  % The inequality $|f(x)|\leq C|x|^{\varepsilon(x)}$ is true for example if $|f(x)|\leq C\exp\sqrt{\log(1+|x|)}$ for some constant
  % $C$.
\end{exa}

\begin{proof}[Proof of Proposition~\ref{prop:unbounded-perturbation}]
  We first prove Conditions~(E2) and~(E4) and conclude the proof with Proposition~\ref{lem:E2thenE}.

  \medskip
  \noindent\textit{Step 1. Conditions~(E2) and~(E4) are satisfied.}

  Let $K_1\subset D$ be a bounded measurable set with positive Lebesgue measure. Then, for all $x\in K_1$, denoting by $\lambda_d$
  the Lebesgue measure on $\R^d$,
  \[
  \P_x(X_1\in K_1)= \P(f(x)+\xi_1\in K_1)\geq \lambda_d(K_1)\,\inf_{u\in K_1+B(0,\sup_{K_1} |f|)} g(u)>0.
  \]
  Fix $\theta_2\in(0,\lambda_d(K_1)\,\inf_{u\in K_1+B(0,\sup_{K_1} |f|)}g(u)\,)$, we deduce that, for all $x\in K_1$,
  \[
  \theta_2^{-n}\inf_{x\in K_1}\P_x(X_n\in K_1)\geq \theta_2^{-n}\inf_{x\in K_1}\P_x(X_1\in K_1,\ldots, X_n\in
  K_1)\xrightarrow[n\rightarrow+\infty]{} +\infty.
  \]
  Fix $0<\theta_1<\theta_2$, and, using~\eqref{eq:Lyap-SD-perturbed}, consider a bounded subset $K\subset D$ containing $K_1$ and
  such that, for all $x\in D\setminus K$, $P_1\varphi(x)\leq \theta_1\varphi(x)$. Since $K$ is bounded, one has
  \[
    \inf_{x\in K} \P_x(X_1\in K_1)\geq \lambda_d(K_1)\,\inf_{u\in K_1+B(0,\sup_{K} |f|)} g(u)>0,
  \]
  so that 
  \begin{align*}
    \theta_2^{-n}\inf_{x\in K}\P_x(X_n\in K) & \geq \theta_2^{-n}\lambda_d(K_1)\,\inf_{u\in K_1+B(0,\sup_{K} |f|)} g(u) \inf_{x\in
      K_1}\P_x(X_{n-1}\in K_1) 
  \end{align*}
  and thus $\theta_2^{-n}\inf_{x\in K}\P_x(X_n\in K)$ converges to $+\infty$ when $n\rightarrow+\infty$. Lemma~\ref{lem:varphi2} then
  entail that Condition~(E4) is satisfied and that there exists a function $\varphi_2:D\rightarrow [0,1]$ such that
  $P_1\varphi_2(x)\geq\theta_2\varphi_2(x)$ for all $x\in D$ and such that $\inf_K \varphi_2>0$. In addition, for all
    $x\in D$, $\PP_x(X_1\in K)\geq \lambda_d(K)\inf_{u\in K-f(x)}g(u)>0$, so that $P_1\11_K(x)>0$. Hence, the function
    $\varphi_2$ of Lemma~\ref{lem:varphi2} also satisfies that $\varphi_2(x)>0$ for all $x\in E$.

  Setting $\varphi_1=\varphi$, we deduce that Conditions~(E2) and~(E4) are satisfied for the set $K$.

  \medskip
  \noindent\textit{Step 2. Comparison of transition probabilities.}

  Let us prove that Proposition~\ref{lem:E2thenE} applies with $n_0=m_0=1$. For all $x\in D$, we have
  \[
    \P_x(X_1\in \cdot\cap K)\leq \sup_{u\in\RR^d} g(u)\,\lambda_d(\cdot\cap K).
  \]
  Moreover, for all $y\in K$,
  \begin{align*}
    \P_y(X_1\in \cdot)&\geq \P(f(y)+\xi_1\in \cdot\cap K) \\
                  &\geq \inf_{u\in K+B(0,\sup_K |f|)} g(u)\,\lambda_d(\cdot\cap K).
  \end{align*}
  Hence, for all $x\in E$ and all $y\in K$,
  \[
    \P_x(X_1\in \cdot\cap K)\leq \frac{\sup_{\RR^d} g}{\inf_{K+B(0,\sup_K |f|)} g}\,\P_y(X_1\in\cdot).
  \]
  We deduce from Step~1 and Proposition~\ref{lem:E2thenE} that Condition~(E) is satisfied with the functions $\varphi_1$ and
  $\varphi_2$, which concludes the proof.
\end{proof}

\subsubsection{An example with unbounded perturbation and singular density}
\label{sec:singul-density}

The last result made strong use of the boundedness of $g$. Actually, our criteria also apply to perturbations with singular density.
We consider here the following example: assume that $f(x)=Ax+B$, where $A$ is an invertible $d\times d$ matrix and $B\in \RR^d$, and
that there exists $a>0$ such that the density $g$ of $\xi_n$ satisfies for some constant $C_g$
\begin{equation}
  \label{eq:hyp-g-pointue}
  g(x)\leq C_g\left(\frac{1}{|x|^{d-a}}\vee 1\right)\,\quad\forall x\in \RR^d.
\end{equation}
We have the following result.

\begin{prop}
  \label{prop:g-singular}
  Let $\|\cdot\|$ be a norm on $\RR^d$ and assume that
  \begin{equation}
    \label{eq:hyp-|A|<1}
    \sup_{x\in\RR^d\setminus\{0\}}\frac{\|Ax\|}{\|x\|}<1.    
  \end{equation}
  Assume also that $\EE e^{\alpha|\xi_1|}<\infty$ for some $\alpha>0$ and that
  \begin{equation*}
    \inf_{|x|\leq R} g(x)>0,\quad\forall R>0.
  \end{equation*}
  Then Condition~(E) is satisfied with $\varphi_1=\varphi$ and $\varphi_2$ positive on $D$.
\end{prop}

The proof of Proposition~\ref{prop:unbounded-perturbation} made use of Proposition~\ref{lem:E2thenE} with $n_0=m_0=1$. The proof of
Proposition~\ref{prop:g-singular} requires to apply Proposition~\ref{lem:E2thenE} with $n_0\geq 2$.

\begin{proof}
  The first step of the proof of Proposition~\ref{prop:unbounded-perturbation} remains valid taking $\varphi(x)=e^{\alpha\|x\|}$ for
  $\alpha>0$ small enough and using~\eqref{eq:hyp-|A|<1} and the equivalence of the norms $|\cdot|$ and $\|\cdot\|$ (the computation
  is similar to the one of Example~\ref{exa:expo-perturbed-SD}). So we only have to prove that~\eqref{eq:condition-E2thenE} is
  satisfy and apply Proposition~\ref{lem:E2thenE}.

  We define $n_0=\lceil d/a\rceil$ and we assume without loss of generality (reducing slightly $a$ if needed) that $n_0a>d$. We
  observe that
  \[
  X_{n_0}=A^{n_0} x+A^{n_0-1}(B+\xi_1)+\cdots+B+\xi_{n_0}.
  \]
  Using~\eqref{eq:hyp-g-pointue} and the fact that $\sup_{x\neq 0}\frac{|Ax|}{|x|}\leq C^2_{\|\cdot\|}$ where the constant
  $C_{\|\cdot\|}$ is such that $C^{-1}_{\|\cdot\|}|\cdot|\leq\|\cdot\|\leq C_{\|\cdot\|}|\cdot|$, the density $g_2$ of $A\xi_1+\xi_2$
  satisfies
  \begin{align}
    g_2(x) & =\frac{1}{|\text{det}A|}\int_{\RR^d} g(x-y)g(A^{-1}y)\mathrm dy \notag \\
    & \leq \frac{C_g^2}{|\text{det}A|}\int_{\{y:|A^{-1}y|\leq 1\}\cap
      B(x,1)}\frac{1}{|x-y|^{d-a}}\frac{1}{|A^{-1}y|^{d-a}}\mathrm dy +C_g\left(1+\frac{1}{|\text{det}A|}\right) \notag \\
    & \leq \frac{C_g^2 C_{\|\cdot\|}^{2(d-a)}}{|\text{det}A|}\int_{B(0,C^2_{\|\cdot\|})}\frac{1}{|x-y|^{d-a}}\frac{1}{|y|^{d-a}}\mathrm dy
    +C_g\left(1+\frac{1}{|\text{det}A|}\right) \notag \\
    &
    =\frac{C_g^2 C_{\|\cdot\|}^{2(d-a)}}{|\text{det}A|}\frac{1}{|x|^{d-2a}}\int_{B(0,C^2_{\|\cdot\|}/|x|)}\frac{1}{\left|\frac{x}{|x|}-u\right|^{d-a}}\frac{1}{|u|^{d-a}}\mathrm du
    +C_g\left(1+\frac{1}{|\text{det}A|}\right), \label{eq:preuve-g-unbounded}
  \end{align}
  where we made the change of variable $u=y/|x|$. 

  If $2a>d$ (i.e.\ if $n_0=2$), we can bound the integral in the right-hand side as follows:
  \begin{align*}
    \int_{B\left(0,\frac{C^2_{\|\cdot\|}}{|x|}\right)}\frac{1}{\left|\frac{x}{|x|}-u\right|^{d-a}}\frac{1}{|u|^{d-a}}\mathrm du & \leq
    C+2^d\int_{B\left(0,\frac{C^2_{\|\cdot\|}}{|x|}\right)\setminus 
      B(0,2)}\frac{1}{|u|^{2d-2a}}\mathrm du \\ & \leq C+\frac{C}{2a-d}\frac{1}{|x|^{2a-d}},
  \end{align*}
  where the constant $C$ may change from line to line. Therefore, $g_2$ is bounded if $2a>d$.

  Otherwise, if $2a<d$, the integral in the right-hand side of~\eqref{eq:preuve-g-unbounded} can be boun\-ded by the same integral over
  $\RR^d$ and thus it is uniformly bounded with respect to $x$, so $g_2$ is bounded by $C(1\vee 1/|x|^{d-2a})$. In this case, we can proceed similarly to
  bound the density $g_3$ of $A^2\xi_1+A\xi_2+\xi_3$, and prove by induction that the density $g_{n_0}$ of
  $A^{n_0-1}\xi_1+\cdots+\xi_{n_0}$ is bounded.

  We deduce that
  \[
  \P_x(X_{n_0}\in \cdot\cap K)\leq \sup_{u\in\RR^d} g_{n_0}(u)\,\lambda_d(\cdot\cap K).
  \]
  The end of the proof is the same as for Proposition~\ref{prop:unbounded-perturbation}, using Proposition~\ref{lem:E2thenE} with
  $m_0=n_0$.
\end{proof}

\subsubsection{Two examples with bounded perturbation}
\label{sec:bounded-perturbation}

The case where $\xi_1$ is a bounded random variable is more involved. To avoid complications, we will focus on the case where $\xi_n$
is a uniform random variable on the unit ball $B(0,1)$ of $\RR^d$. Extensions to different distributions are possible.

We start with the simpler case of bounded domain $D$ and contracting dynamical system $f$.

\begin{prop}
  \label{prop:xi-bounded-D-bounded}
  Assume that $D$ is a bounded, connected open set of $\RR^d$, that $f$ is continuous and satisfies $|f(x)-x|<1$ for all $x\in D$.
  Then Condition~(E) is satisfied.
\end{prop}

\begin{proof}
  Again, the proof makes use of the criterion of Proposition~\ref{lem:E2thenE}.
  \medskip

  \noindent\emph{Step 1. Construction and properties of the sets $K_\varepsilon$, $\varepsilon> 0$.}

  For all $\varepsilon>0$, let $K'_\varepsilon$ be the connected component of $\{x\in D:d(x,\partial D)\geq 2\varepsilon\}$ with
  larger Lebesgue measure and let
  \[
  K_\varepsilon:=\bigcup_{x\in K'_\varepsilon}\overline{B(x,\varepsilon)},
  \]
  which is a also a connected compact subset of $D$ with distance to $D^c$ larger than $\varepsilon$. For all $\delta>0$ and all
  $x,y\in K_\varepsilon$, we call a sequence $(x_0,x_1,\ldots,x_n)\in K_\varepsilon^{n+1}$ for some $n\in\NN$ a $\delta$-path linking
  $x$ to $y$ in $K_\varepsilon$ if $x_0=x$, $x_n=y$ and $|x_k-x_{k-1}|<\delta$ for all $1\leq k\leq n$. By construction, the set
  $K_\varepsilon$ satisfies that, for all $\delta>0$ and all $x,y\in K_\varepsilon$, there exists a $\delta$-path linking $x$ to $y$
  in $K_\varepsilon$. In addition, since $K_\varepsilon$ is compact, there exists an integer $n_{\varepsilon,\delta}$ depending only
  on $\varepsilon$ and $\delta$ such that, for all $x,y\in K_\varepsilon$, there exists a $\delta$-path in $K_\varepsilon$ linking
  $x$ to $y$ with length less than $n_{\varepsilon,\delta}$. For all $x\in K_\varepsilon$ and all
  $k\in\{1,\ldots,n_{\varepsilon,\delta}\}$ let us define
  \begin{multline*}
    K^{(k)}_{\varepsilon,\delta}(x)=\left\{y\in\RR^d:
      \exists x_1,\ldots,x_{k-1}\in K_\varepsilon,\ |x_\ell-x_{\ell-1}|<\delta\text{ for all } 1\leq\ell\leq k
    \right. \\
    \text{ with }x_0=x\text{ and }x_k=y \biggr\}.
  \end{multline*}
  Note that in general, $K^{(k)}_{\varepsilon,\delta}$ is not included in $K_\varepsilon$, but it is included in $D$ if
  $\delta<\varepsilon$. It follows from above that $K^{(n_{\varepsilon,\delta})}_{\varepsilon,\delta}(x)\supset K_\varepsilon$ for
  all $x\in K_\varepsilon$.

  Let us also prove that $\cup_{\varepsilon>0} K_\varepsilon= D$. Let $(x_n)_{n\geq 1}$ be a dense sequence in $D$ and for all $n\geq
  1$, let $r_n=d(x_n,\partial D)/2$. Since $D=\cup_{n\geq 1} B(x_n,r_n)$, there exists $n_0\geq 1$ such that $\cup_{1\leq n\leq n_0}
  B(x_n,r_n)$ has Lebesgue measure larger than $\lambda_d(D)/2$. Since $D$ is connected, there exists a continuous path in $D$
  linking $x_i$ to $x_j$ for all $1\leq i,j\leq n_0$. Since the distance between this path and $\partial D$ is positive,
  % (because $D$ is open and the path is compact), 
  there exists $\varepsilon>0$ small enough such that all the points $x_1,\ldots,x_{n_0}$ belong to
  the same connected component of $\{x\in D:d(x,\partial D)\geq 2\varepsilon\}$. We can assume without loss of generality that
  $\varepsilon<r_n/2$ for all $1\leq n\leq n_0$, so that this connected component actually contains $\cup_{1\leq n\leq n_0}
  B(x_n,r_n)$ and hence has the largest Lebesgue measure among all the connected components of $\{x\in D:d(x,\partial D)\geq
  2\varepsilon\}$. In particular, $K_\varepsilon$ contains $B(x_1,r_1)$ for all $\varepsilon$ small enough. Now, given any $x\in D$,
  there exists a path linking $x$ to $x_1$ in $D$. Since the distance between this path and $\partial D$ is positive, $x$ belongs to
  $K_\varepsilon$ for all $\varepsilon>0$ small enough. Hence, we have proved that $\cup_{\varepsilon>0} K_\varepsilon= D$ and that
  the family $(K_\varepsilon)_{\varepsilon>0}$ is non-increasing with respect to $\varepsilon>0$ when $\varepsilon$ is small enough.

  \medskip
  \noindent\emph{Step 2. Proof of Condition~\eqref{eq:condition-E2thenE} of Proposition~\ref{lem:E2thenE}.}

  For all $\varepsilon>0$, since $f$ is continuous,
  \[
  \delta_\varepsilon:=\left(1-\sup_{x\in K_\varepsilon}|f(x)-x|\right)\wedge \varepsilon>0.
  \]
  Hence, for all $x\in K_\varepsilon$,
  \begin{align}
    \label{eq:proof-step2-one-sided}
    \PP_x(X_1\in \cdot\cap B(x,\delta_\varepsilon))\geq c_d\lambda_d(\cdot\cap B(x,\delta_\varepsilon)),
  \end{align}
  for a positive constant $c_d$ only depending on the dimension of the space. In other words, for all $x\in K_\varepsilon$,
  \[
  \PP_x(X_1\in\cdot)\geq c_d\PP(x+U\in \cdot)
  \]
  where $U$ is a uniform random variable on $B(0,\delta_\varepsilon)$. Hence, defining the Markov chain $Y_n=Y_0+U_1+\ldots+U_n$
  where $U_i$ are i.i.d.\ uniform random variable on $B(0,\delta_\varepsilon)$, we deduce that
  \begin{equation}
    \label{eq:key-bounded-perturbation}
    \PP_x(X_k\in\cdot) \geq c_d^k\PP_x(Y_1,\ldots,Y_{k-1}\in K_\varepsilon\text{ and }Y_k\in \cdot),\quad\forall x\in K_\varepsilon,\ \forall k\in\NN.
  \end{equation}
  In view of Step 1, the following Lemma~\ref{lem:key-lemma-bounded-perturb} about the process $Y$ implies that there exists a
  constant $c'>0$ such that
  \begin{equation}
    \label{eq:step-2-g-unbounded}
    \PP_x(X_{n_{\varepsilon,\delta_\varepsilon/3}}\in\cdot)\geq c'\lambda_d(\cdot\cap K_\varepsilon),\quad\forall x\in K_\varepsilon.
  \end{equation}
  Since the law of $X_1$ is dominated by the Lebesgue measure independently of $X_0$, we have proved that, for all
  $\varepsilon>0$,~\eqref{eq:condition-E2thenE} is satisfied for $K=K_\varepsilon$, $n_0=1$ and
  $m_0=n_{\varepsilon,\delta_\varepsilon/3}$. This concludes Step~2 of the proof.

  \begin{lem}
    \label{lem:key-lemma-bounded-perturb}
    For all $1\leq k\leq n_{\varepsilon,\delta_\varepsilon/3}$, there exists a constant $c'_k>0$ such that, for all $x\in K_\varepsilon$,
    \begin{equation}
      \label{eq:key-lemma-bounded-perturb}
      \PP_x(Y_1,\ldots,Y_{k-1}\in K_\varepsilon\text{ and }Y_k\in \cdot)\geq c'_k\lambda_d\left(\cdot\cap K^{(k)}_{\varepsilon,\delta_\varepsilon/3}(x)\right),
    \end{equation}
    where $\lambda_d$ is Lebesgue's measure on $\RR^d$.
  \end{lem}
  \medskip

  \noindent\emph{Step 3. Proof of (E2) and (E4).}

  Fix $\varepsilon_0 > 0$ such that $K_{\varepsilon_0}$ is non-empty and $(K_\varepsilon)_{\varepsilon\in(0,\varepsilon_0]}$ is
  non-increasing. It follows from the definition of $K_{\varepsilon}$ that 
  $\inf_{x\in K_{\varepsilon_0}}\lambda_d(K_{\varepsilon_0}\cap B(x,\delta_{\varepsilon_0})) > 0$.
  Fixing
  \[
  \theta_2 < 4\wedge\left\{c_d \inf_{x\in K_{\varepsilon_0}}\lambda_d(K_{\varepsilon_0}\cap
    B(x,\delta_{\varepsilon_0}))\right\},
  \]
  we deduce from~\eqref{eq:proof-step2-one-sided} that
  \begin{align}
    \label{eq:pf-g-unbounded-E4}
     \lim_{n\rightarrow+\infty} \theta_2^{-n}\inf_{x\in K_{\varepsilon_0}}\PP_x(X_n\in K_{\varepsilon_0})=+\infty.
  \end{align}
  Since the law of $X_1$ is dominated by the Lebesgue measure and $D=\cup_{0<\varepsilon\leq\varepsilon_0} K_\varepsilon$, there
  exists $\varepsilon_1\in(0,\varepsilon_0]$ small enough such that
  \[
  \sup_{x\in D}\P_x(X_1\in D\setminus K_{\varepsilon_1})\leq \theta_2/4.
  \]
  Hence, the function
  \[
  \varphi_1:x\in D\mapsto
  \begin{cases}
    1&\text{if }x\in K_\varepsilon,\\
    4/\theta_2&\text{if }x\in D\setminus K_{\varepsilon_1},
  \end{cases}
  \]
  satisfies $P_1\varphi_1(x)\leq 2\leq (\theta_2/2)\varphi_1(x)$ for all $x\in D\setminus K_{\varepsilon_1}$. Hence the first and third
  lines of Condition~(E2) are satisfied with $\theta_1=\theta_2/2$ and $K=K_{\varepsilon_1}$. 

  We also deduce from~\eqref{eq:step-2-g-unbounded},~\eqref{eq:pf-g-unbounded-E4}, the fact that $K_{\varepsilon_0}\subset
  K_{\varepsilon_1}$ and Markov's property that
  \[
  \lim_{n\rightarrow+\infty} \theta_2^{-n}\inf_{x\in K_{\varepsilon_1}}\PP_x(X_n\in K_{\varepsilon_1})=+\infty.
  \]
  Hence, it follows from Lemma~\ref{lem:varphi2} that (E4) is satisfied with $K=K_{\varepsilon_1}$ and that there exists a function
  $\varphi_2$ satisfying the conditions of (E2) with $\theta_2$ defined above and $K=K_{\varepsilon_1}$.

  Therefore, the result follows from Step~2 and Proposition~\ref{lem:E2thenE} with $K=K_{\varepsilon_1}$, $n_0=1$ and
  $m_0=n_{\varepsilon_1,\delta_{\varepsilon_1}/3}$.
\end{proof}

\begin{proof}[Proof of Lemma~\ref{lem:key-lemma-bounded-perturb}]
  We prove this result by induction over $k$. Since $Y_1=x+U_1$ is uniform in $B(x,\delta_\varepsilon)$, the case $k=1$ is clear since
  $K^{(1)}_{\varepsilon,\delta_\varepsilon/3}=B(x,\delta_\varepsilon/3)\subset B(x,\delta_\varepsilon)$.

  So assume that~\eqref{eq:key-lemma-bounded-perturb} is satisfied for some $1\leq k\leq n_{\varepsilon,\delta_\varepsilon/3}-1$ and let us prove
  it for $k+1$. Let $A\subset \RR^d$ be measurable. Using~\eqref{eq:key-lemma-bounded-perturb} for $k$ and the fact that $Y_{k+1}$ is
  uniform in $B(Y_k,\delta_\varepsilon)$ conditionally on $Y_k$, we have
  \begin{multline*}
    \PP_x(Y_1,\ldots,Y_{k}\in K_\varepsilon,\ Y_{k+1}\in A) \\
    \begin{aligned}
      & \geq \PP_x\left(Y_1,\ldots,Y_{k-1}\in K_\varepsilon,\ Y_k\in K^{(k)}_{\varepsilon,\delta_\varepsilon/3}(x)\cap
        K_\varepsilon,
      \ Y_{k+1}\in A\cap B(Y_k,\delta_\varepsilon)\right) \\
      & \geq \frac{c'_k}{\lambda_d(B(0,\delta_\varepsilon))}\int_{
          K^{(k)}_{\varepsilon,\delta_\varepsilon/3}(x)\cap K_\varepsilon}\mathrm dy\int_{A\cap
        B(y,\delta_\varepsilon)}\mathrm dz \\
      & =\frac{c'_k}{\lambda_d(B(0,\delta_\varepsilon))}\int_A \lambda_d\left\{
        K^{(k)}_{\varepsilon,\delta_\varepsilon/3}(x)\cap K_\varepsilon\cap
        B(z,\delta_\varepsilon)\right\}\,\mathrm dz \\
      & \geq\frac{c'_k}{\lambda_d(B(0,\delta_\varepsilon))}\int_{A\cap K^{(k+1)}_{\varepsilon,\delta_\varepsilon/3}(x)}\lambda_d\left\{
        K^{(k)}_{\varepsilon,\delta_\varepsilon/3}(x)\cap K_\varepsilon\cap
        B(z,\delta_\varepsilon)\right\}\,\mathrm dz,
    \end{aligned}
  \end{multline*}
  where the third equality follows from Fubini's theorem. 

  Now, for all $z\in K^{(k+1)}_{\varepsilon,\delta_\varepsilon/3}(x)$, there exists a path $x_0=x,x_1,\ldots, x_k\in K_\varepsilon$
  such that $|x_\ell-x_{\ell-1}|<\delta_\varepsilon/3$ for all $1\leq \ell\leq k$ and $|x_k-z|<\delta_\varepsilon/3$. By definition
  of $K_\varepsilon$, there exists $y\in K_\varepsilon$ such that $x_{k-1}\in B(y,\varepsilon)\subset K_\varepsilon$. Let $y'$ be the
  unique point such that $|y'-x_{k-1}|=\delta_\varepsilon/6$ of the half-line with initial point $x_{k-1}$ and containing $y$. Then
  $B(y',\delta_\varepsilon/6)\subset K_\varepsilon$. Since $|x_k-z|<\delta_\varepsilon/3$ and $|x_{k-1}-x_k|<\delta_\varepsilon/3$,
  we also have $B(y',\delta_\varepsilon/6)\subset B(z,\delta_\varepsilon)$. In addition, for all $y''\in B(y',\delta_\varepsilon/6)$,
  the path $x_0=x,x_1,\ldots,x_{k-1},y''$ lies in $K_\varepsilon$ and has distance between consecutive point smaller than
  $\delta_\varepsilon/3$. Therefore, $B(y',\delta_\varepsilon/6)\subset K^{(k)}_{\varepsilon,\delta_\varepsilon/3}(x)$. We conclude
  that, for all $z\in K^{(k+1)}_{\varepsilon,\delta_\varepsilon/3}(x)$, 
  \[
  \lambda_d\left\{
    K^{(k)}_{\varepsilon,\delta_\varepsilon/3}(x)\cap K_\varepsilon\cap B(z,\delta_\varepsilon)\right\}\geq
  \lambda_d(B(0,\delta_\varepsilon/6)).
  \]
  Hence
  \[
  \PP_x(Y_1,\ldots,Y_{k}\in K_\varepsilon,\ Y_{k+1}\in A)\geq c'_{k+1} \lambda_d\left(A\cap
    K^{(k+1)}_{\varepsilon,\delta_\varepsilon/3}(x)\right)
  \]
  for a positive constant $c'_{k+1}$.
\end{proof}

The general case of dynamical systems with bounded perturbations raises several additional difficulties. We illustrate two of them
with the next example in dimension 1. We consider the Markov process in $D=(0,+\infty)$ defined as
\[
X_0\in (0,+\infty),\quad X_{n+1}=\alpha X_n-\frac{1}{1+X_n}+\xi_n,\quad\forall n\geq 0
\]
where $\alpha\in(0,1)$ and $\xi_n$ are i.i.d.\ with uniform distribution on $[-1,1]$ and the process is immediately sent to the cemetery
point $\d$ when it leaves $D$. The first difficulty comes from the fact that
\[
\PP_x(X_1>0)=\left[1-\left(\frac{1}{1+x}-\alpha x\right)\right]\vee 0\xrightarrow[x\rightarrow 0+]{} 0,
\]
which means that the probability of immediate absorption converges to 1 when $x$ approaches the boundary of $D$. The second
difficulty comes from the fact that $|f(x)-x|$ is unbounded on $D$ (in contrast with Proposition~\ref{prop:xi-bounded-D-bounded}).
This example is covered by the following general result.

\begin{prop}
  \label{prop:ex2-bounded-perturbation}
  Assume that $X_{n+1}=f(X_n)+\xi_n$ with $D=(0,+\infty)$, $\xi_n$ i.i.d.\ uniform on $[-1,1]$, $f$ continuous and there exists
  $x^*\in D$ such that 
  \[
  (0,x^*)=\left\{x\in D:|f(x)-x|<1\right\}\quad\text{and}\quad[x^*,+\infty)=\left\{x\in D:f(x)+1\leq x\right\}.
  \]
  Then Condition~(E) is satisfied.
\end{prop}

\begin{proof}
  Fix $K_0\subset (0,x^*)$ a closed interval with non-empty interior. As in the proof of
  Proposition~\ref{prop:xi-bounded-D-bounded}, using in
  particular~(\ref{eq:key-bounded-perturbation})
  and~(\ref{eq:key-lemma-bounded-perturb}), there exists $n_0\geq 1$ and $c_0>0$ 
  such that, for all $x\in K_0$,
  \[
    \PP_x(X_{n_0}\in \cdot)\geq c_0\lambda_1(\cdot\cap K_0).      
  \]
  Hence there exists a constant $\theta_2\in(0,1)$ such that
  \begin{align}
    \label{eq:proof-prop-eq1}
    \theta_2^{-n}\inf_{x\in K_0}\P_x(X_n\in K_0)\xrightarrow[n\rightarrow+\infty]{} +\infty.
  \end{align}
  
  Fix now $\theta_1<\theta_2$ and $K\subset (0,x^*)$ a closed interval such that $K_0\subset K$ and
  \[
  \lambda_1\left\{(0,x^*)\setminus K\right\}\leq \frac{\theta_1}{M},
  \]
  where
  \[
  M:=\frac{2(1+e^{(x^*+2)/\theta_1})}{\theta_1}.
  \]
  As above, there exists $n_1\geq 1$ and $c_1>0$ 
  such that, for all $x\in K$,
  \[
    \PP_x(X_{n_1}\in \cdot)\geq c_1\lambda_1(\cdot\cap K).      
  \]
  In particular, $\inf_{x\in K}\PP_x(X_{n_1}\in K_0)>0$, so that, using Markov property and~\eqref{eq:proof-prop-eq1}, we deduce that
  \[
  \theta_2^{-n}\inf_{x\in K}\P_x(X_n\in K)\xrightarrow[n\rightarrow+\infty]{} +\infty.
  \]
  Using Lemma~\ref{lem:varphi2}, we deduce that there exists a function $\varphi_2$ satisfying the conditions of~(E2)
  and that~(E4) is satisfied. For all $x\in D$, let
  \[
  \varphi_1(x)=
  \begin{cases}
    1 & \text{if }x\in K, \\
    M & \text{if }x\in (0,x^*)\setminus K, \\
    e^{x/\theta_1} & \text{if }x\geq x^*.
  \end{cases}
  \]
  For $x\geq x^*$, using the fact that the density of $X_1$ on $D$ with respect to Lebesgue measure is bounded by $\frac{1}{2}\11_D$
  for all value of $X_0$, we have
  \begin{align*}
    P_1\varphi_1(x) & \leq \EE_x(e^{X_1/\theta_1}\11_{X_1\geq x^*})+\PP_x(X_1\in K)+M\PP_x(X_1\in (0,x^*)\setminus K) \\
    & \leq \EE_x(e^{X_1/\theta_1})+\frac{M}{2}\lambda_1\left\{(0,x^*)\setminus K\right\} \\
    & \leq \varphi_1(x) e^{(f(x)-x)/\theta_1}\EE_x e^{\xi_1/\theta_1}+\frac{\theta_1}{2} \\
    & \leq \varphi_1(x) e^{-\theta_1^{-1}}\frac{e^{\theta_1^{-1}}-e^{-\theta_1^{-1}}}{2\theta_1^{-1}}+\frac{\theta_1}{2}\varphi_1(x)
    \leq \theta_1\varphi_1(x). 
  \end{align*}
  For $x\in (0,x^*)\setminus K$, since $f(x)+\xi_1\leq x+2\leq x^*+2$,
  \begin{align*}
    P_1\varphi_1(x) & \leq \PP_x(X_1\in K)+e^{(x^*+2)/\theta_1}\PP_x(X_1\geq x^*)+M\PP_x(X_1\in (0,x^*)\setminus K) \\
    & \leq 1+e^{(x^*+2)/\theta_1}+\frac{M}{2}\lambda_1\left\{(0,x^*)\setminus K\right\} \\
    & \leq M\left(\frac{\theta_1}{2}+\frac{\theta_1}{2M}\right)  \leq \theta_1\varphi_1(x).
  \end{align*}
  Since $P_1\varphi_1(x)$ is clearly bounded for $x\leq x^*$, we have proved~(E2).

  To conclude, it remains to observe that~\eqref{eq:condition-E2thenE} can be deduced for
  $n_0=1$ and $m_0$ large enough exactly as in the proof of Proposition~\ref{prop:xi-bounded-D-bounded}. Hence the result follows from
  Proposition~\ref{lem:E2thenE}.
\end{proof}

\section{Irreducible processes in discrete state space and discrete time}
\label{sec:discrete-state-space}

The theory of $R$-positive matrices is a powerful tool to study absorbed Markov processes in discrete time and
space~\cite{FerrariKestenMartinez1996}. The goal of Section~\ref{sec:R-positive} is to show that our criteria allow to recover
the results on convergence to quasi-stationarity of this theory. We then study in Section~\ref{sec:ex-GW} a class of discrete Markov
chains in discrete time to which criteria based on $R$-positive matrices do not apply easily.

\subsection{$R$-positive matrices}
\label{sec:R-positive}

We consider a Markov chain $(X_n,n\in\ZZ_+)$ in a countable state space $E\cup\{\d\}$with $\d\not\in E$ an absorbing point and with
irreducible transition probabilities in $E$, i.e.\ such that for all $x,y\in E$, there exists $n=n(x,y)\geq 1$ such that
$\P_x(X_n=y)>0$. In this case, the most general criterion for existence and convergence to a quasi-stationary distribution is
provided in~\cite{FerrariKestenMartinez1996}. In this paper, the authors obtain a convergence result similar to the one of
Theorem~\ref{thm:main} restricted to Dirac initial distributions, and the pointwise convergence to $\eta$ as in Theorem~\ref{thm:eta},
using the theory of R-positive matrices. In this section, we show how our criterion allows to recover these results, providing in
addition the several refinements of Section~\ref{sec:main-discrete-time} (including the characterization of a non-trivial subset of
the domain of attraction, the convergence of~\eqref{eq:thm-main-zxzx} for unbounded functions $f$ and a stronger convergence to $\eta$).

We denote by $P$ the transition matrix of the chain $(X_n,n\in\ZZ_+)$ and we assume that the absorption time $\tau_\d$ is almost
surely finite. Without loss of generality, we will assume that the process is aperiodic, meaning that $\PP_x(X_n=y)>0$ for all
$x,y\in E$ provided $n$ is large enough; the extension to general periodic processes is routine, as observed
in~\cite{FerrariKestenMartinez1996} (see also~\cite{ChampagnatVillemonais2022b} on this topic in our general setting).

\begin{prop}
  \label{prop:comparions-with-FKM96}
  The assumptions of \cite[Theorem~1]{FerrariKestenMartinez1996} imply Assumption~(E).
\end{prop}

\begin{proof}
 Since $E$ is finite or countable and because of
  the irreducibility assumption, it is known~\cite{VereJones1967} that the limit
  \begin{align}
    \label{eq:R-def}
    \frac{1}{R}:=\lim_{n\rightarrow+\infty}\P_x(X_n=y)^{1/n}
  \end{align}
  exists with $1\leq R<\infty$, and is independent of $x,y\in E$. Using~\cite[Lemma~1]{FerrariKestenMartinez1996}, the assumptions of \cite[Theorem~1]{FerrariKestenMartinez1996} can be
  stated as follows: there exist a non-empty set $K\subset E$ and $x_0\in K$ such that
  \begin{description}
  \item[\textmd{(a)}] there exist $\varepsilon_0>0$ and a constant $C_1$ such that, for all $x\in K$
    and all $n\geq 0$,
    \begin{align*}
      \P_x(n<\sigma_K\wedge \tau_\d)\leq C_1(R+\varepsilon_0)^{-n},
    \end{align*}
    where $\sigma_K$ is the first return time in $K$
    \[
    \sigma_K:=\inf\{n\geq 1,X_n\in K\}.
    \]
  \item[\textmd{(b)}] there exists a constant $C_2$ such that, for all $x\in K$ and $n\geq 0$,
    \begin{align*}
      \P_x(n<\tau_\d)\leq C_2\P_{x_0}(n<\tau_\d);
    \end{align*}
  \item[\textmd{(c)}] there exist $n_0\geq 0$ and a constant $C_3>0$ such that, for all $x\in K$,
    \begin{align*}
      \P_x( T_{\{x_0\}}\leq n_0)\geq C_3,
    \end{align*}
    where we recall that $ T_L:=\inf\{n\in\ZZ_+:X_n\in L\}$ for all $L\subset E$.
  \end{description}
  \medskip

  Let us first prove (E1). By aperiodicity and irreducibility, there exists $m_1\geq 1$ such that, for all $n\geq m_1$,
  $\P_{x_0}(X_n=x_0)>0$. Combining this with (c), the Markov Property entails that, for all $x\in K$,
  \begin{align*}
    \P_{x}(X_{n_0+m_1}=x_0)\geq C_3\min_{m_1\leq k\leq n_0+m_1}\P_{x_0}(X_k=x_0).
  \end{align*}
  This is (E1) with $\nu=\delta_{x_0}$ and $n_1=n_0+m_1$.

  We now prove (E2) and (E4). Condition~(a) implies that
  \begin{align*}
    \left(R+\frac{\varepsilon_0}{2}\right)\sup_{y\in K}
    \E_y\left[\11_{1<\tau_\d}\E_{X_1}\left(\left(R+\frac{\varepsilon_0}{2}\right)^{ T_K\wedge\tau_\d}\right)\right]=\sup_{y\in
      K}\E_y\left[\left(R+\frac{\varepsilon_0}{2}\right)^{\sigma_K\wedge\tau_\d}\right]<\infty.
  \end{align*}
  For all $x\in E\setminus K$, the irreducibility assumption implies that there exist $y\in K$ and $n=n(x,y)\geq 1$ such that
  $\P_y(X_n=x\text{ and }n<\sigma_K)>0$. By Markov's property,
  \begin{align*}
    \E_y\left[\left(R+\frac{\varepsilon_0}{2}\right)^{\sigma_K\wedge\tau_\d}\right]\geq \P_y(X_n=x\text{ and }n<\sigma_K)\E_x\left[\left(R+\frac{\varepsilon_0}{2}\right)^{\sigma_K\wedge\tau_\d}\right].
  \end{align*}
  Since $\sigma_K= T_K$ almost surely under $\P_x$ for $x\in E\setminus K$, Lemma~\ref{lem:varphi1} provides a function $\varphi_1$
  satisfying the conditions of~(E2), with $\theta_1:=(R+\frac{\varepsilon_0}{2})^{-1}$. According
  to~\cite[(1.16)]{FerrariKestenMartinez1996}, which holds true under their assumption
  by~\cite[Theorem~1]{FerrariKestenMartinez1996}, and setting $\theta_2=\left(R+\frac{\varepsilon_0}{3}\right)^{-1}$, one has
  \begin{align*}
    \lim_{n\rightarrow+\infty}\theta_2^{-n}\P_{x_0}(X_{n}=x_0)=+\infty.
  \end{align*}
  Using Markov's property, Condition~(c) immediately entails that
  \begin{align*}
    \lim_{n\rightarrow+\infty}\theta_2^{-n}\inf_{x\in K}\P_{x}(X_{n}\in K)=+\infty.
  \end{align*}
  Using Lemma~\ref{lem:varphi2}, we deduce that there exists a function~$\varphi_2:E\rightarrow[0,1]$ satisfying the
  conditions of~(E2) and that~(E4) holds true. This concludes the proof of (E2) and (E4).

  To conclude, Conditions~(b) and~(E1) imply, for all $n\geq 0$,
  \begin{align*}
    \inf_{y\in K}\P_{y}(n<\tau_\d)&\geq \inf_{y\in K}\P_{y}(n+n_1<\tau_\d)\geq  c_1\P_{x_0}(n<\tau_\d)\geq \frac{c_1}{C_2}\sup_{y\in K}\P_y(n<\tau_\d).
  \end{align*}
  This proves (E3) and concludes the proof of Proposition~\ref{prop:comparions-with-FKM96}.
\end{proof}

\begin{rem}
  One can actually prove that, in the particular case of a discrete state space $E$ and aperiodic and irreducible transition probability on $E$,
  Assumption (E) is equivalent to the Conditions (a), (b) and (c) of~\cite{FerrariKestenMartinez1996}. Besides the additional
  properties provided in Section~\ref{sec:main-discrete-time}, one of our main contribution in this particular setting is to provide
  a more tractable criterion. Indeed, the use of Lyapunov type functions has the advantage to be quite flexible.
%   {\color{gray} This is illustrated
%  in the next subsection, with an application to population processes, extending to the multi-dimensional case some models studied
%  in~\cite{gosselin-01}. The direct application of~\cite{FerrariKestenMartinez1996} to this model is more difficult (they are even
%  qualified as ``impractical for such models of biological population extinction'' in~\cite[p. 262]{gosselin-01}) and do not
%  extend as easily by domination arguments (as, for example, in Theorem~\ref{thm:GW-dominated-bis} below).}
\end{rem}

\subsection{Application to the extinction of biological populations dominated by Galton-Watson processes}
\label{sec:ex-GW}

In this section, we show how our criteria can be applied to general
population processes dominated by population-dependent Galton-Watson
processes. In particular, we refine existing results for the classical
multi-type Galton-Watson process.

More precisely, we consider an aperiodic and irreducible Markov population process
$(Z_n)_{n\in\N}$ on $\Z^d_+=E\cup\{\d\}$ absorbed at $\d=0$ such that, for all $n\geq 0$,
\begin{align}
  \label{eq:GW-multitype}
  \|Z_{n+1}\|\leq\sum_{i=1}^{|Z_n|}\xi^{(Z_n)}_{i,n},
\end{align}
where $\|\cdot\|$ is a norm on $\R^d$ and $|z|=z_1+\ldots+z_d$ for all $z\in\Z_+^d$ and, for all
$n\geq 0$, the nonnegative random variables $\xi^{(Z_n)}_{1,n},\ldots, \xi^{(Z_n)}_{|Z_n|,n}$ are assumed
independent (but not necessarily identically distributed) given $Z_n$ and the families
  $(\xi^{(z)}_{i,n},z\in\Z_+^d,1\leq i\leq |z|)$ are i.i.d.\ for $n\in\Z_+$.

We assume that 
\begin{align}
  \label{eq:hyp-subcritical}
  \E\left(\sum_{i=1}^{|z|}\xi^{(z)}_{i,n}\right)\leq m \|z\|,\quad\forall z\in\Z_+^d\text{\ such that\ }|z|\geq n_0,
\end{align}
for some $m<1$ and $n_0\in\N$. This means that the population size has a tendency to
decrease (in mean) when it is too large. This also implies that $\tau_\d<\infty$ a.s.

In the following theorem, $R>0$ is the limiting value defined in~\eqref{eq:R-def}.

\begin{thm}
  \label{thm:GW-dominated-bis}
  Assume that $(Z_n,n\in\ZZ_+)$ is aperiodic irreducible, that it satisfies the
  assumptions~\eqref{eq:GW-multitype} and~\eqref{eq:hyp-subcritical} and that, for some
  $q_0>\frac{\log R}{\log(1/m)}\vee 1$,
  \begin{align*}
    \sup_{z\in\Z_+^d,\ 1\leq i\leq|z|}\E[(\xi^{(z)}_{i,1})^{q_0}]<\infty,
  \end{align*}
  Then Condition~(E) holds true with $\varphi_1(x)=\|x\|^q$, for all
  $q\in\left(\frac{\log R}{\log(1/m)}\vee 1,q_0\right]$.
\end{thm}

\begin{rem}
  This result easily applies if $\sup_{z\in\Z_+^d,\ 1\leq i\leq|z|}\E[(\xi^{(z)}_{i,n})^q]<\infty$ for all $q>0$. In other cases, we
  need an upper bound for $R>0$ to check the assumptions of Theorem~\ref{thm:GW-dominated-bis}. For instance, one may use the fact
  that $R\leq 1/\sup_{z\in\Z_+^d}\P_z(Z_1=z)$. One may also use Lyapunov techniques, in the same spirit as in
  Section~\ref{prop:lambda0-bound} for diffusion processes.
\end{rem}

\begin{rem}
  A particular case of application of the above theorem is when $Z$ is obtained from a Galton-Watson multi-type process (see below
  for a more precise definition) with additional population-dependent death rates. For example, one can assume that additional death
  events may affect a fraction of the population, modelling global death events. In this case, compared to the Galton-Watson case,
  the independence between the progeny of individuals breaks down. Another situation covered by the above result
    is the case where the domain of absorption of $Z$ is a larger set than $0$, for example the process may be absorbed when it
    reaches one edge of $\Z^d_+$ (i.e.\ when one type disappears). Another typical application of Theorem~\ref{thm:GW-dominated-bis}
  is the case of population-depen\-dent Galton-Watson processes, i.e.\ of processes such that, given $Z_n$, $Z_{n+1}$ is the sum of
  $|Z_n|$ i.i.d.\ random variables whose law may depend on $Z_n$. In this situation, Theorem~\ref{thm:GW-dominated-bis} and its
  consequences stated in Section~\ref{sec:main-discrete-time} generalize the results of~\cite{gosselin-01} to the multi-type
  situation and provides finer results on the domain of attraction of the minimal quasi-stationary distribution. The reducible cases
  considered in~\cite{gosselin-01} can also be recovered using the criterion of Theorem~\ref{thm:three-comm-class-1} in
  Section~\ref{sec:three-com-class} or the criteria of~\cite{ChampagnatVillemonais2022}. Of course, the above cases may be combined.
\end{rem}

  Let us now consider the case of multi-type Galton-Watson processes.  A Mar\-kov process $(Z_n,n\in\ZZ_+)$ evolving in
  $\Z^d_+=E\cup\{\d\}$ absorbed at $\d=0$ is called a Galton-Watson process with $d$ types if, for all
  $n\geq 0$ and all $i\in\{1,\ldots,d\}$,
\begin{align}
  \label{eq:GW-multitype-bis}
  Z_{n+1}^i=\sum_{k=1}^{d}\sum_{\ell=1}^{Z^k_n} \zeta^{(n,\ell)}_{k,i},
\end{align}
where the random variables $(\zeta^{(n,\ell)}_{k,1},\ldots,\zeta^{(n,\ell)}_{k,d})_{n,\ell,k}$
in $\ZZ_+$ are assumed independent and such that, for all $k\in\{1,\ldots,d\}$,
$(\zeta^{(n,\ell)}_{k,1},\ldots,\zeta^{(n,\ell)}_{k,d})_{n,\ell}$ is an i.i.d. family. We
define the matrix $M=(M_{k,i})_{1\leq k,i\leq d}$ of mean offspring as
\begin{align*}
  M_{k,i}=\E(\zeta^{(n,\ell)}_{k,i}),\quad\forall k,i\in\{1,\ldots,d\},
\end{align*}
and assume that $M_{k,i}<+\infty$ and that there exists $n\geq 1$ such that
$[M^n]_{k,i}>0$ for all $k,i\in\{1,\ldots,d\}$.

Using the classical formalism of~\cite{harris-63}, we consider a positive right eigenvector $v$ of the matrix $M$ of
mean offspring and we denote by $\rho(M)$ its spectral radius. The sub-critical case corresponds to $\rho(M)<1$. It is
well-known~\cite{JoffeSpitzer1967} (see also~\cite{heathcote-seneta-verejones-67,athreya-ney-72}) that this implies the
existence of a quasi-stationary distribution whose domain of attraction contains all Dirac measures (a so-called
\textit{Yaglom limit} or \textit{minimal quasi-stationary distribution}). The authors also prove that
$\nu_{QSD}(|\cdot|)<\infty$ if and only if $\E[|Z_1|\log(|Z_1|)\mid Z_0=(1,\ldots,1)]<\infty$. While the following result makes the
stronger assumption that $\E[|Z_1|^{q_0}\mid Z_0=(1,\ldots,1)]<\infty$ for some $q_0>1$, we obtain the finer results of
Section~\ref{sec:main-discrete-time}, including a stronger form of convergence (in total variation norm with exponential
speed), a non-trivial subset of the domain of attraction of the minimal quasi-stationary distribution and stronger
moment properties for this quasi-stationary distribution.

\begin{cor}
  \label{cor:GW-bis}
  If $(Z_n,n\geq 0)$ is a $d$-type irreducible, aperiodic sub-critical Galton-Watson
  process, and if, for some $q_0>1$,
  \begin{align*}
    \E[|Z_1|^{q_0}\mid Z_0=(1,\ldots,1)]<\infty,
  \end{align*}
  then Condition~(E) holds true with $\varphi_1(z)=|z|^q$ for any $q\in(1,q_0]$. In
  particular, the domain of attraction of $\nu_{QSD}$ contains all the probability
  measures such that $\mu(|\cdot|^q)<\infty$ for some $q>1$.
\end{cor}

This corollary easily derives from Theorem~\ref{thm:GW-dominated-bis}. Indeed, setting
$\|z\|=\langle v,z \rangle$ and
$\xi_{i,n}^{(Z_n)}=\sum_{j=1}^d v_j \zeta_{k,j}^{(n,\ell)}$ (assuming that $i$ is the
$\ell-th$ individual of type $k$ in the population), one obtains
\[
  \|Z_{n+1}\|=\sum_{i=1}^{|Z_n|}\xi^{(Z_n)}_{i,n}
\]
and
\[
  \E\left.\left(\sum_{i=1}^{|Z_n|}\xi^{(Z_n)}_{i,n}\,\right|\,
    Z_n=z\right)=\sum_{k=1}^{d}\sum_{\ell=1}^{z_k} \sum_{j=1}^d
  v_j\E\left(\zeta^{(n,\ell)}_{k,j}\right) =\rho(M) \|z\|,
\]
for all $z\in \Z_+^d$. Since, in the case of multi-type Galton-Watson process,
one has $R=1/\rho(M)$ (see for instance Theorems~2 and~3 of\cite{JoffeSpitzer1967}),
Theorem~\ref{thm:GW-dominated-bis} applies with $m=\rho(M)$.

\medskip
To prove Theorem~\ref{thm:GW-dominated-bis}, we use the following lemma.
\begin{lem}
  \label{lem:Chatterji-burkholder}
  For all $q\in\left(\frac{\log R}{\log(1/m)}\vee 1,q_0\right]$, there exists a constant $C_q$ such that,
    for all $z\in\Z_+^d$,
  \begin{align*}
    \E\left[\left(\sum_{i=1}^{|z|}\xi^{(z)}_{i,n}-\E(\xi^{(z)}_{i,n})\right)^q\right]\leq C_q
    |z|^{1\vee (q/2)}.
  \end{align*}
\end{lem}

\begin{proof}
  If $q\in(1,2]$, this is exactly Lemma~1 of~\cite{chatterji-69}. If $q\geq 2$, Burkholder's inequality~\cite{burkholder-66} implies
  that there exists a constant $c_q$ such that
  \begin{align*}
    \E\left[\left(\sum_{i=1}^{|z|}\xi^{(z)}_{i,n}-\E(\xi^{(z)}_{i,n})\right)^q\right] & \leq
    c_q\E\left[\left(\sum_{i=1}^{|z|}\left\{\xi^{(z)}_{i,n}-\E(\xi^{(z)}_{i,n})\right\}^2\right)^{q/2}\right] \\
    & =c_q|z|^{q/2}\E\left[\left(\frac{1}{|z|}\sum_{i=1}^{|z|}\left\{\xi^{(z)}_{i,n}-\E(\xi^{(z)}_{i,n})\right\}^2\right)^{q/2}\right] \\
    & \leq c_q|z|^{q/2}\E\left[\frac{1}{|z|}\sum_{i=1}^{|z|}\left|\xi^{(z)}_{i,n}-\E(\xi^{(z)}_{i,n})\right|^q\right] \\
    & \leq c_q|z|^{q/2}\E\left[\frac{1}{|z|}\sum_{i=1}^{|z|}\left|\xi^{(z)}_{i,n}\right|^q+\E(\xi^{(z)}_{i,n})^q\right] \\
    & \leq 2c_q|z|^{q/2}\, \sup_{z\in\Z_+^d,\ 1\leq i\leq|z|}\E[(\xi^{(z)}_{i,n})^{q}],
  \end{align*}
  where we used Jensen's inequality in the third line, that the r.v.\ $\xi^{(z)}_{i,n}$ are nonnegative in the
  fourth line and H\"older's inequality in the last inequality.
\end{proof}

\begin{proof}[Proof of Theorem~\ref{thm:GW-dominated-bis}]
  We introduce an increasing sequence $(K_k,k\geq 0)$ of finite subsets of
  $\Z_+^d\setminus\{\d\}$, where $K_k$ is the smallest set containing
  $\{z\in\Z_+^d:1\leq |z|\leq k\}$ such that the process $Z$ restricted to $K_k$ is
  irreducible and aperiodic. The existence of this set follows from the irreducibility
  assumption and the fact that $\Z_+^d$ is countable. We shall choose $K=K_k$ for an
  appropriate value of $k\geq 0$.

  Fix $q\in\left(\frac{\log R}{\log(1/m)}\vee 1,q_0\right]$, $\theta_1\in (m^q,1/R)$, $\theta_2\in(\theta_1,1/R)$ and
  $\varphi_1(z)=\|z\|^q$. Using Minkowski's inequality in the first inequality, Lemma~\ref{lem:Chatterji-burkholder}
  in the third line and the equivalence between norms on $\R_+^d$,
  \begin{align}
    P_1\varphi_1(z) =\E\left(\left|\sum_{i=1}^{|z|}\xi^{(z)}_{i,n}\right|^q\right)
    & \leq\left[\E\left(\left|\sum_{i=1}^{|z|}\xi^{(z)}_{i,n}-\E(\xi^{(z)}_{i,n})\right|^q\right)^{1/q}+
      \sum_{i=1}^{|z|}\E(\xi^{(z)}_{i,n})\right]^q \notag \\
    & \leq\left[\left(C_q |z|^{1\vee (q/2)}\right)^{1/q}+m\|z\|\right]^q \notag \\ 
    & =m^q\|z\|^q\left(1+C'_q |z|^{1/(q\wedge 2)-1}\right)^q \notag \\
    & \leq m^q\|z\|^q+C''_q|z|^{q-1+1/(q\wedge 2)}, \label{eq:calcul-GW}
  \end{align}
  for constants $C'_q$ and $C''_q$ only depending on $q$ and $m$. Since $q-1+1/(q\wedge 2)<q$, there exists $k_1\geq 0$
  such that, for all $z\not\in K_{k_1}$,
  \begin{align}
    \label{eq:E2-GW-phi_1}
    P_1\varphi_1(z)\leq \theta_1\varphi_1(z).
  \end{align}
  We also deduce that, for all $z\in K_{k_1}$,
  \[
    P_1\varphi_1(z)\leq \max_{x\in K_{k_1}} m^q\|x\|^q+C''_q|x|^{q-1+1/(q\wedge
      2)}<+\infty.
  \]
  Setting $K=K_{k_1}$, we deduce that the first and third lines of Condition~(E2) are satisfied.

  By definition of $R$, we have
  $\lim_{n\to\infty}\theta_2^{-n}\inf_{z\in K}\P_z(X_n\in K)=+\infty$ and hence, using
  Lemma~\ref{lem:varphi2}, there exists a function $\varphi_2:E\rightarrow[0,1]$ such
  that the second and fourth lines of Condition~(E2) are satisfied. It also implies that
  Condition~(E4) holds true.
 
  Since the process is irreducible and aperiodic and $K$ is finite,~\eqref{eq:condition-E2thenE} is clearly satisfied for $n_0=1$ and
  $m_0$ large enough, so that Theorem~\ref{thm:GW-dominated-bis} follows from Proposition~\ref{lem:E2thenE}.
\end{proof}

\section{Proof of Theorem~\ref{thm:main}}
\label{sec:proof}

%In this section, we prove 

%that, for some constants $C>0$ and $\alpha\in(0,1)$, for all probability measure $\mu$ on $E$ such that $\mu(\varphi_1)/\mu(\varphi_2)>0$ and all $|h|\leq \varphi_1$, we have
%\begin{align}
%\label{eq:thm-main-zxzx-simple}
%\left|\frac{\mu P_{n} h}{\mu P_{n}\11_E} -\nu_{QSD}(h)\right|
%&\leq C\,\alpha^n\, \frac{\mu(\varphi_1)}{\mu(\varphi_2)},\quad\forall n\geq 0.
%\end{align}

In all the proof, the constants $C$ are all positive and finite and may change from line to line. We first assume from
Subsections~\ref{sec:main-steps} to~\ref{sec:end-proof} that for all $n\geq 0$ and all $x\in E$, $\PP_x(n<\tau_\d)>0$. The general
case will be handled in Subsection~\ref{sec:general-case-pf-main}.

\subsection{Main steps of the proof}
\label{sec:main-steps}

The proof is based on a careful study of the semigroup of the process conditioned to not be absorbed before time $T$. In this
section, we give the main ideas and steps of the proof of~\eqref{eq:thm-main-zxzx} for the total variation norm
$\|\cdot\|_{TV}:=\|\cdot\|_{TV(1)}$ in place of $\|\cdot\|_{TV(\varphi_1)}$, and leave the details for the following subsections,
where preliminary results and the following
Propositions~\ref{prop:lyap-conditioned},~\ref{prop:dobrushin-conditioned},~\ref{prop:expo-conv-conditioned} and
Lemma~\ref{lem:expo-moment-mu} are proved. The general case of the $\|\cdot\|_{TV(\varphi_1)}$ norm is handled in
Subsection~\ref{sec:pf-cor-main}.

For any $T\in\ZZ_+$, we consider the law of the process $X$ conditioned to not be absorbed before time $T$. We introduce the linear operators $(S_{m,n}^T)_{0\leq m\leq n\leq T}$ defined by
\begin{align*}
S_{m,n}^T f(x)=\E(f(X_n)\mid X_m=x,\ T<\tau_\d)=\frac{P_{n-m}\left(f P_{T-n}\11_E\right)(x)}{P_{T-m}\11_E(x)}.
\end{align*}
It is well-known that $(S_{m,n}^T)_{0\leq m\leq n\leq T}$ forms a time-inhomogeneous semi\-group (i.e.\ $S_{m,n}^TS_{n,p}^T=S_{m,p}^T$ for all $m\leq n\leq
p\leq T$) and that the process $(X_n,0\leq n\leq T)$ under $\PP^{S^T_{0,\cdot}}_x$ is a (time-inhomogeneous) Markov process, where we
denote by $\PP^{S^T_{0,\cdot}}_x$ the law of the process $(X_n,0\leq n\leq T)$ conditionally on $T<\tau_\d$ and $X_0=x$.

Fix $\theta\in(\theta_1/\theta_2,1)$. For any $T\geq 0$, we set, for $x\in E$,
\begin{align*}
\psi_T(x)=\E_x(\theta^{- T_K\wedge T}\mid T<\tau_\d)=\E_x^{S^T_{0,\cdot}}\left(\theta^{- T_K\wedge
    T}\right),
\end{align*}
where 
\begin{align*}
T_K:=\inf\{n\in\ZZ_+:X_n\in K\}
\end{align*}
is the first hitting time of $K$ by the process $(X_n,n\in\ZZ_+)$. Be careful that $ T_K$ is not the first hitting time of $K$ by
the full process $(X_t,t\in I)$, unless $I=\ZZ_+$.

The following proposition provides a Lyapunov-type property for the inhomogeneous semigroup $S$.
\begin{prop}
    \label{prop:lyap-conditioned}
    There exists a constant $\bar{C}>0$ such that, for all $0\leq m< T$ and $1\leq k\leq T-m$,
    \begin{align}
    S^{T}_{m,m+k}\psi_{T-(m+k)}(x)\leq \theta^{k} \psi_{T-m}(x)+\bar{C},\quad\forall x\in E.
    \end{align}
\end{prop}

The next proposition provides a Dobrushin coefficient-type property for the inhomogeneous semigroup $S$.
\begin{prop}
    \label{prop:dobrushin-conditioned}
    There exists a constant $\alpha_0\in(0,1)$ such that, for all $R>0$, there exists $k_R\geq 1$ such that,
    for all $T\geq k_R$ and all $x,y\in E$ such that
    $\psi_T(x)+\psi_T(y)\leq R$, we have
    \begin{align*}
    \left\|\delta_xS_{0,k_R}^T-\delta_y S_{0,k_R}^T\right\|_{TV}\leq 2(1-\alpha_0).
    \end{align*}
\end{prop}

The following property is a consequence of the two previous ones.
\begin{prop}
    \label{prop:expo-conv-conditioned}
    There exist constants $n_0\geq 1$, $C>0$ and $\alpha\in(0,1)$ such that, $\forall n\geq 1$ and all $x,y\in E$,
    \begin{align*}
    \left\|\delta_x S_{0,n_0 n}^{n_0 n}-\delta_y S_{0,n_0 n}^{n_0 n}\right\|_{TV}\leq C\alpha^n(2+\psi_{n_0 n}(x)+\psi_{n_0 n}(y)).
    \end{align*}
\end{prop}

Let us now deduce~\eqref{eq:thm-main-zxzx} with the total variation norm in place of $\|\cdot\|_{TV(\varphi_1)}$, from the last
proposition. We have, for all $x,y\in E$,
\begin{multline*}
\left\|\delta_x P_{nn_0} -\delta_x P_{nn_0}\11_E\,\delta_y S_{0,n_0 n}^{n_0 n}\right\|_{TV}\\
\leq C\alpha^n\left(2\delta_x P_{nn_0}\11_E+\E_x\left(\theta^{- T_K\wedge nn_0}\11_{nn_0<\tau_\d}\right)+\psi_{n_0 n}(y)\delta_x P_{nn_0}\11_E\right).
\end{multline*}
Hence, for any probability measure $\mu$ on $E$, integrating the above inequality over $\mu(\mathrm dx)$ leads to
\begin{multline*}
\left\|\mu P_{nn_0} -\mu P_{nn_0}\11_E\,\delta_y S_{0,n_0 n}^{n_0 n}\right\|_{TV}\\
\leq C\alpha^n\left(2\mu P_{nn_0}\11_E+\E_\mu\left(\theta^{- T_K\wedge nn_0}\11_{nn_0<\tau_\d}\right)+\psi_{n_0 n}(y)\mu P_{nn_0}\11_E\right).
\end{multline*}
We make use of the following lemma.

\begin{lem}
    \label{lem:expo-moment-mu}
    For all $\theta\in(\theta_1/\theta_2,1)$, there exists a constant $C$ such that, for
    all $0\leq m\leq T$ and all probability measure $\mu$ over $E$ such that $\mu(\varphi_2)>0$,
    \begin{align*}
    \E_\mu\left(\theta^{- T_K\wedge T}\11_{T<\tau_\d}\right)\leq C\,\frac{\mu(\varphi_1)}{\mu(\varphi_2)} \P_\mu\left(T<\tau_\d\right).
    \end{align*}
\end{lem}

This implies that, for all $\mu$ such that $\mu(\varphi_2)>0$,
\begin{multline*}
\left\|\mu P_{nn_0} -\delta_y S_{0,n_0 n}^{n_0 n}\mu P_{nn_0}\11_E\right\|_{TV} \\
\leq C\alpha^n\left(2\mu P_{nn_0}\11_E+
\frac{\mu(\varphi_1)}{\mu(\varphi_2)} \mu P_{nn_0}\11_E +\psi_{n_0 n}(y)\mu P_{nn_0}\11_E\right).
\end{multline*}
Hence
\begin{align*}
\left\|\frac{\mu P_{nn_0}}{\mu P_{nn_0}\11_E} -\delta_y S_{0,n_0 n}^{n_0 n}\right\|_{TV} &\leq C\alpha^n\left(2+
\frac{\mu(\varphi_1)}{\mu(\varphi_2)} +\psi_{n_0 n}(y)\right).
\end{align*}
Using the same procedure w.r.t.\ $y$, we deduce that, for any probability measures $\mu_1$ and $\mu_2$ on $E$ such that
$\mu_1(\varphi_2)>0$ and $\mu_2(\varphi_2)>0$,
\begin{align*}
\left\|\frac{\mu_1 P_{nn_0}}{\mu_1 P_{nn_0}\11_E} -\frac{\mu_2 P_{nn_0}}{\mu_2 P_{nn_0}\11_E}\right\|_{TV}
&\leq C\alpha^n\left(\frac{\mu_1(\varphi_1)}{\mu_1(\varphi_2)}+\frac{\mu_2(\varphi_1)}{\mu_2(\varphi_2)}\right),
\end{align*}
where we used the fact that $\mu(\varphi_1)/\mu(\varphi_2)\geq 1$ for all probability measure $\mu$ on $E$ such that $\mu(\varphi_2)>0$.

Because of Lemma~\ref{lem:borne-P_n-phi1/phi2} below, we deduce that, for some constant $D_1>0$ and for all $0\leq k< n_0$,
\begin{align*}
\left\|\frac{\mu_1 P_{nn_0+k}}{\mu_1 P_{nn_0+k}\11_E} -\frac{\mu_2 P_{nn_0+k}}{\mu_2 P_{nn_0+k}\11_E}\right\|_{TV}
&\leq C\alpha^n\left(\frac{\mu_1 P_k\varphi_1}{\mu_1 P_k\varphi_2} +\frac{\mu_2 P_k\varphi_1}{\mu_2
    P_k\varphi_2}\right) \\ & \leq C\alpha^n\left(\frac{\mu_1(\varphi_1)}{\mu_1(\varphi_2)}\vee D_1
+\frac{\mu_2(\varphi_1)}{\mu_2(\varphi_2)}\vee D_1\right).
\end{align*}
Therefore, up to a change in the constant $C$ and replacing $\alpha$ by $\alpha^{1/n_0}$, we deduce that, for all probability
measures $\mu_1$ and $\mu_2$ on $E$ such that $\mu_1(\varphi_2)>0$ and $\mu_2(\varphi_2)>0$ and for all $n\geq 0$,
\begin{align}
\label{eq:pf-main-1}
\left\|\frac{\mu_1 P_{n}}{\mu_1 P_{n}\11_E} -\frac{\mu_2 P_{n}}{\mu_2 P_{n}\11_E}\right\|_{TV}
&\leq C\alpha^n\left(\frac{\mu_1(\varphi_1)}{\mu_1(\varphi_2)} +\frac{\mu_2(\varphi_1)}{\mu_2(\varphi_2)}\right).
\end{align}
Fix $x_0\in K$. We set $\mu_1=\delta_{x_0}$ and $\mu_2=\frac{\mu_1P_1}{\mu_1P_1\11_E}$ in~\eqref{eq:pf-main-1}. Since
$\frac{\mu_1\varphi_1}{\mu_1\varphi_2}<\infty$ and because of Lemma~\ref{lem:borne-P_n-phi1/phi2} below, we have
$\frac{\mu_2\varphi_1}{\mu_2\varphi_2}<\infty$. We deduce that, for some constant $C>0$,
$$
\left\|\frac{\delta_{x_0} P_{n+1}}{\delta_{x_0} P_{n+1}\11_E} -\frac{\delta_{x_0} P_{n}}{\delta_{x_0} P_{n}\11_E}\right\|_{TV}
\leq C\alpha^n,
$$
and hence, using the completeness of the space of probability measures on $E$ for the total variation norm, we deduce that there exists a quasi-limiting measure $\nu_{QSD}$ (which is hence a quasi-stationary
distribution) such that
\begin{align*}
\left\|\frac{\delta_{x_0} P_{n}}{\delta_{x_0} P_{n}\11_E} -\nu_{QSD}\right\|_{TV}
&\leq \frac{2 C}{1-\alpha}\alpha^n.
\end{align*}
In particular, it follows from Lemma~\ref{lem:to-be-in-K} below that $\nu_{QSD}(K)>0$ and hence that $\nu_{QSD}(\varphi_2)>0$. Since
Lemma~\ref{lem:borne-P_n-phi1/phi2} implies that $\frac{P_n\varphi_1(x_0)}{P_n\11_E(x_0)}$ is uniformly bounded in $n\geq 0$, we
deduce that $\nu_{QSD}(\varphi_1\wedge M)$ is bounded uniformly in $M>0$ and hence $\nu_{QSD}(\varphi_1)<\infty$.

Using~\eqref{eq:pf-main-1} again (up to another change of the constant $C$), we obtain that, for all probability measure $\mu$ on $E$
such that $\frac{\mu(\varphi_1)}{\mu(\varphi_2)}<\infty$,
\begin{align*}
\left\|\frac{\mu P_{n}}{\mu P_{n}\11_E} -\nu_{QSD}\right\|_{TV}
&\leq C\alpha^n \frac{\mu(\varphi_1)}{\mu(\varphi_2)}.
\end{align*}
This also entails that there exists a unique quasi-stationary distribution such that
$\nu_{QSD}(\varphi_1)/\nu_{QSD}(\varphi_2)<\infty$.

This ends the proof of~\eqref{eq:thm-main-zxzx} for the total variation norm. The general case with the norm
$\|\cdot\|_{TV(\varphi_1)}$ is proved in Subsection~\ref{sec:pf-cor-main}.

\subsection{Preliminary results}

We start by proving two basic inequalities which are direct consequences of (E2). 

\begin{lem}
    \label{lem:basic-inequalities}
    For all $x\in E\setminus K$ and all $n\geq 0$,
    $$
    \PP_x(n< T_K\wedge\tau_\d)\leq\EE_x[\varphi_1(X_n)\11_{n< T_K\wedge\tau_\d}]\leq\theta_1^n\varphi_1(x).
    $$
    For all $x\in E$ and $n\geq 0$,
    $$
    \PP_x(n<\tau_\d)\geq\EE_x[\varphi_2(X_n)\11_{n<\tau_\d}]\geq\theta_2^n\varphi_2(x).
    $$
\end{lem}

\begin{proof}[Proof of Lemma~\ref{lem:basic-inequalities}]
    These two properties follow easily by induction from (E2). For example, the first one makes use of the following relation: for all
    $n\geq 1$ and $x\in E$,
    \begin{equation*}
    \EE_x[\varphi_1(X_n)\11_{n< T_K\wedge\tau_\d}]=\11_{x\in E\setminus
        K}\,P_1\left[\EE_\cdot\left(\varphi_1(X_{n-1})\11_{n-1< T_K\wedge\tau_\d}\right)\right](x).
    \end{equation*}
    This and~(E2) entail the property at time $n=1$ and, by induction, at any time $n\geq 1$.
\end{proof}

The next lemma states that the expectation of $\varphi_1(X_n)$ is controlled by the expectation of $\varphi_2(X_n)$ uniformly in
time.

\begin{lem}
    \label{lem:borne-P_n-phi1/phi2}
    For all $\theta\in(\theta_1/\theta_2,1]$, there exists a finite constant $D_\theta>0$ such that, for all probability measure $\mu$
    on $E$ such that $\mu(\varphi_1)/\mu(\varphi_2)<\infty$, for all $T\in\ZZ_+$ and all $x\in E$,
    \begin{equation}
    \label{eq:lem-borne-p_n-phi1/phi2}
    \frac{\mu P_T\varphi_1}{\mu P_T\varphi_2}\leq\left(\theta^T\frac{\mu(\varphi_1)}{\mu(\varphi_2)}\right)\vee D_\theta.    
    \end{equation}
\end{lem}

\begin{proof}[Proof of Lemma~\ref{lem:borne-P_n-phi1/phi2}]
    It follows from (E2) that
    $$
    \mu P_{T+1}\varphi_1\leq\theta_1 \mu P_T\varphi_1+C\mu P_T\11_K
    $$
    and
    $$
    \mu P_{T+1}\varphi_2\geq\theta_2\mu P_T\varphi_2.
    $$
    Hence
    \begin{align*}
    \frac{\mu P_{T+1}\varphi_1}{\mu P_{T+1}\varphi_2}
    &\leq \frac{\theta_1 \mu P_{T}\varphi_1+C\mu P_T\11_K(x)}{\theta_2 \mu P_{T}\varphi_2}\\
    &\leq \frac{\theta_1}{\theta_2}\, \frac{\mu P_T\varphi_1}{\mu P_{T}\varphi_2}+\frac{C}{\theta_2\inf_{y\in K}\varphi_2(y)}.
    \end{align*}
    Since $\theta_1/\theta_2<\theta$, these arithmetico-geometric inequalities entail~\eqref{eq:lem-borne-p_n-phi1/phi2}.
\end{proof}

We now give an irreducibility inequality.

\begin{lem}
    \label{lem:E'4}
    For all $C\geq 1$, there exists a time $n_5(C)\in\N$ such that
    \begin{align}
    \label{eq:E'4-repl}
    a_5(C):=\inf_{\mu\in\mathcal{M}_1(E)\text{ s.t.\ }\mu(\varphi_1)\leq C\mu(\varphi_2)}\P_\mu(X_{n_5(C)}\in K)>0.
    \end{align}
\end{lem}

\begin{proof}[Proof of Lemma~\ref{lem:E'4}]
    It follows from (E4) that there exists a time $n_\nu\in\N $ such that, for all $n\geq n_\nu$, 
    $\P_\nu(X_{n}\in K)>0$, and, using~(E1), that for all $n\geq n_\nu+n_1$,
    \[
    \inf_{x\in K}\P_x(X_n\in K)\geq c_1\P_\nu(X_{n-n_1}\in K)>0.
    \]
    Let $C\geq 1$ and $\mu$ be such that $\mu(\varphi_1)\leq C\mu(\varphi_2)$. It follows from Lemma~\ref{lem:basic-inequalities} that,
    for all $n\geq 1$,
    \begin{align*}
    \P_\mu( T_K\wedge\tau_\d> n)\leq\E_\mu\left[\varphi_1(X_n)\11_{ T_K\wedge\tau_\d> n}\right]\leq\theta_1^{n}\mu(\varphi_1)\leq C\theta_1^n\mu(\varphi_2).
    \end{align*}
    and
    \begin{align*}
    \PP_\mu(n<\tau_\d)\geq\EE_\mu[\varphi_2(X_n)]\geq\theta_2^n\mu(\varphi_2).
    \end{align*}
    Therefore,
    \begin{align*}
    \P_\mu( T_K\leq n<\tau_\d)\geq \left(\theta_2^n- C\theta_1^n\right)\mu(\varphi_2).
    \end{align*}
    Choosing $n(C)=\lceil 2C/\log(\theta_2/\theta_1)\rceil$, we deduce that
    \begin{align*}
    \P_\mu( T_K\leq n(C)<\tau_\d)\geq \frac{\theta_2^{n(C)}}{2}\mu(\varphi_2)\geq \frac{\theta_2^{n(C)}}{2C}.
    \end{align*}
    Therefore,
    \begin{align*}
    \P_\mu(X_{n(C)+n_\nu+n_1}\in K) & \geq\E_\mu\left[\11_{ T_K\leq n(C)}\restriction{\P_{X_{ T_K}}(X_{n(C)+n_\nu+n_1-k}\in
        K)}{k= T_K}\right] \\ & \geq\min_{n_\nu+n_1\leq k\leq n_\nu+n_1+n(C)}\,\inf_{x\in K}\P_x(X_{k}\in K)\,\frac{\theta_2^{n(C)}}{2C}.
    \end{align*}
    Hence we have proved Lemma~\ref{lem:E'4} with $n_5(C)=n_\nu+n_1+n(C)$.
\end{proof}

The next lemma shows that conditional distributions with initial conditions in $K$ give to $K$ a mass uniformly bounded from below.

\begin{lem}
    \label{lem:to-be-in-K} There exists a time $n_6\in\N$ such that
    \begin{align*}
    \inf_{T\geq n_6}\inf_{x\in K}\P_x(X_T\in K\mid T<\tau_\d)>0.
    \end{align*}
\end{lem}

\begin{proof}[Proof of Lemma~\ref{lem:to-be-in-K}]
    Since $\varphi_1/\varphi_2$ is bounded over $K$, we deduce from Lemma~\ref{lem:borne-P_n-phi1/phi2} that, setting $C:=D_1+\sup_{x\in
        K}\frac{\varphi_1(x)}{\varphi_2(x)}$, we have for all $x\in K$ and all $T\geq n_5(C)$,
    \begin{align}
    \label{eq:pf-to-be-in-K-a}
    \frac{P_{T-n_5(C)}\varphi_1(x)}{P_{T-n_5(C)}\varphi_2(x)}
    &\leq C.
    \end{align}
    Using Lemma~\ref{lem:E'4} applied to $\mu=\frac{\delta_x P_{T-n_5(C)}}{\delta_x P_{T-n_5(C)}\11_E}$, we deduce that, for all $x\in K$ and
    $T\geq n_5(C)$,
    \begin{equation*}
    \P_x(X_T\in K\mid T<\tau_\d)= \frac{\mu P_{n_5(C)}\11_K}{\mu P_{n_5(C)}\11_E}\geq \mu P_{n_5(C)}\11_K \geq a_5(C). \qedhere
    \end{equation*}
\end{proof}

The next lemma shows that survival probabilities are controlled by the function $\varphi_1$.

\begin{lem}
    \label{lem:comp-surv}
    There exists a constant $C>0$ such that, for all $p\in[1,\log\theta_1/\log\theta_2)$, $x\in E$ and $n\geq 1$,
    \begin{align}
    \label{eq:comp-surv-2}
    \P_x(n<\tau_\d)\leq C\, \frac{\varphi_1(x)^{1/p}}{1-\theta_1^{1/p}/\theta_2}\,\inf_{y\in K}\P_y(n<\tau_\d).
    \end{align}
\end{lem}

\begin{proof}[Proof of Lemma~\ref{lem:comp-surv}]
    It follows from Lemma~\ref{lem:basic-inequalities} that, for all $p\geq 1$, $x\in E\setminus K$ and $n\geq 1$,
    \begin{align}
    \label{eq:comp-surv-1}
    \P_x(n< T_K\wedge\tau_\d)\leq \theta_1^{n/p}\varphi_1(x)^{1/p}.
    \end{align}
    Note that this inequality is trivial for $x\in K$. In particular, for $p\geq 1$ such that $\theta_1^{1/p}<\theta_2$, for all $x\in K$,
    \begin{align}
    \label{eq:pf-comp-surv-a}
    \E_x(\theta_2^{- T_K\wedge\tau_\d})\leq\frac{\varphi_1(x)^{1/p}}{1-\theta_1^{1/p}/\theta_2}.
    \end{align}
    
    Fix $p\in[1,\log\theta_1/\log\theta_2)$. Using~\eqref{eq:comp-surv-1}, the second inequality
    of Lemma~\ref{lem:basic-inequalities} and~(E3), we have for all $x\in E$
    \begin{align}
    \P_x(n<\tau_\d)&= \P_x(n< T_K\wedge \tau_\d)+\P_x( T_K\wedge\tau_\d\leq n <\tau_\d) \notag \\
    &\leq \theta_2^n\varphi_1(x)^{1/p} 
    +\sum_{k=0}^n \P_x( T_K\wedge\tau_\d=k)\sup_{y\in K} \P_y(n-k<\tau_\d) \notag \\
    &\leq \frac{\inf_{z\in K} \P_z(n<\tau_\d)}{\inf_{z\in K}\varphi_2(z)}\,\varphi_1(x)^{1/p}
    +c_3\sum_{k=0}^n\P_x( T_K\wedge\tau_\d=k)\inf_{y\in K} \P_y(n-k<\tau_\d)  \notag   \\
    &\leq C \inf_{z\in K} \P_z(n<\tau_\d)\,\varphi_1(x)^{1/p} +C\inf_{z\in
        K}\P_z(n<\tau_\d)\sum_{k=0}^n \P_x( T_K\wedge\tau_\d=k) \theta_2^{-k}, \label{eq:pf-comp-survv}
    \end{align}
    where we used the fact that, for some constant $C>0$, for all $n\geq k\geq 0$ and all $z\in K$,  
    \begin{align}
    \label{eq:lem:comp-surv-1}
    \P_z(n<\tau_\d)\geq  C\theta_2^k\inf_{y\in K}\P_y(n-k<\tau_\d).
    \end{align}
    This is proved using the three following equations. For all $n\geq k\geq n_6$ and all $z\in K$, by Lemmata~\ref{lem:to-be-in-K}
    and~\ref{lem:basic-inequalities},
    \begin{align*}
    \P_z(n<\tau_\d)&\geq \P_z(X_k\in K\mid k<\tau_\d)\P_z(k<\tau_\d) \inf_{y\in K}\P_y(n-k<\tau_\d)\\
    &\geq C\theta_2^k\varphi_2(z)\inf_{y\in K}\P_y(n-k<\tau_\d)\\
    &\geq C\theta_2^k\inf_{y\in K}\P_y(n-k<\tau_\d),
    \end{align*}
    since $\inf_{z\in K}\varphi_2(z)>0$. Also, for all $n\geq n_6\geq k$, using the last inequality,
    \begin{align*}
    \P_z(n<\tau_\d)&\geq C\theta_2^{n_6}\inf_{y\in K}\P_y(n-n_6<\tau_\d)\\
    &\geq C\theta_2^{n_6}\, \inf_{y\in K}\P_y(n-k<\tau_\d) \\
    &\geq (C\theta_2^{n_6})\, \theta_2^k\, \inf_{y\in K}\P_y(n-k<\tau_\d).
    \end{align*}
    Finally, for all $k\leq n< n_6$,
    \begin{align*}
    \P_z(n<\tau_\d)&\geq \P_z(n_6<\tau_\d)\geq  C\theta_2^{n_6}\geq (C\theta_2^{n_6})\, \theta_2^k\, \inf_{y\in K}\P_y(n-k<\tau_\d),
    \end{align*}
    so~\eqref{eq:lem:comp-surv-1} is proved.
    
    Combining~\eqref{eq:pf-comp-surv-a} and~\eqref{eq:pf-comp-survv} ends the proof of Lemma~\ref{lem:comp-surv}.
\end{proof}

\subsection{Proof of Proposition~\ref{prop:lyap-conditioned}}
Markov's property implies that, for all $x\in E\setminus K$ and $T,m\geq 1$,
\begin{align}
\label{eq:pf-prop-lyap-conditioned-1}
S^{T}_{0,1}\psi_{T-1}(x) & =S^{T+m}_{m,m+1}\psi_{T-1}(x)=\theta \psi_{T}(x).
\end{align}
Indeed,
\begin{align*}
\theta\psi_T(x)
& =\frac{\E_x(\theta^{1- T_K\wedge T}\11_{T<\tau_\d})}{\P_x(T<\tau_\d)} \\
& =\frac{\E_x\left[\11_{1<\tau_\d}\E_{X_1}(\theta^{- T_K\wedge(T-1)}\mid
    T-1<\tau_\d)\P_{X_1}(T-1<\tau_\d)\right]}{\P_x(T<\tau_\d)}=S^{T}_{0,1}\psi_{T-1}(x).
\end{align*}
Similarly, for all $x\in $K,
\begin{align}
\label{eq:pf-prop-lyap-conditioned-2}
S^{T}_{0,1}\psi_{T-1}(x) & =S^{T+m}_{m,m+1}\psi_{T-1}(x)=\theta\EE^{S^T_{0,\cdot}}_x(\theta^{-\sigma_K\wedge T}),
\end{align}
where 
\begin{align*}
\sigma_K:=\min\{n\geq 1,\ X_n\in K\}
\end{align*}
is the first return time in $K$. Setting
\begin{align*}
C & :=\sup_{T\geq 0}\sup_{x\in K}\E^{S_{0,\cdot}^T}_x(\theta^{-\sigma_K\wedge T}),
\end{align*}
which is finite (see Lemma~\ref{lem:expo-moment-x-in-K}), we can apply recursively~\eqref{eq:pf-prop-lyap-conditioned-1}
and~\eqref{eq:pf-prop-lyap-conditioned-2} to obtain
\begin{align*}
S_{m,m+k}^T\psi_{T-(m+k)} & = S_{m,m+k-1}^T\left(\11_{E\setminus K}S_{m+k-1,m+k}^T(\psi_{T-(m+k)})\right) \\ & \qquad
+S_{m,m+k-1}^T\left(\11_{K}S_{m+k-1,m+k}^T(\psi_{T-(m+k)})\right) \\ &
\leq\theta S_{m,m+k-1}^T\psi_{T-(m+k-1)}+C\theta \\ & \leq\ldots\leq
\theta^k \psi_{T-m}(x)+C\sum_{\ell=1}^k \theta^{\ell}.
\end{align*}
Hence Proposition~\ref{prop:lyap-conditioned} follows from the next lemma.

\begin{lem}
    \label{lem:expo-moment-x-in-K}
    For all $\theta\in(\theta_1/\theta_2,1)$, 
    \begin{align*}
    \sup_{T\geq 0}\sup_{x\in K}\E^{S_{0,\cdot}^T}_x(\theta^{-\sigma_K\wedge T})<\infty.
    \end{align*}
\end{lem}

\begin{proof}[Proof of Lemma~\ref{lem:expo-moment-x-in-K}]
    Fix $x\in K$. On the one hand, by Lemma~\ref{lem:comp-surv} (with $p=1$), we have for any $1\leq n<T$,
    \begin{align*}
    \P_x(n<\sigma_K\text{ and }T<\tau_\d)& =\E_x(\11_{n<\sigma_K\wedge\tau_\d}\P_{X_n}(T-n<\tau_\d))\\
    &\leq C\inf_{y\in K}\P_y(T-n<\tau_\d)\E_x(\11_{n<\sigma_K\wedge\tau_\d}\varphi_1(X_n)).
    \end{align*}
    Using (E2) and Markov's property as in the proof of Lemma~\ref{lem:basic-inequalities}, we deduce
    \begin{align}
    \label{eq:pf-lem-expo-in-K}
    \P_x(n<\sigma_K\text{ and }T<\tau_\d) 
    &\leq C\inf_{y\in K}\P_y(T-n<\tau_\d)\theta_1^{n-1} P_1\varphi_1(x)\\
    &\leq C\inf_{y\in K}\P_y(T-n<\tau_\d)\theta_1^n.
    \end{align}
    On the other hand, Lemma~\ref{lem:to-be-in-K} implies the existence of a constant $C>0$ such that, for all $x\in K$ and all $n\geq n_6$,
    \begin{align*}
    \P_x(X_n\in K)\geq C\P_x(n<\tau_\d).
    \end{align*}
    We deduce from Markov's property and Lemma~\ref{lem:basic-inequalities} that, for all $T\geq n\geq n_6$,
    \begin{align*}
    \P_x(T<\tau_\d)&\geq \P_x(X_n\in K)\inf_{y\in K} \P_y(T-n<\tau_\d)\\
    &\geq C\P_x(n<\tau_\d)\inf_{y\in K} \P_y(T-n<\tau_\d)\\
    &\geq C\theta_2^n \inf_{y\in K} \P_y(T-n<\tau_\d).
    \end{align*}
    Combining this with~\eqref{eq:pf-lem-expo-in-K}, we finally deduce that there exists a
    constant $C>0$ such that, for all $x\in K$ and all $T\geq n\geq n_6$,
    \begin{align}
    \label{eq:proba-x-in-K-not-in-K}
    \P_x(n<\sigma_K\mid T<\tau_\d)\leq C \left(\frac{\theta_1}{\theta_2}\right)^n.
    \end{align}
    The extension to any $T\geq n$ is trivial, so the conclusion follows.
\end{proof}

\subsection{Proof of Proposition~\ref{prop:dobrushin-conditioned}}

We start by stating a lemma
proved at the end of this subsection.
\begin{lem}
    \label{lem:prop-dobrushin}
    For all $x\in K$ and $n_1+n_6\leq n\leq T$,
    \begin{equation}
    \label{eq:lem-prop-dobrushin}
    \PP_x(X_n\in\cdot\mid T<\tau_\d)\geq c'_1\nu,    
    \end{equation}
    where the measure $\nu$ and the integer $n_1$ are the one of Condition~(E1), the integer $n_6$ is from Lemma~\ref{lem:E'4} and
    $c'_1>0$ is independent of $x,n$ and $T$.
\end{lem}

Fix $\theta\in(\theta_1/\theta_2,1)$ and set $k_R=\lceil \log(2R)/\log(1/\theta)\rceil+n_1+n_6$ and fix $T\geq k_R$. For all $x\in E$ such that $\psi_T(x)\leq R$, Markov's inequality
implies that
\begin{align*}
\P_x( T_K>k_R-n_1-n_6\mid T<\tau_\d)=\P^{S_{0,\cdot}^T}_x\left( T_K>k_R-n_1-n_6\right)\leq \frac{R}{\theta^{-k_R+n_1+n_6}}\leq\frac{1}{2}.
\end{align*}
It follows from Lemma~\ref{lem:prop-dobrushin} that, for all measurable $A\subset E$,
\begin{align*}
\P^{S_{0,\cdot}^T}_x\left(X_{k_R}\in A\right) & \geq \frac{\E_x\left[\sum_{k=1}^{k_R-n_1-n_6}\11_{ T_K=k}\P_{X_{k}}(
    X_{k_R-k}\in A,\,T-k<\tau_\d)\right]}{\P_x(T<\tau_\d)} \\
& \geq c'_1\nu(A)\frac{\E_x\left[\sum_{k=1}^{k_R-n_1-n_6}\11_{ T_K=k}\P_{X_k}(T-k<\tau_\d)\right]}{\P_x(T<\tau_\d)} \\
& =c'_1\nu(A)\P_x( T_K\leq k_R-n_1-n_6\mid T<\tau_\d) \\
& \geq\frac{1}{2}c'_1\nu(A).
\end{align*}
This concludes the proof of Proposition~\ref{prop:dobrushin-conditioned} with $\alpha_0=c'_1/2$.

\begin{proof}[Proof of Lemma~\ref{lem:prop-dobrushin}]
    For all measurable set $A\subset K$, we deduce from Markov's property that, for all $x\in K$ and all $T\geq n\geq n_1+n_6$,
    \begin{align}
    \P_x(X_n\in A,\, T<\tau_\d)
    & \geq\E_x\left[\11_{X_{n-n_1}\in K}\,\E_{X_{n-n_1}}\left(\11_{X_{n_1}\in A}\P_{X_{n_1}}(T-n<\tau_\d)\right)\right] \notag \\
    & \geq\E_x\left[\11_{X_{n-n_1}\in K}\,\P_{X_{n-n_1}}(X_{n_1}\in A)\right] \inf_{y\in K}\P_y(T-n<\tau_\d) \notag \\
    & \geq c_1\nu(A)\P_x\left(X_{n-n_1}\in K\right) \inf_{y\in K}\P_y(T-n<\tau_\d), \label{eq:prop-dobrushin-pf-1}
    \end{align}
    where we used (E1). Now, using Lemma~\ref{lem:comp-surv}, we deduce that there exists a constant $c>0$ such that
    \begin{align*}
    \P_x(T<\tau_\d)&\leq \P_x(T-n_1<\tau_\d)= \E_x\left(\11_{n-n_1<\tau_\d}\P_{X_{n-n_1}}(T-n<\tau_\d)\right)\\
    &\leq c\E_x\left(\11_{n-n_1<\tau_\d}\varphi_1(X_{n-n_1})\right)\inf_{y\in K}\P_y(T-n<\tau_\d).
    \end{align*}   
    Since $\varphi_1(x)/\varphi_2(x)$ is uniformly bounded over $x\in K$, Lemma~\ref{lem:borne-P_n-phi1/phi2} implies that there
    exists a constant $c'>0$ such that, for all $x\in K$,
    \begin{align*}
    \E_x\left[\11_{n-n_1<\tau_\d}\varphi_1(X_{n-n_1})\right]\leq c'\E_x\left[\11_{n-n_1<\tau_\d}\varphi_2(X_{n-n_1})\right] \leq
    c'\P_x\left(n-n_1<\tau_\d\right).
    \end{align*}
    But $n-n_1\geq n_6$, hence Lemma~\ref{lem:to-be-in-K} entails that there exists a constant $c''>0$ such that, for all $x\in K$,
    \begin{align*}
    \P_x\left(n-n_1<\tau_\d\right)\leq c'' \P_x(X_{n-n_1}\in K).
    \end{align*}
    Hence we obtain
    \begin{align*}
    \P_x(T<\tau_\d)\leq cc'c''\, \P_x\left(X_{n-n_1}\in K\right)\inf_{y\in K}\P_y(T-n<\tau_\d).
    \end{align*}   
    Combining this with~\eqref{eq:prop-dobrushin-pf-1}, we obtain
    \begin{align*}
    \P_x(X_n\in A\mid T<\tau_\d)
    & \geq \frac{c_1}{cc'c''} \,\nu(A).
    \end{align*}
    This ends the proof of Lemma~\ref{lem:prop-dobrushin}.
\end{proof}

\subsection{Proof of Proposition~\ref{prop:expo-conv-conditioned}}

We transpose the ideas of \cite{Hairer2010} (see also~\cite{HairerMattingly2011}) to the time-inhomogeneous setting. We fix the
constants $R=4\bar{C}/(1-\theta)$ and $\beta=\alpha_0/2\bar{C}$, where $\bar{C}$ is the constant of
Proposition~\ref{prop:lyap-conditioned}. For all $T\geq 0$ and all $\varphi:E\rightarrow\R$, we set
\begin{align*}
\vertiii{\varphi}_T=\sup_{x,y\in E}\frac{|\varphi(x)-\varphi(y)|}{2+\beta \psi_T(x)+\beta\psi_T(y)}.
\end{align*}
Fix $n$ and $T\geq 0$ such that $(n+1)k_R\leq T$ and let $\varphi$ be such that $\vertiii{\varphi}_{T-(n+1)k_R}\leq 1$. Then, replacing
$\varphi$ by $\varphi+c$ for some appropriate constant $c$, one has $|\varphi|\leq 1+\beta\psi_{T-(n+1)k_R}$ (see Lemma~3.8 p.14 in~\cite{Hairer2010}). 

\medskip \noindent If $\psi_{T-nk_R}(x)+\psi_{T-nk_R}(y)>R$, then, using Proposition~\ref{prop:lyap-conditioned},
\begin{multline*}
\left|S_{nk_R,(n+1)k_R}^{T}\varphi(x)-S_{nk_R,(n+1)k_R}^{T}\varphi(y)\right| \\
\begin{aligned}
&\leq 2+\theta\beta\psi_{T-nk_R}(x)+\theta\beta\psi_{T-nk_R}(y)+2\beta \bar{C}\\
&\leq 2+\left(\theta+(1-\theta)/2\right)(\beta\psi_{T-nk_R}(x)+\beta\psi_{T-nk_R}(y))\\
&\quad\quad\quad-(R\beta)(1-\theta)/2+2\beta \bar{C}\\
&\leq (1-\alpha_1) (2+\beta\psi_{T-nk_R}(x)+\beta\psi_{T-nk_R}(x)),
\end{aligned}
\end{multline*}
where $\alpha_1\in(0,1)$ is such that $2+\left(\theta+(1-\theta)/2\right)y\leq(1-\alpha_1)(2+y)$ for all $y\geq \beta R$.

\medskip \noindent If $\psi_{T-nk_R}(x)+\psi_{T-nk_R}(y)\leq R$, then, considering
\begin{align*}
\varphi=\varphi'+\varphi'',
\end{align*}
with $|\varphi'|\leq 1$ and $|\varphi''|\leq \beta\psi_{T-(n+1)k_R}$, Propositions~\ref{prop:lyap-conditioned}
and~\ref{prop:dobrushin-conditioned} entail
\begin{multline*}
\left|S_{nk_R,(n+1)k_R}^{T}\varphi(x)-S_{nk_R,(n+1)k_R}^{T}\varphi(y)\right| \\
\leq 2(1-\alpha_0)+\beta\theta\psi_{T-nk_R}(x)+\beta \theta\psi_{T-nk_R}(y)+2\beta \bar{C}.
\end{multline*}
Our choice $\beta=\alpha_0/2\bar{C}$ implies that
\begin{align*}
\left|S_{nk_R,(n+1)k_R}^{T}\varphi(x)-S_{nk_R,(n+1)k_R}^{T}\varphi(y)\right|\leq (1-\alpha_2)(2+\beta\psi_{T-nk_R}(x)+\beta\psi_{T-nk_R}(y)).
\end{align*}
for the constant $\alpha_2=\frac{\alpha_0}{2}\wedge(1-\theta)>0$.

Hence, we obtained
\begin{align*}
\vertiii{S_{nk_R,(n+1)k_R}^{T}\varphi}_{T-nk_R}\leq (1-\alpha_1\wedge\alpha_2)\vertiii{\varphi}_{T-(n+1)k_R},
\end{align*}
which implies by iteration that
\begin{align*}
\vertiii{S_{0,nk_R}^{nk_R}\varphi}_{nk_R}\leq (1-\alpha_1\wedge\alpha_2)^n \vertiii{\varphi}_{0}\leq
(1-\alpha_1\wedge\alpha_2)^n\|\varphi\|_\infty /(1+\beta).
\end{align*}
This concludes the proof of Proposition~\ref{prop:expo-conv-conditioned}.

\subsection{Proof of Lemma~\ref{lem:expo-moment-mu}}
\label{sec:end-proof}

This lemma in a generalization of Lemma~\ref{lem:expo-moment-x-in-K}. Its proof is based on similar computations. We give the details
for sake of completeness.

For all probability measure $\mu$ on $E$, for any $0\leq n<T$, using Lemma~\ref{lem:comp-surv} for the second inequality and
Lemma~\ref{lem:basic-inequalities} for the third inequality, we have
\begin{align}
\P_\mu(n<  T_K\text{ and }T<\tau_\d)&\leq \E_\mu(\11_{n< T_K}\P_{X_n}(T-n<\tau_\d))\notag \\
&\leq C\inf_{y\in K}\P_y(T-n<\tau_\d)\E_\mu(\11_{n< T_K}\varphi_1(X_n)) \notag \\
&\leq C\inf_{y\in K}\P_y(T-n<\tau_\d)\theta_1^{n} \mu(\varphi_1).
\label{eq:lem-expo-moment-mu-1}
\end{align}

For all integer $n\geq n_\mu$, where
\begin{align*}
n_\mu:=\left\lceil n_5(D_\theta)+\frac{\log\frac{\mu(\varphi_1)}{D_\theta\mu(\varphi_2)}}{\log(1/\theta)}\right\rceil,  
\end{align*}
it follows from Lemma~\ref{lem:borne-P_n-phi1/phi2} that
\begin{align*}
\frac{\mu P_{n-n_5(D_\theta)}\varphi_1}{\mu P_{n-n_5(D_\theta)}\varphi_2}\leq
D_\theta\vee\left(\theta^{n-n_5(D_\theta)}\frac{\mu(\varphi_1)}{\mu(\varphi_2)}\right)\leq D_\theta
\end{align*}
and from Lemma~\ref{lem:E'4} that
\begin{align*}
\frac{\mu P_n\11_K}{\mu P_n \11_E}\geq a_5(D_\theta)>0.
\end{align*}
Therefore, we obtain from the Markov property and Lemma~\ref{lem:basic-inequalities} that
\begin{align*}
\P_\mu(T<\tau_\d)&\geq \P_\mu(X_n\in K)\inf_{y\in K} \P_y(T-n<\tau_\d)\\
&\geq
a_5(D_\theta)\P_\mu(n<\tau_\d)\inf_{y\in K} \P_y(T-n<\tau_\d)\\
&\geq a_5(D_\theta)\theta_2^n \mu(\varphi_2)\inf_{y\in K} \P_y(T-n<\tau_\d).
\end{align*}

Combining this with~\eqref{eq:lem-expo-moment-mu-1}, we obtain that, for all $n\geq n_\mu$,
\begin{align*}
\P_\mu(n< T_K\text{ and }T<\tau_\d)\leq  \frac{C}{a_5(D_\theta)}\left(\frac{\theta_1}{\theta_2}\right)^n\frac{\mu(\varphi_1)}{\mu(\varphi_2)}\P_\mu(T<\tau_\d).
\end{align*}
Hence
\begin{align*}
\E_\mu\left(\theta^{- T_K\wedge T}\11_{ T_K\geq n_\mu,\ T<\tau_\d}\right)\leq C\frac{\mu(\varphi_1)}{\mu(\varphi_2)}
\P_\mu\left(T<\tau_\d\right).
\end{align*}
We deduce that
\begin{align*}
\E_\mu\left(\theta^{- T_K\wedge T}\11_{T<\tau_\d}\right)\leq \left(C\frac{\mu(\varphi_1)}{\mu(\varphi_2)}+\theta^{-n_\mu}\right)
\P_\mu\left(T<\tau_\d\right).
\end{align*}
Since $\theta^{-n_\mu}\leq \frac{\theta^{-(n_5(D_\theta)+1)}\mu(\varphi_1)}{D_\theta\mu(\varphi_2)}$, we have proved
Lemma~\ref{lem:expo-moment-mu}.

\subsection{Conclusion of the proof of~\eqref{eq:thm-main-zxzx} for the norm $\|\cdot\|_{TV(\varphi_1)}$}
\label{sec:pf-cor-main}

For all $n\geq 1$, we introduce the linear operator on $L^\infty(\varphi_1)$, defined for all $h\in L^\infty(\varphi_1)$ as
\begin{align}
\label{eq:def-R_n}
R_n h(x)=\E_x\left(h(X_n)\11_{ T_K\leq n<\tau_\d}\right),\quad\forall x\in E.
\end{align}
Note that this operator is well-defined since $|R_n h(x)|\leq \|h\|_{L^\infty(\varphi_1)}\,P_n\varphi_1(x)<\infty$. We first give some properties of $R_n$,
which can be seen as a bounded approximation of $P_n$ in $L^\infty(\varphi_1)$.

\begin{lem}
    \label{lem:lem-cor-eta}
    We have
    \begin{align*}
    \bar{R}:=\sup_{n\geq 1}\sup_{x\in E} R_n\varphi_1(x)<\infty, 
    \end{align*}
    and for all $n\geq 1$ and $x\in E$,
    \begin{align*}
    0\leq P_n\varphi_1(x)-R_n\varphi_1(x)\leq\theta_1^n\varphi_1(x).
    \end{align*}
\end{lem}

\begin{proof}
    Using Markov's property,
    \begin{align*}
    R_n\varphi_1(x) 
    & =\sum_{k\leq n}\E_x[\11_{ T_K=k}P_{n-k}\varphi_1(X_{k})]  \\ & \leq\sup_{y\in K,\ k\geq 0}P_k\varphi_1(y)\,\P_x( T_K\leq n) \\
    & \leq\sup_{y\in K,\ k\geq 0}\frac{P_k\varphi_1(y)}{P_k\varphi_2(y)}  \leq D_1\vee\sup_{y\in K}\frac{\varphi_1(y)}{\varphi_2(y)}<+\infty
    \end{align*}
    by Lemma~\ref{lem:borne-P_n-phi1/phi2}. This proves the first inequality. For the second one, we observe that for all $x\in E$,
    \begin{align*}
    P_n\varphi_1(x)-R_n\varphi_1(x)=\E_x(\varphi_1(X_n)\11_{n< T_K})\leq\theta_1^n\varphi_1(x)
    \end{align*}
    by Lemma~\ref{lem:basic-inequalities}.
\end{proof}

We fix $1\leq k\leq n$, $h$ such that $|h|\leq\varphi_1$ and $\mu$ such that $\mu(\varphi_1)/\mu(\varphi_2)\leq D_\theta$, where $\theta=\frac{1+\theta_1/\theta_2}{2}$ and $D_\theta$ is from Lemma~\ref{lem:borne-P_n-phi1/phi2}.
The inequality~\eqref{eq:thm-main-zxzx} with $\|\cdot\|_{TV}$ in place of $\|\cdot\|_{TV(\varphi_1)}$ and Lemma~\ref{lem:lem-cor-eta} entail
\begin{align*}
\left|\frac{\mu P_{n-k}R_k h}{\mu P_{n-k}\11_E}-\nu_{QSD}(R_kh)\right|\leq
C\alpha^{n-k}\frac{\mu(\varphi_1)}{\mu(\varphi_2)}\sup_{x\in E}|R_k h(x)|\leq CD_\theta\bar{R}
\alpha^{n-k}.
\end{align*}
The second inequality of Lemma~\ref{lem:lem-cor-eta} implies
\begin{align*}
|\nu_{QSD}[(P_k-R_k)h]|\leq\theta_1^k\nu_{QSD}(\varphi_1)
\end{align*}
and, by Lemma~\ref{lem:borne-P_n-phi1/phi2},
\begin{align*}
\frac{\mu P_{n-k}(P_k-R_k)h}{\mu P_{n-k}\11_E} & \leq \theta_1^k \frac{\mu P_{n-k}\varphi_1}{\mu P_{n-k}\varphi_2}\leq
\theta_1^k\left(D_\theta\vee\frac{\mu(\varphi_1)}{\mu(\varphi_2)}\right)=\theta_1^k D_\theta.
\end{align*}

Combining the last three inequalities and recalling that $\nu_{QSD}P_k h=\theta_0^k\nu_{QSD}(h)$, we obtain that, for some constant $C>0$,
\begin{align*}
\left|\frac{\mu P_{n}h}{\theta_0^k\mu P_{n-k}\11_E}-\nu_{QSD}(h)\right| &\leq  C\left(\alpha^{n-k}\theta_0^{-k}+(\theta_1/\theta_0)^k\right).
\end{align*}

Applying the last inequality to $h=\11_E$, we obtain
\[
\left|\frac{1}{\theta_0^k\mu P_{n-k}\11_E}-\frac{1}{\mu P_n\11_E}\right|\leq \frac{C(\alpha^{n-k}\theta_0^{-k}+(\theta_1/\theta_0)^k)}{\mu P_n\11_E}
\]
so that, using Lemma~\ref{lem:borne-P_n-phi1/phi2},
\begin{align*}
\left|\frac{\mu P_{n}h}{\theta_0^k\mu P_{n-k}\11_E}-\frac{\mu P_{n}h}{\mu P_{n}\11_E}\right| &
\leq C(\alpha^{n-k}\theta_0^{-k}+(\theta_1/\theta_0)^k)\,\frac{\mu P_{n}\varphi_1}{\mu P_{n}\11_E}
\leq CD_\theta(\alpha^{n-k}\theta_0^{-k}+(\theta_1/\theta_0)^k).
\end{align*}
Hence, for some $\bar\alpha<1$, for all $n\geq 0$,
\begin{align*}
\left|\frac{\mu P_{n}h}{\mu P_{n}\11_E}-\nu_{QSD}(h)\right|\leq C(\alpha^{n-k}\theta_0^{-k}+(\theta_1/\theta_0)^k)\leq C\bar\alpha^n.
\end{align*}

Finally, if $\mu(\varphi_1)/\mu(\varphi_2)>D_\theta$, then let $T=\left\lceil \frac{\ln(D_\theta\mu(\varphi_1)/\mu(\varphi_2))}{-\ln\theta}\right\rceil$, so that $\mu P_T\varphi_1/\mu P_T \varphi_2\leq D_\theta$ according to Lemma~\ref{lem:borne-P_n-phi1/phi2}. We deduce from the previous inequality applied to $\mu P_T/\mu P_T \11_E$ that, for all $n\geq 0$,
\begin{align*}
\left|\frac{\mu P_{T+n}h}{\mu P_{T+n}\11_E}-\nu_{QSD}(h)\right| \leq C\bar\alpha^n\leq C \bar\alpha^n \theta^{T-1}\frac{\mu(\varphi_1)}{\mu(\varphi_2)}
\end{align*}
while, using again Lemma~\ref{lem:borne-P_n-phi1/phi2}, we obtain, for all $n\in\{0,T-1\}$,
\begin{align*}
\left|\frac{\mu P_{n}h}{\mu P_{n}\11_E}-\nu_{QSD}(h)\right| \leq  D_\theta\vee\left(\theta^{n}\frac{\mu(\varphi_1)}{\mu(\varphi_2)}\right)+\nu_{QSD}(\varphi_1)\leq C\theta^{n}\frac{\mu(\varphi_1)}{\mu(\varphi_2)}.
\end{align*}
The last two inequalities conclude the proof of~\eqref{eq:thm-main-zxzx} with $\alpha=\bar{\alpha}\vee \theta$ and hence of Theorem~\ref{thm:main}.

\subsection{The case where $\PP_x(n<\tau_\d)=0$ for some $x\in E$ and $n\geq 1$}
\label{sec:general-case-pf-main}

In this section, we assume that $X$ satisfies assumption~(E), but we do not assume anymore that $\P_x(n<\tau_\d)>0$ for
all $x\in E$ and all $n\geq 1$. We introduce $\bar E=\{x\in E,\ \P_x(n<\tau_\d)>0\ \forall n\geq 0\}$ and $\bar{\bar E}=E\setminus
\bar E$. One immediately deduces from (E2) that, for all $x\in \bar{\bar E}$ and all $n\geq 0$, $\varphi_2(x)=0$ and $\P_x(X_n\in
K)=0$, and hence that $\delta_x P_n\varphi_1\leq \theta_1^n \varphi_1(x)$ by Lemma~\ref{lem:basic-inequalities}. In addition, one easily checks that the semi-group $P$ restricted to $\bar E\cup\{\partial\}$ still satisfies assumption~(E), and in particular~\eqref{eq:thm-main-zxzx} applies.

Let $\mu$ be a probability measure on $E$ such that $\mu(\varphi_2)>0$ and $\mu(\varphi_1)<+\infty$. Then, for all $n\geq 0$ and all $|h|\leq\varphi_1$,
\begin{align*}
\left|\mu P_n h-\nu_{QSD}(h)\mu P_n\11_E\right|
&\leq \left|\mu P_n h-\mu_{\rvert\bar E} P_n h\right|+\left|\mu_{\rvert\bar E} P_n h-\nu_{QSD}(h)\mu_{\rvert\bar E} P_n\11_E\right|\\
&\qquad+\nu_{QSD}(\varphi_1)\left|\mu_{\rvert\bar E} P_n \11_E-\mu P_n\11_E\right|.
\end{align*}
Each term can be bounded as follows:
\begin{align*}
&\left|\mu P_n h-\mu_{\rvert\bar E} P_n h\right|\leq \mu_{\rvert\bar{\bar E}} P_n \varphi_1\leq \theta_1^n \mu\varphi_1,\\
&\left|\mu_{\rvert\bar E} P_n h-\nu_{QSD}(h)\mu_{\rvert\bar E} P_n\11_E\right|\leq C\alpha^n \frac{\mu_{\rvert\bar E}(\varphi_1)}{\mu_{\rvert\bar E}(\varphi_2)}\mu_{\rvert\bar E} P_n\11_E\leq C\alpha^n \frac{\mu(\varphi_1)}{\mu(\varphi_2)}\mu P_n\11_E
,\\
&\nu_{QSD}(\varphi_1)\left|\mu_{\rvert\bar E} P_n \11_E-\mu P_n\11_E\right|\leq \nu_{QSD}(\varphi_1) \mu_{\rvert\bar{\bar E}} P_n \varphi_1\leq  \nu_{QSD}(\varphi_1)  \theta_1^n \mu\varphi_1.
\end{align*}
Since $\mu P_n\11_E\geq \theta_2^n\mu(\varphi_2)$, we deduce that
\begin{align*}
\left|\frac{\mu P_n h}{\mu P_n\11_E}-\nu_{QSD}(h)\right|\leq \left((\theta_1/\theta_2)^n+\nu_{QSD}(\varphi_1)(\theta_1/\theta_2)^n+C\alpha^n\right)\frac{\mu(\varphi_1)}{\mu(\varphi_2)}.
\end{align*}
This concludes the proof of~\eqref{eq:thm-main-zxzx} in the general case.

\section{Proof of the other results of Section~\ref{sec:main-discrete-time}}
\label{sec:pf-other-results}

The previous section ensures the existence of a quasi-stationary distribution $\nu_{QSD}$ such that $\nu_{QSD}(\varphi_1)<+\infty$ and $\nu_{QSD}(K)>0$. Denoting by $\theta_0$ its associated decay parameter, we observe that $\theta_2\leq\theta_0$, since Lemma~\ref{lem:basic-inequalities} entails that, for all $n\geq 1$,
\[
\theta_0^n=\P_{\nu_{QSD}}(n<\tau_\d)\geq\nu_{QSD}(K)\inf_{y \in K}\P_y(n<\tau_\d)\geq\nu_{QSD}(K)\theta_2^n\inf_{y\in K}\varphi_2(y).
\]

We begin to prove Theorem~\ref{thm:eta} in Section~\ref{sec:pf-thm-eta}, except for the exponential convergence in
$L^\infty(\varphi_1)$. We then prove Theorem~\ref{thm:Q-proc}  in Section~\ref{sec:pf-thm-Q-proc}. In Section~\ref{sec:end-eta-main},
we conclude the proof of Theorem~\ref{thm:eta} and prove Corollary~\ref{cor:quasi-comp}. We prove Corollary~\ref{cor:ratios} in Subsection~\ref{sec:proof-cor-2}.

\subsection{Proof of the existence of the eigenfunction $\eta$}%Theorem~\ref{thm:eta}}
\label{sec:pf-thm-eta}

In this section, we show that the limit~\eqref{eq:conv-eta-general} is well defined pointwise, $\nu_{QSD}(\eta)=1$,
$P_1\eta=\theta_0\eta$, $\eta$ is lower bounded away from $0$ on $K$ and $\eta\in
L^\infty(\varphi_1^{\log(1/\theta_0)/\log(1/\theta_1)})$.

For all $n\geq 0$ and $x\in E\cup\{\d\}$, let us denote
\[
\eta_n(x)=\theta_0^{-n}\P_x(n<\tau_\d)=\frac{\P_x(n<\tau_\d)}{\P_{\nu_{QSD}}(n<\tau_\d)}.
\]
By Lemma~\ref{lem:comp-surv}, for all $x\in E$,
\begin{align}
\eta_n(x) & \leq C \theta_0^{-n}\inf_{y\in K}\P_y(n<\tau_\d)\varphi_1(x) \notag \\
& \leq \frac{C}{\nu_{QSD}(K)}\theta_0^{-n}\P_{\nu_{QSD}}(n<\tau_\d)\varphi_1(x)=\frac{C\varphi_1(x)}{\nu_{QSD}(K)}.
\label{eq:pf-eta-prelim-1}
\end{align}
This implies that the sequence $(\eta_n)_{n\geq 0}$ is uniformly bounded in $L^\infty(\varphi_1)$.

For all probability measure $\mu$ on $E$ and for all $n,m\geq 0$, by Markov's property, 
\[
\mu(\eta_{n+m})=\mu(\eta_n)\E_\mu\left[\theta_0^{-m}\P_{X_n}(m<\tau_\d)\mid n<\tau_\d\right].
\]
Hence, by Theorem~\ref{thm:main}, for all $\mu$ such that $\mu(\varphi_2)>0$ and $\mu(\varphi_1)<+\infty$,
\begin{align*}
|\mu(\eta_{n+m})-\mu(\eta_{n})| & =\mu(\eta_n)\left| \E_{\mu}(\eta_m(X_n)\mid n<\tau_\d)-1 \right|
\\ & =\mu(\eta_n)\left| \E_{\mu}(\eta_m(X_n)\mid n<\tau_\d)-\nu_{QSD}(\eta_m)\right|
\\ & \leq C\mu(\varphi_1)\alpha^n\frac{\mu(\varphi_1)}{\mu(\varphi_2)}.
\end{align*}
For any $x\in E$, applying this result to $\mu=(\delta_x+\nu_{QSD})/2$, we deduce that
\begin{align*}
|\eta_{n+m}(x)-\eta_{n}(x)| &  \leq C\varphi_1(x)^2\alpha^n.
\end{align*}
This shows that $(\eta_{n}(x))_{n\geq 0}$ is a Cauchy sequence and hence that, for all $x\in E$,
\begin{align*}
\eta(x)=\lim_{n\to+\infty} \theta_0^{-n} \P_x(n<\tau_\d)
\end{align*}
and, by~\eqref{eq:pf-eta-prelim-1}, that $\eta\in L^\infty(\varphi_1)$.

Then, since $\eta_n$ is bounded in $L^\infty(\varphi_1)$, we deduce by dominated convergence that $\nu_{QSD}(\eta)=1$ and that, for
all $x\in E$,
\begin{align}
\label{eq:fin-relecture}
\delta_x P_1\eta=\lim_{n\to+\infty} \delta_xP_1\eta_n=\lim_{n\to+\infty} \theta_0\eta_{n+1}(x)=\theta_0\eta(x).
\end{align}
The fact that $\eta$ is lower bounded away from $0$ on $K$ is an immediate consequence of Lemma~\ref{lem:comp-surv}
(integrating~\eqref{eq:comp-surv-2} with respect to $\nu_{QSD}(\mathrm d x)$) and the fact that $\nu_{QSD}(\varphi_1)<+\infty$.

It only remains to prove that $\eta\in
L^\infty\left(\varphi_1^{\log\theta_0/\log\theta_1}\right)$.
% Because of~\eqref{eq:pf-cor-eta-2}, we only need to check that
%\begin{align*}
%|f(x)|\leq C\varphi_1(x)^{\log|\theta|/\log\theta_1},\quad \forall x\in E'.
%\end{align*}
To prove this, we use the operator $R_n$ introduced in~\eqref{eq:def-R_n}. By Lemma~\ref{lem:lem-cor-eta} and using the fact that $\eta\in L^\infty(\varphi_1)$, for all $x\in E$,
\begin{align*}
\eta(x) & =\theta_0^{-n} P_n \eta(x)\leq C \theta_0^{-n} \left[R_n\varphi_1(x)+(P_n-R_n)\varphi_1(x))\right]\\
& \leq C\bar{R} \theta_0^{-n}+C\left(\frac{\theta_1}{\theta_0}\right)^n\varphi_1(x).
\end{align*}
Applying this inequality for $n=\lfloor-\log\varphi_1(x)/\log\theta_1\rfloor$, we deduce
\begin{align*}
\eta(x)\leq C\exp\left(\frac{\log\varphi_1(x)}{\log \theta_1}\,\log \theta_0\right)\leq
C\varphi_1(x)^{\log \theta_0/\log\theta_1},
\end{align*}
which concludes the proof.

\subsection{Proof of Theorem~\ref{thm:Q-proc}}
\label{sec:pf-thm-Q-proc}

We start with Point~(i). We introduce $\Gamma_n=\11_{n<\tau_\d}$ and define for all $x\in E'$ and $n\geq 0$ the probability measure
\begin{align*}
Q^{\Gamma,x}_n=\frac{\Gamma_n}{\E_x\left(\Gamma_n\right)}\P_x,
\end{align*}
so that the $Q$-process exists if and only if $Q_n^{\Gamma,x}$ admits a proper limit when $n\rightarrow\infty$. For all $0\leq k\leq
n$, we have by the Markov property
\begin{align*}
\frac{\E_x\left(\Gamma_n\mid{\cal
        F}_k\right)}{\E_x\left(\Gamma_n\right)}=\frac{\11_{k<\tau_\d}\P_{X_k}\left(n-k<\tau_\d\right)}{\P_x\left(n<\tau_\d\right)}.
\end{align*}
By the pointwise convergence in~\eqref{eq:conv-eta-general} (proved in Subsection~\ref{sec:pf-thm-eta}), this converges almost surely
as $n\rightarrow+\infty$ to
\begin{align*}
M_k:=\11_{k<\tau_\d}\theta_0^{-k}\frac{\eta(X_k)}{\eta(x)}=\theta_0^{-k}\frac{\eta(X_k)}{\eta(x)},
\end{align*}
and $\E_x(M_k)=\theta_0^{-k}\frac{P_k\eta(x)}{\eta(x)}=1$. These two properties allow to apply the penalization's theorem of
Roynette, Vallois and Yor \cite[Theorem~2.1]{RoynetteValloisEtAl2006}, which implies that $M$ is a martingale under $\P_x$ and that
$Q_n^{\Gamma,x}(A)$ converges to $\E_x\left(M_k\11_{A}\right)$ for all $A\in{\cal F}_k$ when $n\rightarrow\infty$. This means that
$\Q_x$ is well defined and
\begin{align*}
\restriction{\frac{\mathrm d\Q_x}{\mathrm d\P_x}}{{\cal F}_k}=M_k.
\end{align*}
Note that the fact that $\eta(x)=0$ for all $x\in E\setminus E'$ implies that $(X_n,n\geq 0)$ is $E'$-valued
$\mathbb{Q}_x$-almost surely for all $x\in E'$.
The fact that $X$ is Markov under $(\QQ_x)_{x\in E'}$ and Point~(ii) can be easily deduced from the last formula (see
e.g.~\cite[Section\,6.1]{ChampagnatVillemonais2016b}).

It remains to prove Point~(iii).
We define the function $\psi=\varphi_1/\eta\times\|\eta\|_{L^\infty(\varphi_1)}$ on $E'$. Note that, since $\eta\in L^{\infty}(\varphi_1)$, $\psi$ is uniformly lower
bounded. Moreover, for all $x\in E'$, 
\begin{align*}
%\label{eq:lyap-Q-proc}
\widetilde{P}_1\psi(x)&=\frac{\theta_0^{-1}\|\eta\|_{L^\infty(\varphi_1)}}{\eta(x)}P_1\varphi_1(x)\leq \frac{\theta_1}{\theta_0}\psi(x)+\frac{c_2\|\eta\|_{L^\infty(\varphi_1)}}{\theta_0\eta(x)}\11_K(x)\leq \widetilde\theta \psi(x)+\widetilde{c},
\end{align*}
where $\widetilde\theta=\theta_1/\theta_0$ and
\[
\widetilde{c}=\frac{c_2\|\eta\|_{L^\infty(\varphi_1)}}{\theta_0\inf_K\eta}.
\] 
% Note that, according to Lemma~\ref{lem:comp-surv}, we have, for all $x\in E$ and $y\in K$,
% \begin{align}
% \label{eq:deftildec}
% \frac{\eta(x)}{\varphi_1(x)\eta(y)}\leq \widetilde c= \frac{C}{1-\widetilde \theta},
% \end{align} so that,
Hence, for all $x\in E$ and all $n\geq 1$,
\begin{align}
    \label{eq:lyap-Q-proc}
\widetilde{P}_{n}\psi(x)&\leq \widetilde\theta \widetilde{P}_{n-1}\psi(x)+\widetilde c\leq ... \leq {\widetilde\theta}^{n}\psi(x)+\frac{\widetilde c}{1-\widetilde\theta}.
\end{align}
%and, for all $x\in K$,
%\begin{align}
%\delta_x \widetilde{P}_{n_1}=\frac{\theta_0^{-{n_1}}}{\eta(x)} \delta_x P_{n_1}(\eta\cdot .)\geq \frac{c_1\theta_0^{-{n_1}}\nu(\eta)}{\eta(x)} \nu_\eta(\cdot\cap K),%\geq \frac{c_1\theta_0^{-{n_1}}}{c_3} \nu_\eta(\cdot\cap K),
%\end{align}
%where $\nu_\eta=\nu(\eta\cdot)/\nu(\eta)$.
Using Lemma~\ref{lem:basic-inequalities}, we have that, for all $x\in
E'$,
%\begin{align}
%\label{eq:moment-Q-process}
%\E^{\Q}_x\left(\theta^{- T_K}\right)\leq \sum_{k=1}^\infty \theta^{-k}\Q_x(T_K>k-1)\leq  \sum_{k=1}^\infty \theta^{-k}\left(\frac{\theta_1}{\theta_2}\right)^{k-1}\frac{\psi(x)}{\inf_{E'} \widetilde\psi}=\frac{\psi(x)}{(\theta-\frac{\theta_1}{\theta_2})\inf_{E'} \widetilde\psi}
%\end{align}
\begin{align}
\label{eq:nouvelle-preuve-eta-bis}
    \Q_x(T_K>n)=\E_x\left(\theta_0^{-n}\frac{\varphi_1(X_n)}{\eta(X_n)}\11_{T_K>n}\11_{X_n\in E'}\right)\leq \widetilde\theta^{n}\psi(x)
\end{align}
Now, choosing $m_K$ large enough so that $\sup_{x\in K}\widetilde\theta^{m_K}[\sup_K \psi+\tilde c/(1-\tilde\theta)]\leq 1/2$, we deduce that, for all $x\in K$ and all $n_0\geq 0$,
\begin{align}
\label{eq:nouvelle-preuve-eta}
    \Q_x(\exists n\in\{n_0,\ldots,n_0+m_K\},\ X_n\in K)\geq 1-\widetilde{\theta}^{m_K}\widetilde{P}_{n_0}\psi(x)\geq 1/2.
\end{align}
Now, let $n_K\geq 1$ be such that $\inf_{x\in K}\P_x(X_n\in K)>0$ for all $n\geq n_K$ (such a $n_K$ exists by (E1) and (E4), see the
proof of Lemma~\ref{lem:E'4}) and let 
\[
a:=\inf_{n\in \{n_K,\ldots,n_K+m_K\}} \inf_{x\in K}\P_x(X_n\in K)>0.
\]
%so that, for all $x\in K$ and all $n\in\{n_K,2n_K\}$, we have
%\begin{align*}
%\P_x(X_n\in K)\geq c_1\P_\nu(X_{n-n_1}\in K))\geq c_1 a
%\end{align*}
%and hence
so that, for all $x\in K$, all $n\in\{n_K,\ldots,n_K+m_K\}$ and all $A\subset E$ measurable,
\begin{align*}
\Q_x(X_{n+n_1}\in A)&\geq \Q_x(X_n\in K,X_{n+n_1}\in A)=\frac{\theta_0^{-n-n_1}}{\eta(x)}\E_{x}(\11_{X_n\in
  K}\E_{X_n}(\eta(X_{n_1})\11_{X_{n_1}\in A}))\\
&\geq \frac{\theta_0^{-n-n_1}}{\eta(x)} ac_1\nu(\eta\11_A)=\frac{\theta_0^{-n-n_1}\nu(\eta)}{\eta(x)} ac_1\nu_\eta(A)\geq \frac{ac_1}{c_3} \nu_\eta(A),
\end{align*}
where we used (E1) and (E3) and defined $\nu_\eta(\mathrm d x):=\frac{\eta(x)\nu(\mathrm d x)}{\nu(\eta)}$.
%Now, for all $x\in E'$, we have
%\begin{align*}
%\Q_x(X_n\in K)&\geq \Q_x(\exists k\in \{n-n_\nu-n_1,n-n_\nu\} \text{ s.t. }X_k\in K,\ X_n\in K)\\ 
%&\geq \Q_x(\exists k\in \{n-n_\nu-n_1,n-n_\nu\} \text{ s.t. }X_k\in K)\,c_1 a/c_3
%\end{align*}
We deduce from the last inequality and~\eqref{eq:nouvelle-preuve-eta} that, for all $n_0\geq 0$,
\[
\Q_x(X_{n_0+n_1+n_K+m_K}\in\cdot)\geq\sum_{n=n_0}^{n_0+m_K}\Q_x\left[\11_{X_n\in K}\Q_{X_n}(X_{n_0+n_1+n_K+m_K-n}\in\cdot)\right]\geq \frac{ac_1}{2c_3}\nu_\eta.
\]
Hence, for all $n\geq n_K+m_K+n_1$,
\begin{align*}
\Q_x(X_{n}\in \cdot)\geq \frac{ac_1}{2c_3}\nu_\eta,\quad \forall x\in K.
\end{align*}

For all $x\in E'$, setting $k_x=  \lceil\frac{\log(2\psi(x))}{-\log\widetilde\theta}\rceil$, it follows from~\eqref{eq:nouvelle-preuve-eta-bis} that
$
\QQ_x( T_K\leq k_x)\geq\frac{1}{2},
$
and hence
\[
\QQ_x( X_{k_x+n_K+m_K+n_1}\in\cdot)\geq \frac{ac_1}{4c_3}\nu_\eta.
\]
In particular, for all $R>0$, setting $k_R=  \lceil\frac{\log(2R)}{-\log\widetilde\theta}\rceil+n_K+m_k+n_1$, we have, for all $x,y\in E'$ such that $\psi(x)+\psi(y)<R$,
\begin{equation}
\label{eq:doeblinproc}
\left\|\delta_x P_{k_R}-\delta_y P_{k_R}\right\|_{TV}\leq 1-\frac{ac_1}{4c_3}.
\end{equation}

By \cite[Thm\,3.9]{Hairer2010}, together with~\eqref{eq:lyap-Q-proc}, the last assertion implies that there exist
constants $C>0$ and $\widetilde{\alpha}_1\in(0,1)$ such that, for all real function $h$ on $E'$ such that $\vertiii{h}<\infty$,
\begin{align}
\label{eq:canbemadeexplicit}
\vertiii{\widetilde{P}_n h}\leq C\widetilde{\alpha}_1^{n}\vertiii{h},
\end{align}
where
\[
\vertiii{h}=\sup_{x,y\in E'}\frac{|h(x)-h(y)|}{2+\psi(x)+\psi(y)}.
\]
This implies~\eqref{eq:Q-proc-11}.
In particular, for all $x\in E'$,
\[
\|\delta_x\widetilde{P}_n-\beta\|_{TV}\xrightarrow[n\rightarrow +\infty]{} 0.
\]
Hence,~\eqref{eq:Q-proc-12} is a consequence of Lebesgue's dominated convergence theorem. This ends the proof of
Theorem~\ref{thm:Q-proc}.

\begin{rem}
    \label{rem:explici-constants}
    As noted in~\cite[Remark~3.10]{Hairer2010}, it is possible to obtain explicit constants $\widetilde{C}$ and
    $\widetilde{\alpha}_1$ in~\eqref{eq:canbemadeexplicit} from the parameters in~\eqref{eq:lyap-Q-proc} and \eqref{eq:doeblinproc}
    (note that a slight modification of the proof of Lemma~\ref{lem:comp-surv} entails that one can actually take
    $\widetilde c=\frac{1/\theta_1}{1-\theta_1/\theta_0}\leq \frac{1/\theta_1}{1-\theta_1/\theta_2}$). More precisely, setting
    $\alpha=\frac{ac_1}{4c_3}$ and $K=\widetilde c/(1-\widetilde\theta)$ and $\gamma=\widetilde{\theta}$, then taking any
    $\alpha_0\in(0,\alpha)$ and $R>\frac{2K}{1-\gamma}$, and setting $b=\frac{2\alpha_0}{\gamma R+2K}$,
    \[
      \alpha_R=(1-\alpha+\alpha_0)\vee \frac{2+b\gamma R+b\gamma K}{2+b R}\in(0,1),
    \]
    and $C_R=\frac{2/b+1+K+K/(1-\gamma)}{\alpha_R}$,
    we obtain that, for all $f\in L^\infty(\varphi_1/\eta)$,
    \[
      \left|\widetilde P_n f-\beta(f)\right|\leq C_R\alpha_R^{n/k_R}\|f\|_{L^\infty(\varphi_1/\eta)}.
    \]
    % This of course immediately extends to $f\in L^\infty(\varphi_1/\eta)$ (with $\|f\eta/\varphi_1\|_\infty$ on the right hand
    % side).
\end{rem}

\subsection{Proof of Corollary~\ref{cor:quasi-comp} and end of the proof of Theorem~\ref{thm:eta}}
\label{sec:end-eta-main}

Let $|g|\leq\varphi_1$ and set $h=g/\eta$. Then~\eqref{eq:Q-proc-11} entails that, for all $x\in E'$,
\begin{align*}
\left|\theta_0^{-n}\E_x(g(X_n)\11_{X_n\in E'})-\eta(x)\nu_{QSD}(g\11_{E'})\right|\leq C\,\bar\alpha^n\varphi_1(x).
\end{align*}
%Up to a change in the constant $\bar\alpha$, this property remains true for all $x\in E'':=E\setminus E'$ since, in this situation, $\eta(x)=0$ and $\E_x(g(X_n)\11_{X_n\in E'})=0$.
In what follows, we set $\nu'=\nu_{QSD}(\cdot\cap {E'})$ and, for all $k\geq 1$,
\[
g_k(x)=\11_{x\in E'}\EE_x\left(\11_{X_1\notin E'}g(X_k)\11_{k<\tau_\d}\right).
\]
Note that, defining $E'':=E\setminus E'$, $E''\cup\{\d\}$ is an absorbing set. Since $K\subset E'$, we thus have
 \[
 g_k(x)\leq \11_{x\in E'}\,\EE_x\left(\11_{k<T_K\wedge\tau_\d}\varphi_1(X_k)\right)\leq \theta_1^k\varphi_1(x).
 \]
 We also define the measure $\nu''$ on $E''$ by
 \[
 \nu''=\sum_{\ell\geq 1} \theta_0^{-\ell}\EE_{\nu'}\left(\11_{X_1\notin E'}\11_{X_\ell\in\cdot}\right)=\sum_{\ell\geq 1} \theta_0^{-\ell}\nu'(g_\ell).
 \]
Hence we have for all $n\geq 1$ and all $x\in E'$
\begin{align*}
\left|\theta_0^{-n}\E_x(g(X_n)\11_{X_n\in E''})-\eta(x)\nu''(g)\right|
&\leq \sum_{\ell=1}^n \left|\theta_0^{-n}\E_x(g_{\ell}(X_{n-\ell}))-\theta_0^{-\ell}\eta(x)\nu'(g_\ell)\right|\\
&\leq \sum_{\ell=1}^n \theta_0^{-\ell}\left|\theta_0^{-(n-\ell)}\E_x(g_{\ell}(X_{n-\ell}))-\eta(x)\nu'(g_\ell)\right|\\
&\leq \sum_{\ell=1}^n \theta_0^{-\ell}\varphi_1(x)C\bar\alpha^{n-\ell}\|g_\ell\|_{L^\infty(\varphi_1)}\\
&\leq C\varphi_1(x)\sum_{\ell=1}^n \left(\frac{\theta_1}{\theta_0}\right)^{\ell}\bar\alpha^{n-\ell}.
\end{align*}
We thus proved that, setting $\nu_0=\nu'+\nu''$, and up to a change in the constants $C$ and $\bar\alpha$, for all $x\in E'$,
\begin{align*}
\left|\theta_0^{-n}\E_x(g(X_n)\11_{n<\tau_\d})-\eta(x)\nu_0(g)\right|\leq C\,\bar\alpha^n\varphi_1(x).
\end{align*}
Now, by Lemma~\ref{lem:basic-inequalities}, for all $x\in ''$, $|\E_x g(X_n)\11_{n<\tau_\d}|\leq \theta_1^n\varphi_1(x)$.
This, we have proved that, up to a change in $\bar\alpha$, for all $x\in E$,
\begin{align}
\label{eq:importantstep}
\left|\theta_0^{-n}\E_x(g(X_n))-\eta(x)\nu_0(g)\right|\leq C\,\bar\alpha^n\varphi_1(x).
\end{align}
Integrating with respect to $\nu_{QSD}$ shows that $\nu_0=\nu_{QSD}$. 
Thus, we have proved~\eqref{eq:quasi-comp}. % concludes the proof of Corollary~\ref{cor:quasi-comp}.

To conclude, we can now end the proof of Theorem~\ref{thm:eta}. Indeed, taking $g\equiv 1$ immediately entails that the
convergence~\eqref{eq:conv-eta-general} is geometric in $L^\infty(\varphi_1)$.

\subsection{Proof of Corollary~\ref{cor:attraction-domain}}
\label{sec:proof-attraction-domain}

If $\eta(x)>0$, it follows from Corollary~\ref{cor:quasi-comp} and the fact that $\nu_{QSD}(\varphi_2)>0$ that there exists $k\geq 0$
such that $\delta_xP_k\varphi_2>0$. Hence $E'\subset\{x\in E:\exists k\geq 0,\ P_k\varphi_2(x)>0\}$. Conversely, if
$P_k\varphi_2(x)>0$, we apply Theorem~\ref{thm:main} to $\mu=\frac{\delta_x P_k}{\delta_x P_k\mathbbm{1}_E}$. Since
$\nu_{QSD}(\eta)>0$, there exists $n\geq 0$ such that
$0<\frac{\mu P_n\eta}{\mu P_n\11_E}=\frac{\delta_x P_{n+k}\eta}{\delta_x P_{n+k}\11_E}=\frac{\theta_0^{n+k}}{\delta_x
  P_{n+k}\11_E}\eta(x)$. Hence we have proved that $E'=\{x\in E:\exists k\geq 0,\ P_k\varphi_2(x)>0\}$.

The fact that any $\mu$ such that $\mu(E')>0$ and $\mu(\varphi_1^{1/p})<+\infty$ for some $p<\log\theta_1/\log\theta_2$
belongs to the domain of attraction of $\nu_{QSD}$ follows from Remark~\ref{rem:dom-attraction} and Corollary~\ref{cor:quasi-comp}.

In the case where $\varphi_1$ is bounded, the domain of atttraction contains all measures $\mu$ such that $\mu(E')>0$. If
$\mu(\eta)=0$, then $\mu P_k \eta=0$ for all $k\geq 0$, which means that $\mu P_k$ gives no mass to $E'$. Hence the convergence of
conditional distributions to $\nu_{QSD}$ cannot hold true. The uniqueness of the quasi-stationary distribution follows immediately.

\subsection{Proof of Corollary~\ref{cor:ratios}}
\label{sec:proof-cor-2}

Applying~\eqref{eq:Q-proc-12} with $\mu(\eta\cdot)/\mu(\eta)$ instead of $\mu$ and recalling that
  $\mu(E\setminus E')=0$, we obtain
  \[
    \sup_{f:E'\to\mathbb{R},\,\|f\|_\infty\leq 1}\left|\theta_0^{-n}\frac{\mu P_n(\eta f)}{\mu(\eta)}-\beta(f)\right|\xrightarrow[n\rightarrow+\infty]{}0.
  \]
This entails the convergence result~\eqref{eq:statementcor2}.

Assume from now on that $\eta$ is positive on $E$. Let $\nu$ be a quasi-stationary distribution on $E$ such that $\nu(\eta)<+\infty$ and
denote by $\bar\theta_0\in(0,1]$ the associated decay parameter, such that $\P_{\nu}(X_n\in\cdot)=\bar\theta_0^n\nu$ for all $n\geq 0$. Then, according to~\eqref{eq:statementcor2}, we have, for all $g\in L^\infty(\eta)$,
\begin{align*}
\left|\theta_0^{-n}\bar\theta_0^n\nu(g)-\nu(\eta)\nu_{QSD}(g)\right|\xrightarrow[n\to+\infty]{}0.
\end{align*}
This entails that $\bar\theta_0=\theta_0$ and that $\nu$ is proportional to $\nu_{QSD}$. Since they both are probability measures, we deduce that $\nu=\nu_{QSD}$, which concludes the proof of the second claim of Corollary~\ref{cor:ratios}.

Finally, assuming that $\eta$ is lower bounded away from $0$ on $E$, we deduce from~\eqref{eq:statementcor2} with $g\equiv 1$ that, for all probability measure $\mu$ on $E$ such that $\mu(\eta)<+\infty$,
\begin{align*}
\theta_0^{-n} \mu P^n \11_E\xrightarrow[n\to+\infty]{} \mu(\eta)>0.
\end{align*}
This and~\eqref{eq:statementcor2} imply that
  \[
    \sup_{g:E\to\mathbb{R},\,\|g\|_{L^\infty(\eta)}\leq 1}\left|\mathbb{E}_\mu(g(X_n)\mid
      n<\tau_\d)-\nu_{QSD}(g)\right|\xrightarrow[n\to+\infty]{}0,
  \]
  hence~\eqref{eq:statementcor2bis} holds true and the proof of Corollary~\ref{cor:ratios} is completed.

\section{Proof of the results of Section~\ref{sec:equivalent-formulations}}
\label{sec:proof-equiv-form}

In this section are proved Lemma~\ref{lem:varphi2} in Subsection~\ref{sec:proof-varphi2}, Lemma~\ref{lem:varphi1} in
Subsection~\ref{sec:proof-varphi1}, Proposition~\ref{lem:E2thenE} in Subsection~\ref{sec:proof-E2thenE} and
Lemma~\ref{lem:theta-2-et-0} in Subsection~\ref{sec:proof-theta-2-et-0}. Then we prove Theorem~\ref{thm:E-F} in
Subsection~\ref{sec:proof-E-F}, Lemma~\ref{prop:E-G} in Subsection~\ref{sec:proof-E-G} and
Proposition~\ref{prop:A1-A2} in Subsection~\ref{sec:proof-A1-A2}.

\subsection{Proof of Lemma~\ref{lem:varphi2}}
\label{sec:proof-varphi2}

The function $\varphi_2$ defined in the statement satisfies, for all $x\in E$, $\varphi_2(x)\in[0,1]$ and, for all $x\in K$, $\varphi_2(x)\geq \frac{\theta^{-1}_2-1}{\theta_2^{-\ell}-1}>0$. Moreover, we have, for all  $x\in E$,
\begin{align*}
P_1\varphi_2(x)&=\theta_2\varphi_2(x)-\frac{\theta^{-1}_2-1}{\theta_2^{-\ell}-1}\,\left(\theta_2\11_K(x)-\theta_2^{-\ell+1}P_{\ell}\11_K(x)\right)\geq
\theta_2\varphi_2(x)
\end{align*}
since $\ell$ is chosen such that $\theta_2^{-\ell}P_{\ell}\11_K(x)\geq \11_K(x)$ for all $x\in E$.

Our assumption also implies that there exists $n_0$ such that, for all $n\geq n_0$, $\theta_2^{-n}\inf_{x\in K}\P_x(X_n\in K)\geq 1$. Choosing $n_4(x)=n_0$ for all $x\in K$ entails~(E4), which concludes the proof of Lemma~\ref{lem:varphi2}.

\subsection{Proof of Lemma~\ref{lem:varphi1}}
\label{sec:proof-varphi1}

Assume that 
\[
\E_x\left(\theta_1^{- T_K\wedge\tau_\d}\right)<+\infty\,\ \forall x\in E\text{ and } \sup_{y\in K}\,
\E_y\left(\E_{X_1}\left(\theta_1^{- T_K\wedge\tau_\d}\right)\11_{1<\tau_\d}\right)<+\infty
\]
and set $\varphi_1(x)=\E_x\left(\theta_1^{- T_K\wedge\lceil\tau_\d\rceil}\right)$ for all $x\in E$. Then, for all $x\in E\setminus K$, using Markov's property at time $1$,
\begin{align*}
P_1\varphi_1(x)&=\E_x\left(\E_{X_1}\left(\theta_1^{- T_K\wedge\lceil\tau_\d\rceil}\right)\11_{1<\tau_\d}\right)\leq\E_x\left(\theta_1^{-( T_K\wedge\lceil\tau_\d\rceil-1)}\right)=\theta_1\varphi_1(x).
\end{align*}
Moreover, for all $x\in K$, $P_1\varphi_1(x)\leq \theta_1^{-1}\sup_{y\in K} \E_y\left(\E_{X_1}\left(\theta_1^{- T_K\wedge\tau_\d}\right)\11_{1<\tau_\d}\right)$, and hence the first part of the lemma is proved.

Assume now that there exist two constants $C>0$, $\theta_1>0$ and a function $\varphi_1:E\to [1,+\infty)$ such that $\sup_{K}\varphi_1<+\infty$ and $P_1\varphi_1\leq \theta_1\varphi_1+C\11_K$. Then, for all $n\geq 1$ and all $x\in E\setminus K$, 
\[
\E_x\left(\varphi_1(X_n)\11_{n< T_K\wedge\tau_\d}\right)\leq \theta_1^n \varphi_1(x).
\]
Thus we deduce that, for all $x\in E$,
\[
\P_x\left(n< T_K\wedge\tau_\d\right)\leq \theta_1^n \varphi_1(x).
\]
In particular, for all $\theta>\theta_1$ and all $x\in E$,
\begin{align*}
\E_x\left(\theta^{- T_K\wedge\tau_\d}\right)\leq \frac{1}{\theta-\theta_1}\,\varphi_1(x)<+\infty.
\end{align*}
We also deduce that
\begin{align*}
\sup_{x\in K} \E_x\left(\E_{X_1}\left(\theta^{- T_K\wedge \tau_\d}\right)\right)\leq \frac{1}{\theta-\theta_1} \sup_{x\in K} P_1\varphi_1(x)<+\infty.
\end{align*}
This concludes the proof of Lemma~\ref{lem:varphi1}.

\subsection{Proof of Proposition~\ref{lem:E2thenE}}
\label{sec:proof-E2thenE}

Condition~(E4) implies that there exists $x_0\in E$ such that
$\P_{x_0}(X_{n_0}\in K)>0$. We then deduce from our assumption~\eqref{eq:condition-E2thenE} that Condition~(E1) is satisfied with the
probability measure  $\nu$ on $K$ defined by
\[
\nu(\cdot)=\frac{\P_{x_0}(X_{n_0}\in \cdot\cap K)}{\P_{x_0}(X_{n_0}\in K)}
\]
and the constants $c_1=\P_{x_0}(X_{n_0}\in K)/C>0$ and $n_1=m_0$.

Let us now check Condition~(E3) and the last part of Proposition~\ref{lem:E2thenE}. We define $ T_K^{(n_0)}= \inf\{n\geq n_0\text{
    s.t. }X_n\in K\}$. Lemma~\ref{lem:basic-inequalities} (which only makes use of Condition~(E2)) implies that, for all $x\in E$,
$\P_x(n< T_K\wedge\tau_\d)\leq \theta_1^{n}\varphi_1(x)$. Hence, for all $x\in E$ and all $n\geq n_0$,
\begin{align*}
\P_x(n<\tau_\d\wedge T_K^{(n_0)})&= \E_x\left(\11_{n_0<\tau_\d}\P_{X_{n_0}}(n-n_0<\tau_\d\wedge T_K)\right)  \\
&\leq \theta_1^{n-n_0}\,\E_x\left(\11_{n_0<\tau_\d}\varphi_1(X_{n_0})\right)  \\
&\leq (\theta_1+c_2)^{n_0} \theta_1^{n-n_0}\,\varphi_1(x).
\end{align*}
Since $\varphi_1\geq 1$, we also have $\P_x(n<\tau_\d)\leq C\theta_1^n\varphi_1(x)$ for all $n<n_0$. Hence we
proved that, for all $x\in E$ and $n\geq 0$, 
\begin{equation}
  \label{eq:pf-E2thenE-eq}
  \P_x(n<\tau_\d\wedge T_K^{(n_0)})\leq C\theta_1^n\varphi_1(x).
\end{equation}
Therefore, for some constant $C>0$,
\begin{align}
\P_x(n<\tau_\d)&\leq \P_x(n<\tau_\d\wedge T^{(n_0)}_K)+\P_x(T^{(n_0)}_K\leq n <\tau_\d)\nonumber \\
&\leq C\,\varphi_1(x)\theta_1^n+\sum_{k=n_0}^{n}\E_x\left(\11_{T^{(n_0)}_K=k} \P_{X_k}(n-k<\tau_\d)\right)\label{eq:proof-E2thenE-1}.
\end{align}
Now, for all $x\in E$, all $y\in K$ and all $k\in\{n_0,\ldots,n\}$,~\eqref{eq:condition-E2thenE} and~\eqref{eq:pf-E2thenE-eq} entail
\begin{align*}
\E_x\left(\11_{T^{(n_0)}_K=k} \P_{X_k}(n-k<\tau_\d)\right)&\leq \E_x\left(\11_{k-n_0< T_K^{(n_0)}\wedge\tau_\d} \E_{X_{k-n_0}}\left(\11_{X_{n_0}\in K}\,\P_{X_{n_0}}(n-k<\tau_\d)\right)\right)\\
&\leq \E_x\left(\11_{k-n_0<T^{(n_0)}_K\wedge\tau_\d} C\,\P_y(n+m_0-k<\tau_\d)\right)\\
&\leq \theta_1^{k-n_0}\varphi_1(x)\,C\,\P_y(n-k<\tau_\d),
\end{align*}
where the constant $C$ may change from line to line.
Using Lemma~\ref{lem:to-be-in-K}, which only makes use of (E1), (E2) and (E4), there exists $n_6\in\Z_+$
such that, for all $y\in K$ and for all $n,k\in\Z_+$ such that $n-k\geq n_6$,
\begin{align*}
\P_y(n<\tau_\d)& \geq \P_y(X_{n-k}\in K)\inf_{z\in K}\P_z(k< \tau_\d)\\
& \geq \P_y(n-k<\tau_\d)\,\inf_{T\geq n_6}\inf_{z\in K}\P_z(X_T\in K\mid T<\tau_\d)\,\inf_{z\in K} P_{k}\varphi_2(z)\\
& \geq C'' \theta_2^{k}\,\P_y(n-k<\tau_\d),
\end{align*}
where $C''>0$. Hence,
\[
\E_x\left(\11_{T^{(n_0)}_K=k} \P_{X_k}(n-k<\tau_\d)\right)\leq \varphi_1(x)\,\left(\frac{\theta_1}{\theta_2}\right)^k\,\frac{\theta_1^{-n_0}\,C}{C''}\,\P_y(n<\tau_\d).
\]
Now, we deduce from~\eqref{eq:proof-E2thenE-1} and~\eqref{eq:pf-E2thenE-eq} that, for all $x\in E$ and all $y\in K$,
\begin{align*}
\P_x(n<\tau_\d)&\leq C\, \varphi_1(x)\,\left[\theta_1^n+\P_y(n<\tau_\d)\,
\,\sum_{k=1}^{n-n_6}\left(\frac{\theta_1}{\theta_2}\right)^k\right]
+\PP_x( T_K^{(n_0)}\wedge\tau_\d\geq n-n_6)\\
&\leq C\, \varphi_1(x)\,\left[\theta_1^n+\P_y(n<\tau_\d)
+\theta_1^{n-n_6}\right]\\
&\leq C\,\varphi_1(x)\,\P_y(n<\tau_\d)
\end{align*}
since $\P_y(n<\tau_\d)\geq \theta_2^n\,\inf_K\varphi_2$. This implies (E3) since
$\sup_K\varphi_1<\infty$.

\subsection{Proof of Lemma~\ref{lem:theta-2-et-0}}
\label{sec:proof-theta-2-et-0}

Combining Theorem~\ref{thm:eta} and the fact that $\inf_K\eta>0$, we deduce that
\[
\liminf_{n\rightarrow+\infty}\ \inf_{x\in K}\theta_0^{-n}\PP_x(n<\tau_\d)>0.
\]
Let $\theta'_2<\theta_0$. Using Lemma~\ref{lem:to-be-in-K},
\[
\lim_{n\rightarrow+\infty}\, (\theta'_2)^{-n}\,\inf_{x\in K}\PP_x(X_n\in K)=+\infty.
\]
Hence the result follows from Lemma~\ref{lem:varphi2}.

\subsection{Proof of Theorem~\ref{thm:E-F}}
\label{sec:proof-E-F}

We assume that Assumption~(F) is satisfied. In Subsection~\ref{sec:proof-(E)}, we prove that Assumption~(E) holds true for the
sub-Markovian semigroup $(P_n)_{n\geq 0}$ of the absorbed Markov process $(X_{nt_2},n\in\ZZ_+)$. In Subsection~\ref{sec:full-QSD}, we
prove the existence of a quasi-stationary distribution for $(X_t)_{t\in I}$ with the claimed properties and in
Subsection~\ref{sec:full-eta}, we prove the convergence of $e^{\lambda_0 t}\PP_x(t<\tau_\d)$ to $\eta(x)$ for $t\in I$,
$t\rightarrow+\infty$.

\subsubsection{Proof of (E)}
\label{sec:proof-(E)}

We fix $\theta_1\in(\gamma_1^{t_2},\gamma_2^{t_2})$ and set $\theta_2=\gamma_2^{t_2}$. Let us first remark that the last line of
Condition~(F2) implies that $\gamma_2^{-t}\P_\nu(X_{t}\in L)\rightarrow +\infty$ when $t\rightarrow+\infty$. Hence, using
Condition~(F1), we deduce that
\begin{align}
\label{eq:proof-E-F-1}
\gamma_2^{-t}\inf_{x\in L}\P_x(X_{t}\in L)\xrightarrow[t\rightarrow+\infty]{} +\infty.
\end{align}
We consider a number $n_0\in\N^*$ large enough so that $\gamma_2^{-t}\inf_{x\in L}\P_x(X_{t}\in L)\geq
1\vee\frac{c_2}{\theta_1-\gamma_1^{t_2}}$, for all $t\geq (n_0-1)t_2$ and we set
\[
\varphi_1=\psi_1\quad\text{ and }\quad \varphi_2=\frac{\gamma_2^{-t_2}-1}{\gamma_2^{-n_0t_2}-1}\sum_{k=0}^{n_0-1}\gamma_2^{-kt_2}P_k\11_L.
\]

\medskip\noindent\textit{Step 1. Proof of~(E2), (E4) and (E1) for $(P_n)_{n\in\Z_+}$.}

For all $x\in E\setminus L$, it follows from (F0) and the second line of~(F2) that
\begin{align*}
P_{1}\psi_1(x)&=\E_x\left(\psi_1(X_{t_2})\11_{t_2<\tau_L\wedge\tau_\d}\right)+\E_x\left(\11_{\tau_L\leq
    t_2\wedge\tau_\d}\restriction{\E_{X_{\tau_L}}(\11_{t_2-s<\tau_\d}\psi_1(X_{t_2-s}))}{s=\tau_L}\right)\\
&\leq \gamma_1^{t_2}\psi_1(x)+\P_x(\tau_L\leq t_2) c_2.
\end{align*}

We define 
\[
K=\left\{y\in E,\ \P_y(\tau_L\leq t_2)/\psi_1(y)\geq (\theta_1-\gamma_1^{t_2})/c_2\right\}.
\]
The second line of (F2) at time $t=0$ and the fact that $\theta_1-\gamma_1^{t_2}<1$ imply that $L\subset K$. Moreover, we have, for
all $x\notin K$,
\begin{align}
P_1\psi_1(x)&\leq  \theta_1\psi_1(x). \label{eq:pf-E-F-5}
\end{align}
Hence, for all $x\in E$,
\begin{equation}
\label{eq:E2-for-phi_1-in-(F)}
P_1\psi_1(x)\leq\theta_1\psi_1(x)+c_2\11_K(x).
\end{equation}
Note that it immediately follows from the definition of $K$ that $\sup_{x\in K}\psi_1(x)<\infty$. In particular, the first and third
lines of (E2) are proved.

Moreover, using the Markov property provided by (F0) and the definition of $n_0$, we deduce that, for all $t\geq n_0 t_2$,
\begin{align}
\label{eq:borne-psi-sur-K}
\inf_{x\in K}\gamma_2^{-t} \P_x(X_t\in L)&\geq \inf_{x\in K}\P_x(\tau_L\leq t_2)\inf_{s\in[0,t_2]}\inf_{y\in L}\gamma_2^{-t}\P_{y}\left(X_{t-s}\in L\right)\geq 1,
\end{align}
where we used the fact that, for all $x\in K$,
$\P_x(\tau_L\leq t_2)\geq \frac{\theta_1-\gamma_1^{t_2}}{c_2}$. In particular,
\begin{align*}
P_1\varphi_2=\gamma_2^{t_2}\varphi_2+\frac{\gamma_2^{-t_2}-1}{\gamma_2^{-n_0t_2}-1}\left(\gamma_2^{-(n_0-1) t_2} P_{n_0}\11_L -\gamma_2^{t_2}\11_L\right)\geq \gamma_2^{t_2}\varphi_2=\theta_2\varphi_2.
\end{align*}
In addition, for all $x\in K$, 
\[
\varphi_2(x)\geq \frac{\gamma_2^{-t_2}-1}{\gamma_2^{-n_0t_2}-1}
\gamma_2^{-(n_0-1)t_2}\P_x(X_{n_0t_2}\in L)\geq \frac{1-\gamma_2^{t_2}}{\gamma_2^{-n_0t_2}-1}.
\]
Hence (E2) is proved. Moreover, \eqref{eq:borne-psi-sur-K} also entails that (E4) holds true.

Fix $n_1>n_0$ such that
$n_1t_2-t_1\geq n_0t_2$. Condition~(F1) and then~\eqref{eq:borne-psi-sur-K} imply that, for all $x\in K$,
\begin{align*}
\P_x(X_{n_1t_2}\in\cdot\cap K)\geq \P_x(X_{n_1t_2-t_1}\in L)c_1\nu(\cdot\cap L)\geq \gamma_2^{n_1t_2-t_1}c_1\nu(\cdot\cap L).
\end{align*}
Extending $\nu$ as a probability measure on $K$, we obtain (E1).

\medskip\noindent\textit{Step 3. Estimation of the survival probability.}

Our goal here is to prove a version of Lemma~\ref{lem:comp-surv}, where~\eqref{eq:comp-surv-2} is replaced by
\begin{align}
\label{eq:pf-F-E-2}
\P_x(nt_2<\tau_\d)\leq C\frac{\varphi_1(x)}{1-\theta_1/\theta_2}\inf_{y\in L}\P_y(nt_2<\tau_\d),\quad\forall x\in E, \forall n\in\N.
\end{align}
Since the proof is similar, we only highlight the main
differences. First, Lemma~\ref{lem:to-be-in-K} only uses (E1), (E2)
and (E4), so that there exist $n_6\geq 1$ and $\zeta_1>0$ such that, for all $x\in K$ and all $n\geq n_6$,
\[
\delta_x P_n \11_K\geq \zeta_1 \delta_x P_n\11_E.
\]
Hence,  for all $x\in K$ and all $N\geq n_0+n_6$, using~\eqref{eq:borne-psi-sur-K},
\begin{align*}
\delta_x P_N\11_L & \geq \gamma_2^{n_0 t_2}\,\delta_x P_{N-n_0}\11_K 
\geq \zeta_1\gamma_2^{n_0 t_2}\,\delta_x P_{N-n_0}\11_E
\geq \zeta_1\gamma_2^{n_0 t_2}\,\delta_x P_{N}\11_E.
\end{align*}
Hence,
\begin{equation}
\label{eq:pf-E-F-3-ter}
\inf_{N\geq n_0+n_6}\inf_{x\in K}\P_x(X_{N t_2}\in L\mid N t_2<\tau_\d)>0.  
\end{equation}
Third, it follows from (F2) that, for all $x\in E\setminus L$,
\begin{equation}
\label{eq:pf-E-F-3}
\P_x(n t_2<\tau_L\wedge\tau_\d)\leq \gamma_1^{nt_2} \psi_1(x)=\theta_1^n\varphi_1(x).
\end{equation}
and from~(E2) that, for all $x\in E$,
\begin{equation}
\label{eq:pf-E-F-3bis}
\PP_x(nt_2<\tau_\d)\geq\gamma_2^{nt_2}\varphi_2(x).
\end{equation}
Therefore, following the same lines as in~\eqref{eq:pf-comp-survv} (replacing $K$ with $L$), we deduce
from~\eqref{eq:pf-E-F-3} and~\eqref{eq:pf-E-F-3bis} that, for all $x\in E$
\begin{align*}
\PP_x(nt_2<\tau_\d) & \leq \theta_1^n\varphi_1(x)+c_3\int_0^{n t_2}\inf_{y\in
    L}\PP_y\left(\left(n-\lceil s/t_2\rceil\right)t_2<\tau_\d\right)\,\PP_x(\tau_L\wedge \tau_\d \in \mathrm ds) \\ & \leq C\inf_{z\in
    L}\PP_z(nt_2<\tau_\d)\varphi_1(x)+ \frac{c_3\gamma_2^{-t_2}}{c}\inf_{z\in L}\PP_z(n
t_2<\tau_\d)\EE_x\left(\gamma_2^{-\tau_L\wedge\tau_\d}\right),
\end{align*}
which entails~\eqref{eq:pf-F-E-2}, where we used in the second inequality the fact that
\begin{align*}
\PP_x(n t_2<\tau_\d)\geq c\gamma_2^{kt_2}\inf_{y\in L}\PP_y\left((n-k)t_2<\tau_\d\right),\quad\forall x\in L,
\end{align*}
which is deduced from~\eqref{eq:pf-E-F-3-ter} exactly as in Lemma~\ref{lem:comp-surv}.

\medskip\noindent\textit{Step 4. Proof of~(E3).}

Using~\eqref{eq:pf-F-E-2} and the fact that $\sup_{x\in K} \varphi_1(x)<+\infty $, we deduce that there exists a constant $C>0$ such
that, for all $n\in\N$,
\begin{align*}
\sup_{x\in K}\P_x(n t_2 < \tau_\d)\leq C \inf_{y\in L} \P_y(n t_2< \tau_\d).
\end{align*}
Moreover,  using the Markov property at time $n_0t_2$ and ~\eqref{eq:borne-psi-sur-K}, we have that, for all $t\geq 0$,
\begin{align*}
\inf_{x\in K}\P_x(t<\tau_\d)&\geq \inf_{x\in K}\P_x(t+n_0t_2<\tau_\d)\geq \gamma_2^{n_0 t_2} \inf_{y\in L}\P_y(t<\tau_\d).
\end{align*}
These inequalities imply~(E3).

\subsubsection{Existence of a quasi-stationary distribution for $(X_t)_{t\in I}$}
\label{sec:full-QSD}

Subsection~\ref{sec:proof-(E)} and Theorem~\ref{thm:main} imply that there exists a probability measure $\nu_{QSD}$ on $E$
such that
\begin{align*}
\P_{\nu_{QSD}}(X_{nt_2}\in\cdot\mid nt_2<\tau_\d)=\nu_{QSD},\ \forall n\in \Z_+,
\end{align*}
such that $\nu_{QSD}(\varphi_1)<\infty$ and $\nu_{QSD}(\varphi_2)>0$, which is equivalent to $\nu_{QSD}(L)>0$ because of the
quasi-stationarity and the form of $\varphi_2$. For all $t\in[0,t_2]$, let us define the probability measure $ \nu_t$ on $E$ by
\begin{align*}
\nu_t=\P_{\nu_{QSD}}(X_t\in\cdot\mid t<\tau_\d).
\end{align*}
For all $n\in\Z_+$, we have, using the Markov property and the fact that $\nu_{QSD}$ is a quasi-stationary distribution for $(X_{nt_2})_{n\geq 0}$, 
\begin{align*}
\P_{\nu_t}(X_{nt_2}\in\cdot\mid nt_2<\tau_\d)= \E_{\nu_{QSD}}(\P_{X_{nt_2}}(X_t\in \cdot\mid t<\tau_\d)\mid
nt_2<\tau_\d)=\P_{\nu_{QSD}}(X_t\in\cdot\mid t<\tau_\d),
\end{align*}
hence $\nu_t$ is a quasi-stationary distribution for $(P_n)_{n\geq 0}$. Moreover, the third line of~(F2) and the quasi-stationarity
of $\nu_t$ imply that $\nu_t(L)$ is positive.

Fix $\rho_1\in(\theta_1^{1/t_2},\gamma_2)$. It follows from~\eqref{eq:pf-E-F-3} that there exists a constant $C>0$ such that, for all $x\in E$,
\begin{align*}
\varphi'_1(x):= \E_x\left(\rho_1^{- \tau_L\wedge\tau_\d}\right) & \leq C\,\varphi_1(x).
\end{align*}
We also have that, for all $x\in E\setminus L$,
\begin{align}
\E_x\left(\11_{t_2< \tau_L\wedge\tau_\d}\varphi'_1(X_{t_2})\right) & =\rho_1^{t_2}\E_x\left(\11_{t_2< \tau_L\wedge\tau_\d}
\rho_1^{-\tau_L\wedge \tau_\d}\right) \notag \\ & \leq \rho_1^{t_2}\varphi'_1(x) \label{eq:fction-Lyap-tps-cont}
\end{align}
and the inequality is trivial for $x\in L$. In addition, for all $t\in[0,t_2]$ and all $x\in L$,
$\E_x\left(\varphi'_1(X_t)\11_{t<\tau_\d}\right)\leq C \E_x\left(\psi_1(X_t)\11_{t<\tau_\d}\right)\leq C c_2$. Hence Condition~(F) is
satisfied replacing $\gamma_1$ with $\rho_1$ and $\psi_1$ with $\varphi'_1$. Therefore, we can apply Step~1 to prove that (E) is
satisfied with $\varphi'_1$ and $\varphi'_2$ where
\[
\varphi'_2=\frac{\gamma_2^{-t_2}-1}{\gamma_2^{-n'_0t_2}-1}\sum_{k=0}^{n'_0-1}\gamma_2^{-kt_2}P_k\11_L
\]
for an integer $n'_0$ that can be chosen larger than $n_0$. We also deduce as in the beginning of Step~2 that
$\nu_{QSD}$ is the unique quasi-stationary distribution of $(P_n)_{n\geq 0}$ such that $\nu_{QSD}(\varphi'_1)<\infty$ and
$\nu_{QSD}(L)>0$.

Moreover, by Markov's property at time $t$ we have for all $x\in E$ and $t\geq 0$,
\begin{align}
\varphi'_1(x) & = \E_x\left[\11_{t<\tau_L\wedge\tau_\d}{\rho_1}^{-\tau_L\wedge \tau_\d}\right]
+\E_x\left[\11_{t\geq \tau_L\wedge\tau_\d}{\rho_1}^{-\tau_L\wedge \tau_\d}\right] \notag \\
& \leq {\rho_1}^{-t}\E_x\left[\11_{t<\tau_L\wedge\tau_\d}\varphi'_1(X_t)\right]
+{\rho_1}^{-t}\P_x(t\geq \tau_L\wedge\tau_\d) \notag \\
& \leq {\rho_1}^{-t}\left(\E_x[\11_{t<\tau_\d}\varphi'_1(X_t)]+1\right) \label{eq:pf-E-F-4}
\end{align}
so that, for all $t\in[0,t_2]$,
\begin{align*}
\nu_t(\varphi_1')&\leq {\rho_1}^{-(t_2-t)}\left[\E_{\nu_{QSD}}\left(\11_{t_2<\tau_\d}\varphi'_1(X_{t_2})\right)/\P_{\nu_{QSD}}(t<\tau_\d)+1\right]\\
&\leq {\rho_1}^{-(t_2-t)}\left[\E_{\nu_{QSD}}\left(\11_{t_2<\tau_\d}\varphi'_1(X_{t_2})\right)/\P_{\nu_{QSD}}(t_2<\tau_\d)+1\right]\\
&= {\rho_1}^{-(t_2-t)}\left( \nu_{QSD}(\varphi'_1)+1\right)<\infty.
\end{align*}
Since we observed that $\nu_t(L)>0$, we deduce that $\nu_t=\nu_{QSD}$ for all $t\in I\cap[0,t_2]$.

Using the Markov property, we deduce that $\nu_t=\nu_{QSD}$ for all $t\in I$ and hence that $\nu_{QSD}$ is a
quasi-stationary distribution for $(X_t)_{t\in I}$. Since any quasi-statio\-na\-ry distribution for $(X_t)_{t\in I}$ is
also a quasi-stationary distribution for $(P_n)_{n\geq 0}$, we deduce that $\nu_{QSD}$ is the unique quasi-stationary
distribution for $(X_t)_{t\in I}$ such that $\nu_{QSD}(\varphi_1)<+\infty$ and $\nu_{QSD}(L)>0$. % By the
% quasi-stationarity property of $\nu_{QSD}$, it is also the unique one satisfying $\nu_{QSD}(\varphi_1)<+\infty$ and
% $\P_{\nu_{QSD}}(X_{t}\in L)>0$ for some $t\in I$.

\medskip

Let $t\geq t_2$ be fixed and define $k\in\NN$ such that $0\leq t-kt_2<t_2$. It follows from the fact that $P_1\varphi'_1\leq
\bar{C}\varphi'_1$ and from~\eqref{eq:pf-E-F-4} that
\begin{align}
\E_x [\11_{t<\tau_\d}\varphi'_1(X_t)]&\leq \bar{C}^k\E_x\left[\11_{t-kt_2<\tau_\d}\varphi'_1(X_{t-kt_2})\right] \notag \\ 
& \leq \bar{C}^k{\rho_1}^{-(k+1)t_2+t}\E_x\left[\11_{t_2<\tau_\d}\varphi'_1(X_{t_2})+\11_{t-kt_2<\tau_\d}\right] \notag \\
& \leq C\bar{C}^k{\rho_1}^{-(k+1)t_2+t} \E_x\left[\11_{t_2<\tau_\d}\varphi_1(X_{t_2})+1\right] \notag \\
& \leq C\bar{C}^k{\rho_1}^{-(k+1)t_2+t}(\theta_1+c_2+1)\varphi_1(x).
\label{eq:last-eq-of-pf-of-E-F}
\end{align}
Note that a similar inequality may not hold true with $\varphi'_1$ replaced by $\varphi_1$ under our assumptions.
  This explains why we need to introduce $\varphi'_1$.

Now, let $\mu$ be a probability measure such that $\mu(\varphi_1)<\infty$ and $\mu(\varphi_2)>0$. Then, for all $t\geq n_0 t_2$, it
follows from~\eqref{eq:borne-psi-sur-K} that, for all $k\geq 0$,
\[
\P_\mu(X_{t+kt_2}\in L) \geq \P_\mu(X_{kt_2}\in L)\inf_{y\in L}\P_y(X_t\in L)\geq \gamma_2^{t}\P_\mu(X_{kt_2}\in L).
\]
Therefore, for all $t\in[n_0t_2,(n_0+1)t_2]$,
\begin{align*}
\E_\mu(\varphi_2(X_t)) & =\frac{\gamma_2^{-t_2}-1}{\gamma_2^{-n_0 t_2}-1}\sum_{k=0}^{n_0-1}\gamma_2^{kt_2}\P_\mu(X_{t+kt_2}\in L)
\\ & \geq \frac{\gamma_2^{-t_2}-1}{\gamma_2^{-n_0 t_2}-1}\gamma_2^{(n_0+1)t_2}\sum_{k=0}^{n_0-1}\gamma_2^{kt_2}\P_\mu(X_{kt_2}\in
L)=\gamma_2^{(n_0+1)t_2} \mu(\varphi_2).
\end{align*}
This and inequality~\eqref{eq:last-eq-of-pf-of-E-F} imply that (using that $n'_0\geq n_0$), for all $t\in[n_0t_2,(n_0+1)t_2]$ and for a
constant $C>0$ that may change from line to line,
\begin{align*}
\frac{\mu_t(\varphi'_1)}{\mu_t(\varphi'_2)}\leq C\frac{\mu_t(\varphi'_1)}{\mu_t(\varphi_2)}\leq C\frac{\mu(\varphi_1)}{\mu(\varphi_2)},
\end{align*}
where $\mu_t:=\P_\mu(X_t\in\cdot\mid t<\tau_\d)$. It then follows the fact that (E) is satisfied by $(P_n,n\geq 0)$ with the
functions $\varphi'_1$ and $\varphi'_2$ that there exist constants $\alpha<1$ and $C>0$ such that, for all
$t\in[n_0t_2,(n_0+1)t_2]$,
\begin{align*}
\left\|\frac{\mu_t P_{n}}{\mu_t P_{n}\11_E} -\nu_{QSD}\right\|_{TV}
&\leq C\alpha^{n}\,\frac{\mu( \varphi_1)}{\mu( \varphi_2)},
\end{align*}
Using Markov property, we deduce that
\begin{align*}
\left\|\P_\mu(X_{nt_2+t}\in\cdot\mid nt_2+t<\tau_\d) -\nu_{QSD}\right\|_{TV}
&\leq C\alpha^{n}\,\frac{\mu( \varphi_1)}{\mu( \varphi_2)}.
\end{align*}
This ends the proof of~\eqref{eq:expo-cv-prop-E-F}.

\subsubsection{Convergence to $\eta$}
\label{sec:full-eta}

To finish the proof of Theorem~\ref{thm:E-F}, it remains to prove that the convergence~\eqref{eq:eta-prop-E-F} is exponential in
$L^\infty(\psi^{1/p}_1)$ and that $P_t\eta=e^{-\lambda_0 t}\eta$. Because of
Remark~\ref{rem:dom-attraction}, it is enough to prove this for $p=1$. Since we proved that (E) holds true for the semigroup
$(P_n)_{n\geq 0}$ and for the functions $\varphi'_1$ and $\varphi'_2$, it follows from Theorem~\ref{thm:eta} that there exist
constants $\lambda_0\in[0,\log(1/\gamma_2)]$, $\alpha\in(0,1)$ and $C>0$ such that, for all $y\in E$,
\begin{align*}
\left|e^{\lambda_0 n t_2}\PP_y(n t_2<\tau_\d)-\eta(y)\right|\leq C \alpha^n\varphi'_1(y).
\end{align*}
For any $t\in[t_2,2t_2]$, integrating this inequality with respect to $\PP_x(X_t\in dy; t<\tau_\d)$, we deduce
from~\eqref{eq:last-eq-of-pf-of-E-F} that
\begin{align*}
\left|e^{\lambda_0 n t_2}\PP_{x}(n t_2+t<\tau_\d)-\E_x(\eta(X_t)\11_{t<\tau_\d})\right|\leq C \alpha^n\varphi_1(x)
\end{align*}
for a constant $C$ independent of $t\in[t_2,2t_2]$. Setting $\eta_t(x)=\EE_x\left[ e^{\lambda_0 t}\eta(X_t)\11_{t<\tau_\d}\right]$,
we obtain for all $t\in[t_2,2t_2]$
\begin{align*}
\left|e^{\lambda_0 (n t_2+t)}\PP_{x}(n t_2+t<\tau_\d)-\eta_t(x)\right|\leq C e^{2\lambda_0 t_2} \alpha^n\varphi_1(x).
\end{align*}
Proceeding as in~\eqref{eq:fin-relecture}, we deduce, letting $n\rightarrow+\infty$, that $P_1 \eta_t=e^{-\lambda_0 t_2}\eta_t$. It
then follows from Corollary~\ref{cor:quasi-comp} that $\eta_t(x)=\eta(x)\nu_{QSD}(\eta_t)$ for all $x\in E$. Since we proved above
that $\nu_{QSD}$ is a quasi-stationary distribution with decay parameter $\lambda_0$, by definition of $\eta_t$,
$\nu_{QSD}(\eta_t)=1$ and thus $P_t\eta=e^{-\lambda_0 t}\eta$. This ends the proof of Theorem~\ref{thm:E-F}.

\subsection{Proof of Lemma~\ref{prop:E-G}}
\label{sec:proof-E-G}

Proceeding as in~\eqref{eq:fction-Lyap-tps-cont} and~\eqref{eq:pf-E-F-4}, we have that, for all $x\in E$ and $t\in I$,
\begin{align*}
\E_x\left(\psi_1(X_{t_2})\11_{t_2<\tau_L\wedge\tau_\d}\right)
\leq \gamma_1^{t_2} \psi_1(x)\quad\text{and}\quad
\psi_1(x) \leq
\gamma_1^{-t}\left(\E_x\left[\11_{t<\tau_\d}\psi_1(X_t)\right]+1\right).
\end{align*}
Therefore, for all $t\leq t_2$ and all $x\in L$,
\begin{align*}
\E_x\left[\11_{t<\tau_\d}\psi_1(X_t)\right] & \leq
\gamma_1^{-(t_2-t)}\E_x\left\{\left[\E_{X_{t}}\left(\11_{t_2-t<\tau_\d}\psi_1(X_{t_2-t})\right)+1\right]\11_{t<\tau_\d}\right\}\\
& \leq\gamma_1^{-(t_2-t)}\left[\E_x\left(\11_{t_2<\tau_\d}\psi_1(X_{t_2})\right)+1\right]\\
& \leq c_2:=\gamma_1^{-t_2}\left[\sup_{y\in L}\E_y\left(\11_{t_2<\tau_\d}\psi_1(X_{t_2})\right)+1\right].
\end{align*}
This concludes the proof of Lemma~\ref{prop:E-G}.

\subsection{Proof of Proposition~\ref{prop:A1-A2}}
\label{sec:proof-A1-A2}

Let us first assume that (E) is satisfied with $\varphi_1$ bounded and~\eqref{eq:unif-2} and prove that~\eqref{eq:unif-1} holds true.
Theorem~\ref{thm:main} and Remark~\ref{rem:dom-attraction} entail that, for all $n\geq n'_4$,
\begin{align*}
\left\|\frac{\mu P_{n}}{\mu P_{n}\11_E} -\nu_{QSD}\right\|_{TV}
&\leq \alpha^{n-n'_4} \frac{\|\varphi_1\|_\infty}{\inf_{x\in K}\varphi_2(x)}\,\frac{\mu P_{n'_4}\11_E}{\mu
    P_{n'_4}\11_K} \\ & \leq \alpha^{n-n'_4} \frac{\|\varphi_1\|_\infty}{\underline{c}\inf_{x\in K}\varphi_2(x)}.
\end{align*}
Hence the convergence is uniform.

Let us now assume that~\eqref{eq:unif-1} holds true. It was proved in \cite{ChampagnatVillemonais2016b} that this is equivalent to the
following condition.

\medskip\noindent\textbf{Condition (A).} There exist positive constants $c_1,c_2$, a positive integer $k_0$ and a probability
measure $\nu$ on $E$ such that
\begin{itemize}
    \item[(A1)] \textit{(Conditional Dobrushin coefficient)} For all $x\in E$, 
    \begin{align*}
    \P_x(X_{k_0}\in\cdot\mid k_0<\tau_\d)\geq c_1 \nu.
    \end{align*}
    \item[(A2)] \textit{(Global Harnack inequality)} We have
    \begin{align*}
    \sup_{k\in\ZZ_+}\frac{\sup_{y\in E} \P_y(k<\tau_\d)}{\P_\nu(k<\tau_\d)}\leq c_2.
    \end{align*}
\end{itemize}

Several consequences of Condition~(A) were deduced in \cite{ChampagnatVillemonais2016b}, among which the fact that the
convergence~\eqref{eq:conv-eta-general} in Theorem~\ref{thm:eta} holds true with respect to the $L^\infty$ norm on $E$ with
$\eta(x)>0$ for all $x\in E$. In particular, $\eta$ is bounded, $P_1\eta=\theta_0\eta$ and there exists a constant $C'$ such that,
for all $n\geq 0$,

\begin{equation}
\label{eq:proof-unif-1}
\sup_{x\in E}\P_x(n<\tau_\d)\leq C'\theta_0^n.    
\end{equation}
We fix $\varepsilon\in(0,1/(4C'))$. Since $\eta$ is positive on $E$, there exists $\delta>0$ such that the set $K:=\{x\in
E:\eta(x)\geq\delta\}$ satisfies $\nu_{QSD}(K)\geq 1-\varepsilon$ and $\nu(K)>0$. Setting $\varphi_2=\eta/\|\eta\|_\infty$, the part
of (E2) dealing about $\varphi_2$ is satisfied with $\theta_2=\theta_0$. Since the convergence in Theorem~\ref{thm:eta} holds true
with respect to the $L^\infty$ norm, we deduce from the choice of $K$ that there exists $k\geq k_0$ such that
\[
c:=\inf_{x\in K}\PP_x(k_0<\tau_\d)\geq \inf_{x\in K}\PP_x(k<\tau_\d)>0.
\]
It follows from~(A1) and~(A2) that, for all $n\geq 0$,
\begin{align*}
\inf_{x\in K}\P_x(n<\tau_\d)\geq\inf_{x\in K}\PP_x(n+k_0<\tau_\d)  \geq c_1c\,\PP_\nu(n<\tau_\d)  & \geq \frac{c_1c}{c_2}\sup_{y\in
    E}\PP_y(n<\tau_\d).
\end{align*}
This implies~(E3) and that $\inf_{x\in K}\PP_x(k_0<\tau_\d)>0$. Hence,~(E1) follows from~(A1) with the probability measure
$\frac{\nu(\cdot\cap K)}{\nu(K)}$. Moreover, for any $n$ large enough to have $C\alpha^n\leq 1/2$ where the constants $C$ and
$\alpha$ are those of~\eqref{eq:unif-1}, we have $\P_x(X_n\in K\mid t<\tau_\d)\geq\nu_{QSD}(K)-C\alpha^n\geq 1/2-\varepsilon>0$ and
hence~(E4) is satisfied. The last computation also entails~\eqref{eq:unif-2} with $n'_4=n$.

It remains to construct a function $\varphi_1$ satisfying (E2) with $\theta_1<\theta_0$. For all $x\in E$,
\begin{align*}
\P_x(X_{n}\in E\setminus K\mid n<\tau_\d) & \leq \nu_{QSD}(E\setminus K)+C\alpha^n 
\leq \varepsilon+C\alpha^n.
\end{align*}
Using~\eqref{eq:proof-unif-1}, we deduce that
\begin{align*}
\P_x(X_{n}\in E\setminus K) & \leq C'(\varepsilon+C\alpha^n)\theta_0^n,
\end{align*}
so that there exists $n_0$ large enough such that
\begin{align*}
\P_x(n_0< T_K\wedge\tau_\d) & \leq \frac{1}{3}\theta_0^{n_0}=\left(\frac{\theta_0}{3^{1/n_0}}\right)^{n_0}.
\end{align*}
From this follows that, for all $k\in\NN$ and all $x\in E$,
\begin{align*}
\P_x(kn_0< T_K\wedge\tau_\d) & \leq \left(\frac{\theta_0}{3^{1/n_0}}\right)^{kn_0}.
\end{align*}
In particular, for $\theta_1:=\theta_0/2^{1/n_0}$,
\begin{align*}
\varphi_1(x):=\EE_x\left(\theta_1^{- T_K\wedge\lceil\tau_\d\rceil}\right),\quad\forall x\in E,
\end{align*}
is a bounded function on $E$ and Lemma~\ref{lem:varphi1} implies that, for all $x\in E$,
\begin{align*}
P_1\varphi_1(x)\leq \theta_1\varphi_1(x)+\|\varphi_1\|_\infty\11_{K}(x).
\end{align*}
Since $\theta_1<\theta_0$, (E2) is proved.

\section{Proof of the results of Section~\ref{sec:general-diffusion}}
\label{sec:pf-general-diffusion}

In order to prove Theorem~\ref{thm:main-diffusion}, we check Condition~(F). The goal of
Subsection~\ref{sec:construction-diff-multidim} is to give the construction of the process $X$ and to check (F0) with $L=K_k$ for any
$k\geq 1$. In Subsection~\ref{sec:harnack-diff}, we explain how (F1) and (F3) can be deduced from general Harnack inequalities.
Finally, Subsection~\ref{sec:proof-diff-multidim} completes the proof of Theorem~\ref{thm:main-diffusion}. The proof of
Corollary~\ref{cor:main-diffusion} is then given in Subsection~\ref{sec:pf-cor-diff}.

\subsection{Construction of the diffusion process $X$ and Markov property}
\label{sec:construction-diff-multidim}

The goal of this section is to construct a weak solution $X$ to the SDE~\eqref{eq:SDE} with absorption out of $D$, and prove that it is
Markov and satisfies a strong Markov property at appropriate stopping times, enough to entail Condition~(F0) for $L=K_k$ for any
$k\geq 1$. We introduce the natural path space for the process $X$ as
\begin{multline*}
\mathcal{D}:=\Biggl\{w:\R_+\rightarrow D\cup\{\d\}:\ \forall k\geq 1,\ w\text{ is continuous on }[0,\tau_k(w)] \\ \left.\text{ and }w(t)=\d,\ \forall
t\geq\sup_{k\geq 1} \tau_k(w)\right\},
\end{multline*}
where $\tau_k(w):=\inf\{t\geq 0: w_t\in D\setminus K_k\}$. Note that $\mathcal{D}$ contains functions which are not c\`adl\`ag since
they may not have a left limit at $\tau_\d-$ and, indeed, it is easy to construct examples where $X$ is not c\`adl\`ag
$\P$-a.s.\footnote{For example, one may consider $D$ the open disc of radius 1 centered at 0 in $\R^2$, $\sigma=\text{Id}$ and
  $b(x)=(-x_2\beta(|x|),x_1\beta(|x|))$ where $x=(x_1,x_2)\in D$. Decomposing the process in polar coordinates
  $(R_t,\theta_t):=(|X_t|,\arctan(X^{(1)}_t/X^{(2)}_t))$, the radius $R_t$ is a 2-dimensional Bessel process, and $X_t$ is sent to
  $\d$ when $R_t$ hits 1 (in a.s.\ finite time). The angle $\theta_t$ is solution to $\mathrm d\theta_t=R_t^{-1}\mathrm
  dW_t-\beta(R_t) \mathrm dt$ before $\tau_\d$, for some Brownian motion $W$. Hence, if $\beta(r)$ converges sufficiently fast to
  $+\infty$ when $r\rightarrow 1$, $\theta_t$ a.s.\ converges to $-\infty$ when $t\rightarrow\tau_\d-$, so $X$ does not admit a left
  limit at time $\tau_\d$.} Note also that this definition means that we are looking for a process $X$ such that
\begin{align*}
\tau_\d:=\sup_{k \geq 1}\tau_{D\setminus K_k},
\end{align*}
which is the natural definition of $\tau_\d$ when the left limit of $X$ at time $\tau_\d$ does not exist.

We endow the path space $\mathcal{D}$ with its natural filtration 
$$
\mathcal{F}_t=\sigma(w_s,s\leq t)=\bigvee_{n\geq 1, 0\leq
    t_1<t_2<\ldots<t_n\leq t}\sigma(w_{t_1},w_{t_2},\ldots, w_{t_n})
$$
and we follow the usual method which consists in constructing for all $x\in D$ a probability measure $\P_x$ on $\mathcal{D}$ and a
stochastic process $(B_t,t\geq 0)$ on $\mathcal{D}\times \mathcal{C}(\R_+,\R^r)$, such that $B$ is a standard $r$-dimensional
Brownian motion under $\PP_x\otimes\mathbb{W}^r$, where $\mathbb{W}^r$ is the $r$-dimensional Wiener measure and such that $w_0=x\ $
$\PP_x\otimes\mathbb{W}^r$-almost surely and the canonical process $(w_t,t\geq 0)$ solves the SDE~\eqref{eq:SDE} for this Brownian
motion $B$ on the time interval $[0,\sup_k \tau_k(w))$\ \footnote{Since $\sigma(x)$ is non-degenerate for all $x\in D$, the space
    $\mathcal{C}(\R_+,\R^r)$ equipped with the Wiener measure $\mathbb{W}^r$ is only used to construct the Brownian path $B_t$ after
    time $\sup_k \tau_k(w)$ and could be omitted for our purpose since we only need to construct the process $B$ up to time $\sup_k
    \tau_k(w)$.}.

For this construction, we use the fact that $b$ and $\sigma$ can be extended out of $K_k$ to $\R^d$ as globally H\"older and bounded
functions $b_k$ and $\sigma_k$ and such that $\sigma_k$ is uniformly elliptic on $\R^d$. Hence (see
e.g.~\cite[Rk.\,5.4.30]{KaratzasShreve1991}) the martingale problem is well-posed for the SDE
\begin{align*}
\mathrm d X^k_t=b_k(X^k_t) \mathrm dt+\sigma_k(X^k_t) \mathrm d B_t.
\end{align*}
Let us denote by $\P^k_x$ the solution to this martingale problem for the initial condition $x\in\R^d$. This is a probability measure
on $\mathcal{C}:=\mathcal{C}(\R_+,\R^d)$, equipped with its canonical filtration $(\mathcal{G}_t)_{t\geq 0}$. 

For all $k\geq 1$, we define $\tau'_k(w)=\inf\{t\geq 0, w_t\not\in\text{int}(K_k)\}$, where $\text{int}(K_k)$ is the interior of
$K_k$. Since the paths $w\in\mathcal{D}$ or $\mathcal{C}$ are continuous at time $\tau'_{k}$ and $\R^d\setminus\text{int}(K_k)$ is
closed, it is standard to prove that $\tau'_{k}$ is a stopping time for the canonical filtration $(\mathcal{F}_t)_{t\geq 0}$ on
$\mathcal{D}$ and for the canonical filtration $(\mathcal{G}_t)_{t\geq 0}$ on $\mathcal{C}$. We define as usual the stopped
$\sigma$-fields $\mathcal{F}_{\tau'_k}$ and $\mathcal{G}_{\tau'_k}$, and we define for all $x\in \text{int}(K_k)$ the restriction of
$\P_x$ to $\mathcal{F}_{\tau'_k}$ as the restriction of $\P^k_x$ to $\mathcal{G}_{\tau'_k}$, where we can identify the events of the
two filtrations since they both concern continuous parts of the paths. This construction is consistent for $k$ and $k+1$ (meaning
that if $x\in K_k$, they give the same probability to events of $\mathcal{F}_{\tau_k}$) by uniqueness of the solutions $\P^k_x$ and
$\P^{k+1}_x$ to the above martingale problems. Hence there exists a unique extension $\P_x$ of the above measures to $\bigvee_{k\geq
    1}\mathcal{F}_{\tau'_k}$. Note that, because of the specific structure of the path space $\mathcal{D}$, we have
\begin{equation}
\label{eq:simple-structure-D}
\bigvee_{k\geq 1}\mathcal{F}_{\tau'_k}=\mathcal{F}_\infty.
\end{equation}
To check this, it suffices to observe that, for all $t\geq 0$ and all measurable $A\subset D\cup\{\d\}$,
\begin{align}
\{w_t\in A\} & =\{t<\tau_\d,\ w_t\in A\cap D\}\cup\{\tau_\d\leq t,\ \d\in A\} \notag \\ & =\left(\bigcup_{k\geq 1}\{t<\tau'_k,\ w_t\in A\cap
D\}\right)\cup\left(\bigcap_{k\geq 1}\{\tau'_k\leq t,\d\in A\}\right), \label{eq:pf-Markov-pty}
\end{align}
hence $\{w_t\in A\}\in \bigvee_{k\geq 1}\mathcal{F}_{\tau'_k}$, and, proceeding similarly, the same property holds for events of the
form $\{w_{t_1}\in A_1,\ldots,w_{t_n}\in A_n\}$.

We recall (see~\cite[Section\,5.4]{KaratzasShreve1991}) that $(\P^k_x)_{x\in\R^d}$ forms a strong Markov family on the canonical
space $\mathcal{C}$. Our goal is now to prove that the family of probability measures $(\P_x)_{x\in D\cup\{\d\}}$, where $\P_\d$ is
defined as the Dirac measure on the constant path equal to $\d$, forms a Markov kernel of probability measures, for which the strong
Markov property applies at well-chosen stopping times.

We first need to prove that $(\P_x)_{x\in D}$ defines a kernel of probability measures, i.e.\ that $x\mapsto \P_x(\Gamma)$ is
measurable for all events $\Gamma$ of $\mathcal{F}_\infty$. We prove it for an event of the form $\{w_t\in A\}$, the extension to
events of the form $\{w_{t_1}\in A_1,\ldots, w_{t_n}\in A_n\}$, and hence to all events of $\mathcal{F}_\infty$, being easy. This
follows from~\eqref{eq:pf-Markov-pty}:
\begin{align*}
\P_x(w_t\in A) & =\lim_{k\rightarrow+\infty}\P_x(t<\tau'_k,\ w_t\in A\cap D)+\11_{\d\in
    A}\lim_{k\rightarrow+\infty}\P_x(\tau'_k\leq t) \\ & =\lim_{k\rightarrow+\infty}\P^{k+1}_x(t<\tau'_k,\ w_t\in A\cap D)+\11_{\d\in
    A}\lim_{k\rightarrow+\infty}\P^{k+1}_x(\tau'_k\leq t).
\end{align*}
Since all the probabilities in the right-hand side are measurable functions of $x$, so is $x\mapsto\P_x(w_t\in A)$.

Now, let us prove that $(X_t,t\geq 0)$ is Markov. It is well-known that this is implied by the following property: for all $n\geq 1$
and $0\leq t_1\leq\ldots\leq t_{n+1}$ and $A_1,\ldots, A_{n+1}$ measurable subsets of $D\cup\{\d\}$,
\begin{align*}
\P_x(w_{t_1}\in A_1,\ldots,w_{t_{n+1}}\in A_{n+1})=\E_x\left[\11_{w_{t_1}\in A_1,\ldots,w_{t_{n}}\in
    A_n}\P_{w_{t_n}}(w_{t_{n+1}-t_n}\in A_{n+1})\right].
\end{align*}
We prove this property only for $n=1$. It is easy to extend the proof to all values of $n\geq 1$. We have
\begin{multline*}
\P_x(w_{t_1}\in A_1,w_{t_2}\in A_2) =\P_x(w_{t_1}\in A_1,w_{t_2}\in A_2,\tau_\d>t_2) \\ +\P_x(w_{t_1}\in A_1,t_1<\tau_\d\leq
t_2)\11_{\d\in A_2}+\P_x(\tau_\d\leq t_1)\11_{\d\in A_1\cap A_2}.
\end{multline*}
Now, using that $(\P^k_x)_{x\in\RR^d}$ is a Markov family for all $k\geq 1$,
\begin{multline*}
\P_x(w_{t_1}\in A_1,w_{t_2}\in A_2,\tau_\d>t_2)  \\
\begin{aligned}
& =\lim_{k\rightarrow\infty}\P_x(w_{t_1}\in A_1,w_{t_2}\in A_2,\tau_k>t_2) \\
& =\lim_{k\rightarrow\infty}\P^k_x(w_{t_1}\in A_1,w_{t_2}\in A_2,\tau_k>t_2) \\
& =\lim_{k\rightarrow\infty}\E^k_x\left[\11_{w_{t_1}\in A_1,t_1<\tau_k}\P^k_{w_{t_1}}(w_{t_2-t_1}\in
A_2,\tau_k>t_2-t_1)\right] \\
& =\lim_{k\rightarrow\infty}\E_x\left[\11_{w_{t_1}\in A_1,t_1<\tau_k}\P_{w_{t_1}}(w_{t_2-t_1}\in
A_2,\tau_k>t_2-t_1)\right] \\
& =\E_x\left[\11_{w_{t_1}\in A_1,t_1<\tau_\d}\P_{w_{t_1}}(w_{t_2-t_1}\in
A_2,\tau_\d>t_2-t_1)\right]
\end{aligned}
\end{multline*}
and similarly
\begin{align*}
\P_x(w_{t_1}\in A_1,t_1<\tau_\d\leq t_2) \11_{\d\in A_2} & =\E_x\left[\11_{w_{t_1}\in A_1,t_1<\tau_\d}\P_{w_{t_1}}(\tau_\d\leq
t_2-t_1)\right]\11_{\d\in A_2} \\ & =\E_x\left[\11_{w_{t_1}\in A_1,t_1<\tau_\d}\P_{w_{t_1}}(\tau_\d\leq
t_2-t_1,w_{t_2-t_1}\in A_2)\right].
\end{align*}
Since
\begin{align*}
\P_x(\tau_\d\leq t_1)\11_{\d\in A_1\cap A_2} & =\E_x\left[\11_{w_{t_1}\in A_1,\tau_\d\leq t_1}\P_{w_{t_1}}(w_{t_2-t_1}\in A_2)\right],
\end{align*}
we have proved that $\P_x(w_{t_1}\in A_1,w_{t_2}\in A_2)=\E_x\left[\11_{w_{t_1}\in A_1}\P_{w_{t_1}}(w_{t_{2}-t_1}\in A_{2})\right]$.
This ends the proof of the Markov property.

To conclude this subsection, let us prove that the strong Markov property holds for all stopping times $\tau_{F}$ where $F\subset D$ is
closed in $D$. Note that $\tau_F$ is indeed a stopping time for the filtration $\mathcal{F}_t$ since
$\tau_F=\sup_{k}\tau_F\wedge\tau'_k=\sup_k \tau_{(F\cup D^c)\cup \text{int}(K_k)^c}$, where the complement is understood in $\RR^d$, $(F\cup
D^c)\cup \text{int}(K_k)^c$ is a closed subset of $\RR^d$ and all $w\in\mathcal{D}$ is continuous at time $\tau_{(F\cup D^c)\cup
    \text{int}(K_k)^c}$. Let $x\in D$, $t_1,t_2,s\geq 0$ and $A,B\subset D$ be measurable sets. We proceed as above: first, observe
that
\begin{multline*}
\{w_{t_1}\in A,\ t_1<\tau_{F}\leq t_2,\ w_{\tau_{F}+s}\in B\} \\ 
\begin{aligned}
& =\bigcup_{\ell\geq 1}\{w_{t_1}\in A,\ t_1<\tau_{F}\leq t_2,\
w_{\tau_{F}+s}\in B,\ w_r\in K_\ell\ \forall r\in [0,\tau_{F}+s]\} \\ & =\bigcup_{\ell\geq 1}\{w_{t_1}\in A,\ t_1<\tau_{F}\wedge\tau'_\ell\leq t_2,\
w_{\tau_{F}\wedge\tau'_\ell+s}\in B,\ \tau'_\ell>\tau_F+s\}.
\end{aligned}
\end{multline*}
Since $\tau_F\wedge\tau'_\ell$ is a $\mathcal{G}_t$-stopping time on $\mathcal{C}(\RR_+,\RR^d)$ and using the strong Markov property
under $\PP^\ell$, we deduce that
\begin{multline*}
\P_x(w_{t_1}\in A,\ t_1<\tau_{F}\leq t_2,\ w_{\tau_{F}+s}\in B) \\
\begin{aligned}
& =\lim_{\ell\rightarrow+\infty}\P^\ell_x(w_{t_1}\in A,\ t_1<\tau_{F}\wedge\tau'_\ell\leq t_2,\ w_{\tau_{F}\wedge\tau'_\ell+s}\in
B,\ \tau'_\ell>\tau_F+s) \\ 
& =\lim_{\ell\rightarrow+\infty}\E^\ell_x\left[\11_{w_{t_1}\in A,\
    t_1<\tau_{F}\wedge\tau'_\ell\leq t_2}\P^\ell_{w_{\tau_{F}\wedge\tau'_\ell}}(w_{s}\in B,\ s<\tau'_\ell)\right] \\     &
=\lim_{\ell\rightarrow+\infty}\E^\ell_x\left[\11_{w_{t_1}\in A,\
    t_1<\tau_{F}\leq\tau'_\ell\wedge t_2}\P^\ell_{w_{\tau_{F}}}(w_{s}\in B,\ s<\tau'_\ell)\right] \\ & =\E_x\left[\11_{w_{t_1}\in A,\
    t_1<\tau_{F}\leq\tau_\d\wedge t_2}\P_{w_{\tau_{F}}}(w_{s}\in B,\ s<\tau_\d)\right].
\end{aligned}
\end{multline*}
Similarly,
\begin{align*}
\P_x(w_{t_1}\in A,\ t_1<\tau_{F}\leq t_2,\ &w_{\tau_{F}+s}=\d)\\
&=\lim_{\ell\rightarrow+\infty}\P^\ell_x(w_{t_1}\in A,\
t_1<\tau_{F}\leq t_2\wedge\tau'_\ell,\ \tau'_\ell\leq \tau_F+ s) \\ & =\E_x\left[\11_{w_{t_1}\in A,\
    t_1<\tau_{F}\leq t_2\wedge\tau_\d}\P_{w_{\tau_{F}}}(w_s=\d)\right]
\end{align*}
and thus 
$$
\P_x(w_{t_1}\in A,\ t_1<\tau_{F}\leq t_2,\ w_{\tau_{F}+s}\in B)=\E_x\left[\11_{w_{t_1}\in A,\
    t_1<\tau_{F}\leq t_2\wedge\tau_\d}\P_{w_{\tau_{F}}}(w_{s}\in B)\right]
$$
for all $A,B\subset D\cup\{\d\}$ measurable. The previous computation extends without difficulty to prove
\begin{multline}
\P_x\left(w_{t_1}\in A_1,\ldots, w_{t_n}\in A_n,\ t_n<\tau_F\leq t_{n+1},\ w_{\tau_{F}+s_1}\in B_1,\ldots,w_{\tau_F+s_m}\in
B_m\right) \\ =\E_x\left[\11_{w_{t_1}\in A_1,\ldots, w_{t_n}\in A_n,\ t_n<\tau_{F}\leq t_{n+1}}
\P_{w_{\tau_{F}}}(w_{s_1}\in B_1,\ldots,w_{s_m}\in B_m)\right] \label{eq:strong-Markov-diff}
\end{multline}
for all $n,m\geq 1$, $0\leq t_1\leq\ldots\leq t_{n+1}$, $0\leq s_1\leq\ldots\leq s_m$ and $A_1,\ldots,A_n,B_1,\ldots,B_m\subset
D\cup\{\d\}$ measurable. This implies the strong Markov property at time $\tau_{F}$, in the sense that, for all $k\geq 1$, all
$x\in E$ and all $\Gamma\in\mathcal{F}_\infty$,
\begin{align*}
\PP_x\left(w^{\tau_{F}}\in\Gamma\mid \mathcal{H}_{\tau_{F}}\right)=\PP_{w_{\tau_{F}}}(\Gamma),\quad\PP_x\text{-almost surely},
\end{align*}
where $w^{\tau_{F}}=(w_{\tau_{F}+s},s\geq 0)$ and
\begin{multline*}
\mathcal{H}_{\tau_F}=\sigma\Big(\left\{w_{t_1}\in A_1,\ldots,w_{t_n}\in A_n, t_n<\tau_F\leq t_{n+1}\right\},\, n\in\N, \\ 0\leq t_1\leq\ldots\leq
t_{n+1},\ A_1,\ldots,A_n\in D\text{ measurable}\Big).
\end{multline*}
This form of strong Markov property at time $\tau_{F}$ is enough for our purpose, since it entails~(F0) for $L=K_k$ for all $k\geq
1$. % It will be also needed in the next section.

\subsection{Harnack inequalities}
\label{sec:harnack-diff}

Our goal here is to check Conditions (F1) and (F3) for the diffusion process constructed above. We will make use of
general Harnack inequalities of Krylov and Safonov~\cite{KrylovSafonov1980}.

\begin{prop}
    \label{prop:harnack}
    There exist a probability measure $\nu$ on $D$ and a constant $t_\nu>0$ such that, for all $k\geq 1$, there exists a constant
    $b_k>0$ such that
    \begin{align}
    \label{eq:harnack1}
    \P_x(X_{t_\nu}\in\cdot)\geq b_k\nu(\cdot),\ \forall x\in K_k.
    \end{align}
    Moreover, for all $k\geq 1$ such that $K_k$ is non-empty,
    \begin{align}
    \label{eq:harnack2}
    \inf_{t\geq 0}\frac{\inf_{x\in K_k} \P_x(t<\tau_\d)}{\sup_{x\in K_k}\P_x(t<\tau_\d)}>0.
    \end{align}
\end{prop}

\begin{proof}
    Consider a bounded measurable function $f:D\rightarrow \R_+$ with $\|f\|_\infty\leq 1$ and define the application
    $u:(t,x)\in\R_+\times E\mapsto \E_x[\11_{t<\tau_\d}f(X_t)]$. It is proved in~\cite{ChampagnatVillemonais2017}
    using~\cite{KrylovSafonov1980} that, for all $k\geq 1$, there exist two constants $N_k>0$ and $\delta_k>0$, which do not depend on
    $f$ (provided $\|f\|_\infty\leq 1$), such that
    \begin{align}
    \label{eq:Harnack-u}
    u(\delta_k^2,x)\leq N_k u(2\delta_k^2,y),\text{ for all $x,y\in K_k$ such that }|x-y|\leq \delta_k/2.
    \end{align}
    Note that the proof given in~\cite{ChampagnatVillemonais2017} makes use of the following strong Markov property: for all open ball $B$
    such that $B\subset K_k$ for some $k\geq 1$, all $x\in B$, $t\geq 0$ and all measurable $f:D\cup\{\d\}\rightarrow\RR_+$,
    $$
    \EE_x\left[f(X_t)\11_{\tau_{D\setminus B}\leq t<\tau_\d}\right]=\EE_x\left[\11_{\tau_{D\setminus B}\leq
        t}\restriction{\EE_{X_{\tau_{D\setminus B}}}\left[f(X_{t-u})\11_{t-u<\tau_\d}\right]}{u=\tau_{D\setminus B}}\right].
    $$
    This property follows from~\eqref{eq:strong-Markov-diff}.
    
    \medskip\noindent\textit{Step 1 : Proof of~\eqref{eq:harnack1}} 
    
    Fix $x_1\in D$ and $k_1\geq 1$ such that $x_1\in \text{int}(K_{k_1})$. Let $\nu$ denote the conditional law
    $\P_{x_1}(X_{\delta_{k_1}^2}\in \cdot\mid \delta_{k_1}^2<\tau_\d)$. Then, for all measurable $A\subset
    D\cup\{\d\}$, Harnack's inequality~\eqref{eq:Harnack-u} with $f=\11_A$ entails that, for all $x\in D$ such that $|x-x_1|<
    \frac{\delta_{k_1}}{2}\wedge d(x_1,D\setminus K_{k_1})$,
    \begin{align*}
    \P_x(2\delta_{k_1}^2\in A)\geq \frac{\P_{x_1}(\delta_{k_1}^2<\tau_\d)}{N_{k_1}}\,\nu(A).
    \end{align*}
    Since the diffusion is locally elliptic and $D$ is connected, for all $k\geq 1$, there exists a constant $d_k>0$ such that
    \begin{align*}
    \inf_{x\in K_k} \P_x(X_1\in B(x_1,(\delta_{k_1}/2)\wedge d(x_1,D\setminus K_{k_1}))\geq d_k.
    \end{align*}
    This and Markov's property entail that, for all $x\in K_k$,
    \begin{align*}
    \P_x(X_{1+2\delta_{k_1}^2}\in \cdot)\geq d_k\frac{\P_{x_1}(\delta_{k_1}^2<\tau_\d)}{N_{k_1}}\,\nu.
    \end{align*}
    This implies the first part of Proposition~\ref{prop:harnack}.
    
    \medskip\noindent\textit{Step 2 : Proof of~\eqref{eq:harnack2}} 
    
    Fix $k\geq 1$ such that $K_k$ is non-empty and consider $\ell>k$ such that $K_k$ is included in one connected component of
    $\text{int}(K_\ell)$. For all $t\geq 2\delta_{\ell}^2$, the inequality~\eqref{eq:Harnack-u} applied to
    $f(x)=\P_x(t-2\delta_\ell^2<\tau_\d)$ and the Markov property entail that
    \begin{align*}
    \P_x(t-\delta_\ell^2<\tau_\d)\leq N_\ell \P_y(t<\tau_\d),\text{ for all $x,y\in K_\ell$ such that }|x-y|\leq \delta_\ell/2.
    \end{align*}
    Since $s\mapsto \P_x(s<\tau_\d)$ is non-increasing, we deduce that
    \begin{align*}
    \P_x(t<\tau_\d)\leq N_\ell \P_y(t<\tau_\d),\text{ for all $x,y\in K_\ell$ such that }|x-y|\leq \delta_\ell/2.
    \end{align*}
    Since $K_k$ has a finite diameter and is included in a connected component of $K_\ell$, we deduce that there exists $N'_k$ equal to
    some power of $N_\ell$ such that,
    for all $t\geq 2\delta_\ell^2$,
    \begin{align*}
    \P_x(t<\tau_\d)\leq N'_k \P_y(t<\tau_\d),\text{ for all $x,y\in K_k$.}
    \end{align*}
    Now, for $t\leq 2\delta_\ell^2$, we simply use the fact that, for all $x\in K_k$,
    $\P_x(2\delta_\ell^2<\tau_\d)\geq\P_x(2\delta_\ell^2<\tau_B)$ where $B=(x,1/2k)$ and hence $x\mapsto \P_x(2\delta_\ell^2<\tau_\d)$
    is uniformly bounded from below on $K_k$ by a constant $1/N''_k>0$. In particular,
    \begin{align*}
    \P_x(t<\tau_\d)\leq 1\leq N''_k \P_y(2\delta_\ell^2<\tau_\d)\leq N''_k \P_y(t<\tau_\d),\text{ for all $x,y\in K_k$.}
    \end{align*}
    This concludes the proof of Proposition~\ref{prop:harnack}.
\end{proof}

\subsection{Proof of Theorem~\ref{thm:main-diffusion}}
\label{sec:proof-diff-multidim}

Our aim is to prove that Condition~(F) holds true with $L=K_k$ for some $k\geq 1$ large enough. We have already proved (F0), (F1) and (F3) with
$L=K_k$ for any $k\geq 1$. Hence we only have to check (F2). Fix $\rho_1\in(\lambda_0,\lambda_1)$, $\rho_2\in (\lambda_0,\rho_1)$ and
$p\in(1,\lambda_1/\rho_1)$ and define
\begin{equation}
\label{eq:psi_1-diffusion}
\psi_1(x)=\varphi(x)^{1/p},\ \forall x\in D.  
\end{equation}
Fix $\rho'_1\in(\rho_1,\lambda_1/p)$ and 
\[
t_2\geq \frac{2s_1 (C+\lambda_1)}{\lambda_1-p\rho'_1}\,\vee\,\frac{\log 2}{\rho'_1-\rho_1},
\]
where the constant $C$ comes from~\eqref{eq:lyapunov-diffusions}. Set $L=K_{k_0}$ with $k_0$ large enough so that $\nu(K_{k_0})>0$
and, using~\eqref{eq:conv-to-0},
\[
\PP_x(s_1< \tau_{K_{k_0}}\wedge\tau_\d)\leq e^{-(\rho'_1+C/p)t_2}
\]
for all $x\in D_0$.

From the definition of $\lambda_0$ and applying the same argument as in Step 2 of the proof of Proposition~\ref{prop:harnack} with
$f(x)=\P_x(X_{t-2\delta_\ell^2}\in L)$ with $\ell$ large enough to have $K_{k_0}$ included in one connected component of
$K_\ell$, we deduce that
\begin{align*}
\liminf_{t\rightarrow+\infty} e^{\rho_2 t}\inf_{x\in L} \P_x(X_t\in L)=+\infty,
\end{align*}
and hence the last line of~(F2) is proved with $\gamma_2=e^{-\rho_2}$.

Let us now check that the first line of Assumption~(F2) holds true for all $x\in D_0$ and then for all $x\in D\setminus D_0$. For all
$x\in D_0$, we have $\psi_1(x)\leq \sup_{x\in D_0}\varphi^{1/p}(x)<+\infty$, and hence, for all $t\in [s_1,t_2]$, using H\"older's inequality
and the definition of $k_0$,
\begin{align}
\E_x\left(\psi_1(X_t)\11_{t<\tau_L\wedge \tau_\d}\right)&\leq
\E_x\left(\11_{t<\tau_\d}\varphi(X_t)\right)^{1/p}\P_x(t<\tau_L\wedge\tau_\d)^{\frac{p-1}{p}}\nonumber\\
&\leq \varphi(x)^{1/p}e^{C t_2/p}\P_x(s_1<\tau_L\wedge \tau_\d)^{\frac{p-1}{p}}\label{eq:diff-E-2-ct1-step1}\\
&\leq e^{-\rho'_1 t_2}\leq e^{-\rho_1 t_2}\psi_1(x)\nonumber.
\end{align}
To prove~\eqref{eq:diff-E-2-ct1-step1}, we used the fact that $\mathcal{L}\varphi\leq C\leq C\varphi$ and It\^o's formula
to obtain $P_t\varphi\leq e^{Ct}\varphi$. Since this argument is used repeatedly in the sequel, we give it in details for sake of
completeness. It follows from It\^o's formula that, for all $k\geq 1$, $\PP_x$-almost surely,
\begin{align*}
e^{-C\left(t\wedge \tau_{K_k^c}\right)}\varphi\left(X_{t\wedge \tau_{K_k^c}}\right)=\varphi(x) & +\int_0^t \11_{s\leq
    \tau_{K_k^c}}e^{-Cs}\left(\mathcal{L}\varphi(X_s)-C\varphi(X_s)\right) \mathrm ds \\ & +\int_0^t \11_{s\leq  \tau_{K_k^c}}e^{-Cs}\nabla\varphi(X_s)^*\sigma(X_s) \mathrm dB_s.
\end{align*}
Since $\nabla\varphi(x)$ and $\sigma(x)$ are uniformly bounded on $K_k$, the last term has zero expectation, and thus
\begin{align*}
\EE_x\left[e^{-C\left(t\wedge  \tau_{K_k^c}\right)}\varphi\left(X_{t\wedge \tau_{K_k^c}}\right)\right] & \leq \varphi(x).
\end{align*}
Letting $k\rightarrow+\infty$, we deduce form Fatou's lemma that
\begin{align}
\EE_x\left[e^{-Ct}\11_{t<\tau_\d}\varphi(X_{t})\right] & \leq \varphi(x) \label{eq:Ito-Lyap}
\end{align}
as claimed.

This proves the second line of (F2) for all $x\in D_0$ and $\gamma_1=e^{-\rho_1}$.

Now, for all $x\in D\setminus D_0$, since $D_0$ is closed in $D$, it follows from the strong Markov
property~\eqref{eq:strong-Markov-diff} at time $ \tau_{D_0}$ that
\begin{multline}
\label{eq:diff-E-2-ct1-step2}
\E_x\left(\psi_1(X_{t_2})\11_{t_2<\tau_L\wedge \tau_\d}\right)=\E_x\left(\11_{t_2-s_1<\tau_L\wedge\tau_\d\wedge
    \tau_{D_0}}\E_{X_{t_2-s_1}}\left(\psi_1(X_{s_1})\11_{s_1<\tau_L\wedge\tau_\d}\right)\right)\\
+\E_x\left(\11_{ \tau_{D_0}\leq t_2-s_1}\restriction{\E_{X_{ \tau_{D_0}}}\left(\psi_1(X_{t_2-u})\11_{t_2-u<\tau_\d\wedge \tau_L}\right)}{u= \tau_{D_0}}\right).
\end{multline}
Using H\"older's inequality and~\eqref{eq:Ito-Lyap}, we deduce that, for all $y\in D$,
\begin{align*}
\E_{y}\left(\psi_1(X_{s_1})\11_{s_1<\tau_L\wedge\tau_\d}\right)\leq
\E_{y}\left(\varphi(X_{s_1})\11_{s_1<\tau_\d}\right)^{1/p}\leq e^{\frac{s_1 C}{p}}\,\varphi(y)^{1/p}= e^{\frac{s_1
        C}{p}}\,\psi_1(y).
\end{align*}
Hence, the first term in the right-hand side of~\eqref{eq:diff-E-2-ct1-step2} satisfies 
\begin{align*}
\E_x\left(\11_{t_2-s_1<\tau_L\wedge\tau_\d\wedge  \tau_{D_0}}\E_{X_{t_2-s_1}}\left(\psi_1(X_{s_1})\11_{s_1<\tau_L\wedge\tau_\d}\right)\right)\leq
e^{\frac{s_1 C}{p}}\E_x\left(\11_{t_2-s_1<\tau_L\wedge\tau_\d\wedge  \tau_{D_0}}\psi_1(X_{t_2-s_1})\right).
\end{align*} 
As a consequence, using again H\"older's inequality and applying as above It\^o's formula using that
$\mathcal{L}\varphi(x)\leq -\lambda_1\varphi(x)$ for all $x\notin D_0$, one has
\begin{align*}
\E_x\left(\11_{t_2-s_1< \tau_L\wedge\tau_\d\wedge
    \tau_{D_0}}\E_{X_{t_2-s_1}}\left(\psi_1(X_{s_1})\11_{s_1<\tau_L\wedge\tau_\d}\right)\right)&\leq e^{-\lambda_1 \frac{t_2-s_1}{p}}
e^{\frac{s_1 C}{p}}\,\varphi(x)^{1/p}\\
&\leq e^{-t_2\frac{\rho'_1+\lambda_1/p}{2}}\psi_1(x),
\end{align*}
where we used in the last inequality that $t_2\geq \frac{2s_1 (C+\lambda_1)}{\lambda_1-p\rho'_1}$. Moreover,
using~\eqref{eq:diff-E-2-ct1-step1}, we obtain that the second term in the right-hand side of~\eqref{eq:diff-E-2-ct1-step2} satisfies
\begin{multline*}
\E_x\left(\11_{ \tau_{D_0}\leq t_2-s_1}\restriction{\E_{X_{ \tau_{D_0}}}\left(\psi_1(X_{t_2-u})\11_{t_2-u<\tau_\d\wedge \tau_L}\right)}{u= \tau_{D_0}}\right)\\
\leq e^{-\rho'_1 t_2}\P_x(\tau_{D_0}\leq t_2-s_1)\leq e^{-\rho'_1 t_2}\psi_1(x).
\end{multline*}
We finally deduce from~\eqref{eq:diff-E-2-ct1-step2} % and from the definition of $L=K_{k_0}$ 
that, for all $x\in D\setminus D_0$,
\begin{align*}
\E_x\left(\psi_1(X_{t_2})\11_{t_2<\tau_L\wedge\tau_\d}\right)\leq 2 e^{-\rho'_1 t_2}\psi_1(x)\leq e^{-\rho_1 t_2}\psi_1(x),
\end{align*}
where we used that $t_2\geq \log 2/(\rho'_1-\rho_1)$. This concludes the proof that the first line of~(F2) holds true with
$\gamma_1=e^{-\rho_1}$.

Since $\varphi$ is locally bounded, $\sup_L \varphi<\infty$, and hence, using again~\eqref{eq:Ito-Lyap}, we deduce that, for all
$t\geq 0$,
\begin{align*}
\sup_{x\in L}\ \E_x(\psi_1(X_t)\11_{t<\tau_\d})\leq \sup_{x\in L}\ \E_x(\varphi(X_t)\11_{t<\tau_\d})\leq e^{Ct}\,\sup_{x\in L}\varphi(x)<\infty,
\end{align*}
which implies the scond line of Assumption~(F2).

In addition, because of the local uniform ellipticity of the diffusion $X$, for all $ n_0\geq 1$, $\psi_2:=\sum_{k=0}^{n_0} P_k\11_L$ is
uniformly bounded away from zero on all compact subsets of $D$. This and Theorem~\ref{thm:E-F} concludes the proof of
Theorem~\ref{thm:main-diffusion}.

\subsection{Proof of Corollary~\ref{cor:main-diffusion}}
\label{sec:pf-cor-diff}

Using Theorem~\ref{thm:E-F}, there exists $\lambda'_0$ such that, for all $x\in D$,
\[
\eta(x)=\lim_{t\rightarrow+\infty}e^{\lambda'_0 t}\PP_x(t<\tau_\d).
\]
We choose in the definition of $\lambda_0$ a ball $B$ such that $\nu_{QSD}(B)>0$ (recall that $\lambda_0$ is independent of the choice
of $B$). Given $x\in D$ such that $\eta(x)>0$,
\[
\lim_{t\rightarrow+\infty}e^{\lambda'_0 t}\PP_x(X_t\in B)=\eta(x)\nu_{QSD}(B)\in(0,+\infty).
\]
Hence, $\lambda_0=\lambda'_0$ and the infimum in the definition of $\lambda_0$ is a minimum. The facts that
$\P_{\nu_{QSD}}(t<\tau_\d)=e^{-\lambda_0 t}$ and $P_t\eta=e^{-\lambda_0 t}\eta$ are then direct consequences of Theorem~\ref{thm:E-F}.

Let us now prove that $\eta$ is $\mathcal{C}^2$. First, it follows from~\cite[Theorem 7.2.4]{StroockVaradhan2006} that $x\mapsto
e^{\lambda_0 t}\PP_x(t<\tau_\d)$ is continuous for all $t\geq 0$ (see e.g.~\cite{ChampagnatVillemonais2017} for a detailed proof). Hence
the uniform convergence in Theorem~\ref{thm:eta} implies that $\eta$ is continuous on $D$.

Now, let $B$ be any non-empty open ball such that $\overline{B}\subset D$. We consider the following initial-boundary value problem
(in the terminology of~\cite{Friedman1964}) associated to the differential operator $\mathcal{L}$ defined in~\eqref{eq:def-gene-diff}
\[
\begin{cases}
\partial_t u(t,x)-\mathcal{L}u(t,x)-\lambda_0 u(t,x)=0 & \text{for all }(t,x)\in(0,T]\times B, \\
u(0,x)=\eta(x) & \text{for all }x\in B, \\
u(t,x)=\eta(x) & \text{for all }(t,x)\in(0,T]\times\partial B.
\end{cases}
\]
Since the coefficients of $\mathcal{L}$ are H\"older and uniformly elliptic in $\overline{B}$ and since $\eta$ is continuous, we can
apply Corollary~1 of Chapter 3 of~\cite{Friedman1964} to obtain the existence and uniqueness of a solution $u$ to the above problem,
continuous on $[0,T]\times\overline{B}$ and $\mathcal{C}^{1,2}((0,T]\times B)$. Now, we can apply It\^o's formula to $e^{\lambda_0
    s}u(T-s, X_s)$: for all $s<\tau_{B^c}\wedge T$ and all $x\in B$, $\PP_x$-almost surely,
\begin{align*}
e^{\lambda_0 s}u(T-s,X_s) & =u(T,x)+\int_0^s e^{\lambda_0 r}\left(-\frac{\partial u}{\partial t}+\mathcal{L}u+\lambda_0
u\right)(T-r,X_r)\,\mathrm dr \\ & +\int_0^s e^{\lambda_0 r}\nabla u(T-r,X_r) \sigma(X_r)\,\mathrm dB_r.
\end{align*}
Since $u$ is bounded and continuous on $[0,T]\times\overline{B}$ and $\nabla u(t,x)$ is locally bounded in $(0,T]\times B$, it
follows from standard localization arguments that
\begin{align*}
u(T,x) & =\EE_x\left[e^{\lambda_0 (T\wedge\tau_{B^c})}u(T-(T\wedge\tau_{B^c}),X_{T\wedge\tau_{B^c}})\right] \\
& =\EE_x\left[e^{\lambda_0 (T\wedge\tau_{B^c})}\eta(X_{T\wedge\tau_{B^c}})\right].
\end{align*}
Now, the Markov property and the fact that $P_t\eta=e^{-\lambda_0 t}\eta$ entail that $e^{\lambda_0 t}\eta(X_t)$ is a martingale on
$(\mathcal{D},(\mathcal{F}_t)_{t\geq 0},\PP_x)$, hence
\[
\eta(x) =\EE_x\left[e^{\lambda_0 (T\wedge\tau_{B^c})}\eta(X_{T\wedge\tau_{B^c}})\right]= u(T,x).
\]
Therefore, $\eta\in\mathcal{C}^2(D)$ and $\mathcal{L}\eta(x)=-\lambda_0\eta(x)$ for all $x\in D$.

\bibliographystyle{abbrv}
\bibliography{biblio-bio,biblio-denis,biblio-math}

\end{document}